\documentclass[12pt,a4paper]{article}

\usepackage[margin=1in]{geometry}

\usepackage[english,activeacute]{babel}
\usepackage{etex}
\usepackage{latexsym}
\usepackage{mathtools}
\usepackage{amsmath}
\usepackage{amssymb}

\usepackage{wasysym}
\usepackage{mathrsfs}
\usepackage{euscript}
\usepackage{amsthm}
\usepackage{amsmath}
\usepackage{amsfonts}
\usepackage{cancel}
\usepackage{graphicx,color}
\usepackage{multicol}
\usepackage{color}
\usepackage{multirow}
\usepackage{makeidx}
\usepackage{upgreek}
\usepackage{setspace} 
\usepackage{import}
\usepackage{float}
\usepackage{mleftright}
\usepackage{stmaryrd}
\usepackage[table,xcdraw]{xcolor}
\usepackage{comment}

\newcommand{\brbinom}[2]{\genfrac{[}{]}{0pt}{}{#1}{#2}}

\usepackage{mathalpha}

\DeclareFontFamily{U}{txcal}{\skewchar \font =45}
\DeclareFontShape{U}{txcal}{m}{n}{<-> txr-cal}{}
\DeclareMathAlphabet{\mathcalpxtx}{U}{txcal}{m}{n}

\DeclareMathAlphabet{\mathdutchcal}{U}{dutchcal}{m}{n}

\DeclareFontFamily{U}{stixtwobb}{\skewchar\font=45}%
\DeclareFontShape{U}{stixtwobb}{m}{n}{<->\mathalfa@bbscaled stix2-mathbb}{}
\DeclareMathAlphabet{\mathbbst}{U}{stixtwobb}{m}{n}

\usepackage{array}
\newcommand{\PreserveBackslash}[1]{\let\temp=\\#1\let\\=\temp}
\newcolumntype{C}[1]{>{\PreserveBackslash\centering}p{#1}}
\newcolumntype{R}[1]{>{\PreserveBackslash\raggedleft}p{#1}}
\newcolumntype{L}[1]{>{\PreserveBackslash\raggedright}p{#1}}

\usepackage{rotating}

\usepackage{bbold}

\usepackage{tikz}
\usepackage{tikz-cd}
\usetikzlibrary{matrix,arrows,calc,decorations.pathmorphing,fit,shapes.geometric,decorations.pathreplacing,positioning,snakes}

\usetikzlibrary{backgrounds}
\usepackage{contour}

\usepackage{xcolor}
\definecolor{babyblueeyes}{rgb}{0.54, 0.81, 0.94}
\definecolor{blizzardblue}{rgb}{0.67, 0.9, 0.93}
\definecolor{blue(munsell)}{rgb}{0.0, 0.5, 0.69}
\definecolor{bluegray}{rgb}{0.4, 0.6, 0.8}
\definecolor{bondiblue}{rgb}{0.0, 0.58, 0.71}
\definecolor{cadetblue}{rgb}{0.37, 0.62, 0.63}
\definecolor{carolinablue}{rgb}{0.6, 0.73, 0.89}
\definecolor{cinnamon}{rgb}{0.82, 0.41, 0.12}
\definecolor{darkcandyapplered}{rgb}{0.64, 0.0, 0.0}
\definecolor{darkcyan}{rgb}{0.0, 0.55, 0.55}
\definecolor{darkmidnightblue}{rgb}{0.0, 0.2, 0.4}
\definecolor{darkpastelblue}{rgb}{0.47, 0.62, 0.8}
\definecolor{frenchblue}{rgb}{0.0, 0.45, 0.73}

\definecolor{lamme}{RGB}{40,91,110}
\definecolor{evry}{RGB}{236, 128,35}

\definecolor{azul}{rgb}{0.54, 0.81, 0.94}

\usepackage{url}
\usepackage[bookmarksopen,colorlinks,allcolors=frenchblue,urlcolor=frenchblue]{hyperref}

\makeindex

\usepackage{scalerel,stackengine}

\usepackage{wrapfig}
\usepackage{ifpdf}

\usepackage{syntonly}

\DeclareMathOperator{\ob}{\mathsf{ob}}
\DeclareMathOperator{\id}{\mathsf{id}}
\DeclareMathOperator{\op}{\mathsf{op}}

\DeclareMathOperator{\hocolim}{\mathsf{hocolim}}

\DeclareMathOperator{\upc}{\mathsf{c}}
\DeclareMathOperator{\upd}{\mathsf{d}}

\DeclareMathOperator{\uph}{\mathsf{h}}

\DeclareMathOperator{\Crect}{\mathsf{C}}
\DeclareMathOperator{\Drect}{\mathsf{D}}

\DeclareMathOperator{\Hrect}{\mathsf{H}}

\DeclareMathOperator{\Krect}{\mathsf{K}}
\DeclareMathOperator{\Lrect}{\mathsf{L}}

\DeclareMathOperator{\Nrect}{\mathsf{N}}

\DeclareMathOperator{\Prect}{\mathsf{P}}

\DeclareMathOperator{\Rrect}{\mathsf{R}}

\DeclareMathOperator{\Trect}{\mathsf{T}}
\DeclareMathOperator{\Urect}{\mathsf{U}}
\DeclareMathOperator{\Vrect}{\mathsf{V}}
\DeclareMathOperator{\Wrect}{\mathsf{W}}

\DeclareMathOperator{\Ctt}{\mathtt{C}}

\DeclareMathOperator{\Wtt}{\mathtt{W}}

\DeclareMathOperator{\C}{\EuScript{C}}
\DeclareMathOperator{\Op}{\EuScript{O}}

\DeclareMathOperator{\D}{\EuScript{D}}

\DeclareMathOperator{\Disc}{\mathsf{Disc}}

\DeclareMathOperator{\Disj}{\mathsf{Disj}}

\DeclareMathOperator{\W}{\EuScript{W}}

\DeclareMathOperator{\Lfrak}{\mathdutchcal{L}}

\DeclareMathOperator{\CSSpc}{\EuScript{CSS}\mathsf{pc}}
\DeclareMathOperator{\Catinfty}{\EuScript{C}\mathsf{at}_{\infty}}

\DeclareMathOperator{\Qfrak}{\mathdutchcal{Q}}

\DeclareMathOperator{\Fun}{\mathsf{Fun}}

\DeclareMathOperator{\Map}{\mathsf{Map}}

\DeclareMathOperator{\Hom}{\mathsf{Hom}}

\DeclareMathOperator{\Conf}{\mathsf{Conf}}
\DeclareMathOperator{\Emb}{\mathsf{Emb}}

\DeclareMathOperator{\Fin}{\mathsf{Fin}}

\DeclareMathOperator{\Opdinfty}{\EuScript{O}\mathsf{pd}_{\infty}}
\DeclareMathOperator{\sm}{\mathsf{sm}}
\DeclareMathOperator{\Env}{\mathsf{Env}}

\DeclareMathOperator{\Sym}{\mathsf{Sym}}

\DeclareMathOperator{\Fscr}{\mathscr{F}}

\DeclareMathOperator{\boxdottedline}{\mathbin{\ThisStyle{\ensurestackMath{%
  \stackinset{c}{}{c}{-.2\LMpt}%
  {\SavedStyle\scaleobj{.5}{\bullet}}{\SavedStyle\boxminus}}}}}

\usepackage{titlesec}

\usepackage{float}

\usepackage{fancyhdr}
\usepackage{tikz}
\usetikzlibrary{calc,decorations.pathmorphing,shapes}

 \newtheorem{thm}{Theorem}[section]
 \newtheorem{cor}[thm]{Corollary}
 \newtheorem{lem}[thm]{Lemma}
 \newtheorem{prop}[thm]{Proposition}

  \newtheorem{conj}[thm]{Conjecture}
 \theoremstyle{definition}
 \newtheorem{defn}[thm]{Definition}
 \newtheorem{notat}[thm]{Notation}
 
 \theoremstyle{remark}
 \newtheorem{rem}[thm]{Remark}

 \newtheorem{ex}[thm]{Example}
 \newtheorem{exs}[thm]{Examples}

\title{\textsc{Algebras over not too little discs}
}
\author{Damien Calaque\thanks{IMAG, Univ Montpellier, CNRS, Montpellier, France \\ \href{mailto:damien.calaque@umontpellier.fr}{damien.calaque@umontpellier.fr}},\quad Victor Carmona\thanks{Max-Planck-Institut f\"ur Mathematik in den Naturwissenschaften. Leipzig, Germany \\ \href{mailto:victor.carmona@mis.mpg.de}{victor.carmona@mis.mpg.de}}}
\date{2024}

\begin{document}

\maketitle

 \begin{abstract}
By the introduction of locally constant prefactorization
algebras at a fixed scale, we show a mathematical incarnation of the fact that observables at a given scale of a topological field theory propagate to every scale over euclidean spaces. The key is
that these prefactorization algebras over $\mathbbst{R}^n$ are equivalent to algebras over the little $n$-disc operad. For topological field theories with defects, we
get analogous results by replacing $\mathbbst{R}^n$ with the spaces modelling corners $\mathbbst{R}^p\times\mathbbst{R}^q_{\geq 0}$. As a toy example in $1d $, we quantize, once more, constant Poisson structures.
 \end{abstract}

\setcounter{tocdepth}{2} 
\tableofcontents

\section{Introduction}
\subsection*{Algebraic structure of observables of a QFT}

According to Costello--Gwilliam \cite{costello_factorization_2017,costello_factorization_2021}, observables of a Quantum Field Theory (QFT) should provide the following assignment, known as a \textit{factorization algebra}: \begin{itemize}
\item To every open region $U$ of a space-time manifold $M$, it assigns an object $\mathcal O_U$ of observables located/supported in $U$; 
\item For a finite number of pairwise disjoint open regions $U_1,\dots,U_n$ sitting in a bigger open region $V$, there is a product-expansion map $\mathcal O_{U_1}\otimes\cdots\otimes\mathcal O_{U_n}\to \mathcal O_V$ (that one can view as an open analog of the point supported operator product expansion). 
\end{itemize}
The above shall satisfy natural axioms: 
\begin{enumerate}
\item Symmetry and associativity. 
\item Multiplicativity: if $U_1\sqcup U_2=V$ then $\mathcal O_{U_1}\otimes \mathcal O_{U_2}\to\mathcal O_V$ is an equivalence. 
\item Descent/glueing with respect to nice enough covers. 
\end{enumerate}

Several observations are in order. Condition 1 states that $\mathcal O$ is an algebra over a certain operad $\Disj(M)$ whose colors are open subsets of $M$ (and which comes with an operation $(U_1,\dots,U_n)\to V$ whenever $U_1\sqcup\dots\sqcup U_n\subset V$). Condition 2 ensures that one can restrict to connected open subsets. Condition 3 allows one to restrict to a nice enough basis of the topology; in the case when $M=\mathbbst{R}^n$ (that is the situation of interest in this paper) one can consider the basis consisting of open (euclidean) discs. We write $\Disc(\mathbbst{R}^n)$ for the full sub-operad of $\Disj(\mathbbst{R}^n)$ spanned by open (euclidean) discs. 

\medskip

One says that a QFT is \textit{topological} if it does not really depend on any geometric data (such as a metric, or the size of open regions) but only on the \textit{shape} of the region where we are ``making measurements''. For our observables, this means that the map $\mathcal O_U\to\mathcal O_V$ associated with an inclusion of two open discs $U\subset V$ is an equivalence. Algebras over the operad $\Disc(\mathbbst{R}^n)$ that satisfy the above property are called \textit{locally constant}, and are known to be the same as $\mathbbst{E}_n$-algebras (see \cite[Theorem 5.4.5.9]{lurie_higher_nodate}, or \cite[Corollary 5.2.11]{harpaz_little_nodate}). 
We refer to \cite{karlsson_assembly_2024} for the details and subtleties concerning condition 3, that is essentially automatically satisfied in the locally constant case (this is the case we are interested in). See also \cite{carmona_model_2022} for this last claim; in particular, Theorem 6.14 in loc.cit..
In a similar spirit as the result of Lurie, let us also mention a result \cite[Theorem 2.29]{Elliott-Safronov} by Elliott--Safronov identifying $\mathbbst{E}_n$-algebras with locally constant algebras over an operad whose colors are radii of discs (note that, contrary to us, they don't impose any lower bound on the radius of discs, which is the main source of difficulty in the present paper).

\subsection*{Renormalization}
 
Concrete examples of QFTs in the physics literature frequently involve \textit{a priori} ill-defined mathematical operations, e.g.\ integrating over infinite-dimensional spaces or discarding divergences to extract meaningful values. A mathematical treatment of these \textit{renormalization} problems was proposed by Costello in \cite{costello_renormalization_2011}, where an axiomatic framework for the ideas of low-energy effective field theory, developed by Kadanoff, Wilson and others, is proposed for the lagrangian formalism.
 An oversimplified summary of this approach is: first, due to natural physical limitations, we can only describe phenomena that can be observed in a bounded range of energies; second, if our description works for some energy bound, it should restrict to lower bounds; finally, relating the family of these descriptions for different bounds approaches a meaningful quantum field theory, describing phenomena at arbitrarily high-energies, if renormalization applies. A major breakthrough in the work of Costello--Gwilliam  \cite{costello_factorization_2017,costello_factorization_2021} is to make the renormalization framework of \cite{costello_renormalization_2011} compatible with their factorization algebra formalism. This requires a considerable amount of efforts to make the analysis ``get along well'' with the homological and homotopical algebra. 

\medskip

Despite tremendous technicalities, the idea is relatively simple, at least in the case $M=\mathbbst{R}^n$ that we explain now. Let $R$ be a positive real parameter that controls length scale (i.e.~inverse energy scale), and consider the family of operads $\Drect^{R}_{n}$ whose operations correspond to the algebraic structure that observables defined at scale bigger than $R>0$ have (in other words, $\Drect^{R}_{n}$ is the full sub-operad of $\Disc(\mathbbst{R}^n)$ spanned by discs of radius bigger than $R$)\footnote{In principle, condition 3 shall still be enforced but we are going to ignore it. }. Given a factorization algebra $\mathcal O$ on $\mathbbst{R}^n$, by restricting it to discs of radius bigger than $R$, for every $R$, one gets a compatible family $\{\mathcal O^{(R)}\}_{R>0}$, where $\mathcal O^{(R)}$ is
a $\Drect^{R}_{n}$-algebra. Conversely, one can ask whether such a family admits a meaningful limit when $R\to 0$ (i.e. when we approach arbitrarily small scale, or equivalently arbitrarily high energy). 

\medskip

The goal of this paper is to describe an interesting phenomenon occurring in the topological (meaning, locally constant) case: if $\mathcal O$ is a locally constant (pre)factorization algebra on $\mathbbst{R}^n$, then $\mathcal O$ can be reconstructed from the data $\mathcal O^{(R)}$ of its observables at scale bigger than $R$, for any choice of $R>0$. 
In other words, if we are given an effective theory at a fixed scale $R_0$ that is already topological (meaning that $\mathcal{O}^{(R_0)}$ is locally constant\footnote{Note that $\mathcal{O}^{(R_0)}$ will not be \emph{a priori} defined on small enough regions.}), then $\mathcal{O}^{(R_0)}$ produces an $\mathbbst{E}_{n}$-algebra (or equivalently a locally constant factorization algebra over $\mathbbst{R}^n$); in loose terms, for topological field theories it suffices to provide a description of observables (as a locally constant $\Drect^{R}_{n}$-algebra) for a fixed scale because it canonically propagates to every  scale. This seems to be an incarnation of the claim that renormalization is ``easy'' for topological field theories.


\subsection*{Main results and organization of the paper}

The main result of the paper, that is the mathematical incarnation of the fact observables at a given scale of a topological field theory propagate to every scale, is stated and proven in Section \ref{sect:Renormalization}, as Theorem \ref{thm:RenormalizationTFTs}. It says that there is an equivalence between the $\infty$-category of locally constant $\Drect_n^R$-algebras and the $\infty $-category of $\mathbbst{E}_n$-algebras. 
As we already mentioned above in this introduction, a similar result had been proven (by Lurie\footnote{However, our proof is orthogonal to Lurie's. Ultimately, his result follows from a shrinking argument and it works for any manifold $M$ while our approach is very specific to euclidean spaces $\mathbbst{R}^n$: for instance, it partly relies on the fact that one can move discs far away enough from each other. }) for locally constant $\Disc(\mathbbst{R}^n)$-algebras; it is a consequence of a more fundamental result claiming that the morphism of operads $\Disc(\mathbbst{R}^n)\to \mathbbst{E}_n$ exhibits $\mathbbst{E}_n$ as the localization of $\Disc(\mathbbst{R}^n)$ at all unary operations (see \cite[Lemma 5.4.5.11]{lurie_higher_nodate}, or \cite[Proposition 5.2.4]{harpaz_little_nodate}).
In the case of $\Drect_n^R$ this is also true: Theorem \ref{thm_LocalizationForTruncatedDisj} says that the morphism of operads $\Drect_n^R\to \mathbbst{E}_n$ exhibits $\mathbbst{E}_n$ as the localization of $\Drect_n^R$ at all unary operations. Nevertheless, the strategy of proof is different from the one of Lurie \cite{lurie_higher_nodate} and Harpaz \cite{harpaz_little_nodate}. In fact, their strategy, that uses weak approximations, cannot work in our case: indeed, we show in Appendix \ref{sect:CounterexampleWeakApprox} that the morphism $\Drect_n^R\to \mathbbst{E}_n$ is not a weak approximation. 

\medskip

In Section \ref{sect:Defects}, we extend the results from Section \ref{sect:Renormalization} to generalizations of $\mathbbst{E}_n$-algebras such as, for example, algebras over higher dimensional Swiss-cheese operads (see Theorems \ref{thm:ConstructiblePrefactAlgsScaleRonRn} and \ref{thm:ConstructiblePrefactAlgsScaleRonRpq}). The physical importance of these extensions
is that it also works for various types of topological field theories with defects.

\medskip

Section \ref{sect:Quantization} is devoted to an application of our main result: given a constant Poisson structure, we construct its Weyl algebra quantization using locally constant $\Drect^{R}_1$-algebras for $R=\frac12$. This toy example is meant to indicate the potential of Theorem \ref{thm:RenormalizationTFTs} when combined  with discrete models. Our strategy is inspired by the discussion of quantum mechanics in \cite{costello_factorization_2017} and $1d$ Chern-Simons theory from \cite[\textsection 3]{grady_batalin-vilkovisky_2017}, but strongly differs from both approaches, since ours does not rely on any analytical technique. In fact, we use a discretized version of the compactly supported de Rham complex, allowing us to avoid the traditional infinite-dimensional analysis. As a counterpart, we loose the locally constant property when the scale is smaller than the discretization mesh, but this is resolved using our main result.


\medskip

Appendix \ref{sect:LocalizationQuasioperads} contains useful results about localization of ($\infty$-)operads, that we needed in Section \ref{sect:Renormalization} and couldn't find in the literature. 

\medskip

The whole paper makes a crucial use of $\infty$-categories (see \cite{lurie_higher_2009}) and $\infty$-operads (see \cite{lurie_higher_nodate}).

\subsubsection*{Acknowledgments}

Damien Calaque thanks Giovanni Felder for countless discussions about discrete models, factorization algebras, and quantization.  Victor Carmona would like to thank Rune Haugseng for spotting a mistake in a previous proof of Corollary \ref{cor:InfinityLocalizationOfOperatorCats}. Both authors thank Roy Magen, who essentially proved a version of Theorem \ref{thm_LocalizationForTruncatedDisj} at the level of $\pi_0$ during an internship with the first author in 2019. VC was partially supported by the grant PID2020-117971GB-C21 funded by MCIN/AEI/10.13039/501100011033. 
Both authors have received funding from the European Research Council (ERC) under the
European Union’s Horizon 2020 research and innovation programme (Grant Agreement No.
768679). 

\section{Little discs at large scale}\label{sect:Renormalization}
We adopt here the point of view on prefactorization algebras in terms of (colored) operads from \cite[Ch.~3 \S1.2]{costello_factorization_2017}. Recall the operad with values in sets $\Disc(\mathbbst{R}^n)$ whose colors are (euclidean bounded) open discs of $\mathbbst{R}^n$, and having a (unique) multimorphism $(U_1,\dots,U_k)\to U_0$ if and only if $U_1\sqcup \dots\sqcup U_k\subset U_0$. 
Let $\EuScript{V}$ be a symmetric monoidal (sm) $\infty$-category. Recall that a \emph{prefactorization algebra} in $\EuScript{V}$ is a $\Disc(\mathbbst{R}^n)$-algebra in $\EuScript{V}$. 
A prefactorization algebra is \emph{locally constant} if it sends unary operations (that are inclusions of discs) to weak equivalences. 

\medskip

Lurie proved that locally constant prefactorization algebras are the same as $\mathbbst{E}_n$-algebras. More precisely, the functor $\upgamma^*\colon\EuScript{A}\mathsf{lg}_{\mathbbst{E}_n}(\EuScript{V})\to \EuScript{A}\mathsf{lg}_{\Disc(\mathbbst{R}^n)}^{\mathsf{lc}}(\EuScript{V})$ is an equivalence of $\infty$-categories. The latter functor is induced by the morphism of topological operads $\upgamma\colon\Disc(\mathbbst{R}^n)\to\mathbbst{E}_n$, that sends the multimorphism associated with an inclusion of discs $U_1\sqcup \dots\sqcup U_k\subset U_0$ to the following configuration of discs in the unit disc $\mathbbst{D}^n$: it is the image of the configuration $U_1\sqcup \dots\sqcup U_k$ through the unique affine bijection $U_0\cong\mathbbst{D}^n$. 
All colors of $\Disc(\mathbbst{R}^n)$ are sent to the single color of the topological operad $\mathbbst{E}_n$. 


\medskip

\begin{defn}\label{defn:TruncatedDisj} Let $R>0$. The $R$\emph{-truncated $n$-discs operad}, denoted  $\Drect_n^{R}$, is the full suboperad of $\Disc(\mathbbst{R}^n)$ spanned by open discs of radii strictly bigger than $R$.
\end{defn}

\begin{defn}\label{defn:PrefactAtUnitaryScale} Let $R>0$ and $\EuScript{V}$ be a symmetric monoidal $\infty$-category. 
\begin{itemize}
    \item A \emph{prefactorization algebra at scale} $R$ over $\mathbbst{R}^n$ is a $\Drect_n^{R}$-algebra (in $\EuScript{V}$), i.e.\ it is a prefactorization algebra over $\mathbbst{R}^n$ but only defined over euclidean discs of radii $>R$.
    \item A prefactorization algebra (possibly at a fixed scale) is \emph{locally constant} if it sends inclusions of discs to equivalences in $\EuScript{V}$.
\end{itemize}
\end{defn}

We still denote by $\upgamma\colon\Drect_n^{R}\to \mathbbst{E}_n$ the restriction of the morphism of topological operads $\upgamma$ to the $R$-truncated $n$-discs operad. The pullback functor 
$\upgamma^*\colon\EuScript{A}\mathsf{lg}_{\mathbbst{E}_n}(\EuScript{V})\to \EuScript{A}\mathsf{lg}_{\Drect_n^R}(\EuScript{V})$ 
factors through the full $\infty$-subcategory 
$\EuScript{A}\mathsf{lg}_{\Drect_n^R}^{\mathsf{lc}}(\EuScript{V})\subset\EuScript{A}\mathsf{lg}_{\Drect_n^R}(\EuScript{V})$ 
spanned by locally constant prefactorization algebras at scale $R$ over $\mathbbst{R}^n$. 

\begin{thm}\label{thm:RenormalizationTFTs} 
For any symmmetric monoidal $\infty$-category $\EuScript{V}$, the functor 
\[
\upgamma^*\colon\EuScript{A}\mathsf{lg}_{\mathbbst{E}_n}(\EuScript{V})\longrightarrow\EuScript{A}\mathsf{lg}^{\mathsf{lc}}_{\Drect_n^R}(\EuScript{V})
\]
is an equivalence of $\infty$-categories, between $\mathbbst{E}_n$-algebras and locally constant prefactorization algebras at scale $R$ over $\mathbbst{R}^n$. 
\end{thm}

This follows from the more abstract result concerning the $\infty$-operads themselves:
\begin{thm}\label{thm_LocalizationForTruncatedDisj}
The morphism $\upgamma\colon \Drect_n^{R}\to\mathbbst{E}_n$ exhibits $\mathbbst{E}_n$ as the $\infty$-localization of 
$\Drect_n^R$ at all unary operations as an $\infty$-operad.
\end{thm}

We will focus on Theorem \ref{thm_LocalizationForTruncatedDisj} and its proof will take the remainder of this section. 
It will require some detours: $(i)$  a topological discussion about configuration spaces in $\mathbbst{R}^n$ 
(Subsection \ref{subsect:DConfAndEquivalences})  and $(ii)$ a technically convenient variation of the operads $\Drect^R_n$ and $\mathbbst{E}_n$
(Subsection \ref{subsect:ClosedLittleDiscs}).

\begin{rem} In this document, we will focus on the \emph{unital} side of the story, e.g.\ $\mathbbst{E}_n$ is considered to have $\mathbbst{E}_n(0)\simeq \mathsf{pt}$ and $\Drect^R_n$ contains nullary operations. However, our results, and proofs, are valid for the obvious \emph{non-unital} variants. In practice, this means restricting to operations of arity strictly bigger than $0$ in $\Drect_n^R$ and taking $\mathbbst{E}_n^{\mathsf{nu}}=\big(\mathbbst{E}_n(k)\big)_{k\geq 1}$.
\end{rem}


\subsection{Fattened configurations in euclidean space}\label{subsect:DConfAndEquivalences}

\paragraph{Fattened configurations.} 
We now provide a brief discussion of fattened configurations spaces and their remarkable property of being equivalent 
to ordinary configuration spaces in $\mathbbst{R}^n$. This last property is key to prove 
Theorem \ref{thm_LocalizationForTruncatedDisj}.

\begin{defn}\label{defn:FattenedConfigurationsInRn}
Let $m$ be a non-negative integer. We define the space of \emph{configurations of $R$-fattened $m$-points in} 
$\mathbbst{R}^n$ to be the subspace of rectilinear embeddings of closed discs$\,$\footnote{Intersections of boundaries 
are not allowed.}
	\[
	\Drect^R\Conf_{m}(\mathbbst{R}^n)\subseteq\Emb^{\mathsf{rect}}(\overline{\mathbbst{D}}_1\sqcup\cdots\sqcup\overline{\mathbbst{D}}_m;\mathbbst{R}^n)
	\]
spanned by embeddings $f\colon\overline{\mathbbst{D}}_1\sqcup\cdots\sqcup\overline{\mathbbst{D}}_m\hookrightarrow \mathbbst{R}^n$ 
such that  $f(\overline{\mathbbst{D}}_j)$ is a closed disc with radius strictly bigger than $R$ for any $j$. By convention, we set $\Drect^R\Conf_0(\mathbbst{R}^n)=\mathsf{pt}$.
\end{defn}

Notice that the previous definition makes sense for $R=0$.

One fundamental feature about these configuration spaces is that, working on $\mathbbst{R}^n$, they are homotopy 
equivalent to genuine configurations spaces of points. .

\begin{rem}\label{rem:dilations}
Before proving that this is the case, we recall the following obvious facts: a dilation sends discs to discs, and is bijective (in particular it preserves disjointness and inclusions). 
\end{rem}

\begin{lem}\label{lem:InflationInducesHomotopyEquivalence} 
Evaluation at centers of discs determines homotopy equivalences 
	\[
        \Drect^R\Conf_{m}(\mathbbst{R}^n) \simeq \Drect^0\Conf_{m}(\mathbbst{R}^n) \simeq \Conf_{m}(\mathbbst{R}^n)\,,
    \]
where $\Conf_{m}(\mathbbst{R}^n)=\Emb(\{1,\dots,m\};\mathbbst{R}^n)$ denotes the space of configurations of $m$ 
(ordered) points in $\mathbbst{R}^n$. 
\end{lem}
\begin{proof} Let us fix a number $\upepsilon>R$ in the rest of this proof. We have a diagram of spaces
$$
\begin{tikzcd}[ampersand replacement=\&]
     \Drect^R\Conf_{m}(\mathbbst{R}^n) \ar[r,hookrightarrow,"\mathsf{inc}", shift left=1] \ar[r,"\mathsf{infl}"', leftarrow, dashed, shift right=1] \& \Drect^0\Conf_{m}(\mathbbst{R}^n) \ar[r,"\mathsf{ev}"] \& \Conf_{m}(\mathbbst{R}^n)\,,
\end{tikzcd}
$$
where $\mathsf{ev}$ denotes evaluation at centers of discs, $\mathsf{inc}$ is the canonical inclusion and $\mathsf{infl}$ is the continuous map described as follows: 
$$
\big(f\colon \overline{\mathbbst{D}}_1\sqcup\cdots\sqcup\overline{\mathbbst{D}}_m\hookrightarrow \mathbbst{R}^n\big)\;\longmapsto\; \big(\,\overline{\mathbbst{D}}_1\sqcup\cdots\sqcup\overline{\mathbbst{D}}_m\xrightarrow{\;\;\;f\;\;\;} \mathbbst{R} ^n\xrightarrow{\;\;\;\uplambda(f)\cdot\;\;\;}\mathbbst{R}^n\big),
$$
where $\uplambda(f)\cdot\colon \mathbbst{R}^n\to \mathbbst{R}^n$ is the dilation by $\uplambda(f):=\frac{\upepsilon}{\mathsf{min}(\upepsilon,r_1,\dots,r_m)}$ with $r_i$ the radius of the $i^{\text{th}}$-disc $f(\overline{\mathbbst{D}}_i)$.
In other words, $\mathsf{infl}$ sends the $i^{\text{th}}$-closed ball $f(\mathbbst{D}_i)=\overline{\mathbbst{D}}(x_i,r_i)$ in the configuration to $\overline{\mathbbst{D}}(\uplambda(f)x_i,\uplambda(f)r_i)$. Remark \ref{rem:dilations} ensures that the inflation map $\mathsf{infl}$ is well-defined.

The map $\mathsf{ev}$ is a homotopy equivalence by a classical shrinking argument. It remains to see that $\mathsf{infl}$ is a homotopy inverse of $\mathsf{inc}$. The composition $\mathsf{inc}\circ\mathsf{infl}\simeq \id$ is homotopic to the identity via the following $h\colon \Drect^0\Conf_m(\mathbbst{R}^n)\times [0,1]\to \Drect^0\Conf_m(\mathbbst{R}^n)$:
\[
\big(f\colon \overline{\mathbbst{D}}_1\sqcup\cdots\sqcup\overline{\mathbbst{D}}_m\hookrightarrow \mathbbst{R}^n,t\big)\;\longmapsto\; \big(\,\overline{\mathbbst{D}}_1\sqcup\cdots\sqcup\overline{\mathbbst{D}}_m\xrightarrow{\;\;\;f\;\;\;} \mathbbst{R} ^n\xrightarrow{\;\;\;\uplambda(f,t)\cdot\;\;\;}\mathbbst{R}^n\big),
\]
with $\uplambda(f,t)=1-t+\uplambda(f)t$. The other composition $\mathsf{infl}\circ\mathsf{inc}\simeq \id$ is analogous.
\end{proof}

In the proof of Theorem \ref{thm_LocalizationForTruncatedDisj}, we will need a more general version of this where we allow for \textit{nested configurations} of discs in $\mathbbst{R}^n$. For this reason, it is better to focus on the homotopy equivalence $\Drect^R\Conf_m(\mathbbst{R}^n)\simeq \Drect^0\Conf_m(\mathbbst{R}^n)$ rather than $\Drect^R\Conf_m(\mathbbst{R}^n)\simeq\Conf_m(\mathbbst{R}^n)$.

\paragraph{Nested configurations.} To keep track of the amount of closed discs conforming the configuration and how nested they are, we use a $k$-chain $\upalpha\in\Fun([k],\Fin_*)$, i.e.
\[
\upalpha\colon \left[\langle m_0\rangle \xrightarrow{\;\;\upalpha_1\;\;}\langle m_1\rangle  \xrightarrow{\;\;\upalpha_2\;\;} \cdots  \xrightarrow{\;\;\upalpha_k\;\;}\langle m_k\rangle \right].
\]

Our goal is to define the space of $\upalpha$\textit{-nested configurations} $\Drect^R\Conf_{\upalpha}(\mathbbst{R}^n)$ and to show that it is homotopy equivalent to $\Drect^0\Conf_{\upalpha}(\mathbbst{R}^n)$. A direct definition of such a space for a general $\upalpha$ is far from enlightening, and so we present a definition in two steps for the sake of visualization.

Let us start by assuming that all the components of $\upalpha$ are active maps in $\Fin_*$, i.e.\ the pointed function $\upalpha_r\colon \langle m_{r-1}\rangle \to \langle m_r\rangle$ satisfies $\vert \upalpha_r^{-1}(*)\vert =1$ for any $1\leq r\leq k$, and that $m_k=1$. In this case, $\upalpha$ represents a leveled-tree, see Picture \ref{fig:TreeAndPointDRConf}.
\begin{defn} Assume $\upalpha\in \Fun([k],\Fin_*)$ is a $k$-chain of active maps and that its target is $\langle m_k\rangle=\langle 1\rangle$. We define the space of $\upalpha$\emph{-nested configurations of $R$-fattened points in} 
$\mathbbst{R}^n$ to be the subspace
	\[
	\Drect^R\Conf_{\upalpha}(\mathbbst{R}^n)\subseteq\prod_{1\leq i_r\leq m}\Emb^{\mathsf{rect}}(\overline{\mathbbst{D}};\mathbbst{R}^n),
	\]
with $m=m_0+\dots +m_{k-1}$ and where we label the factors corresponding to $m_r$ by indices of the form $i_r$, spanned by collections of rectilinear embeddings $\big(f_{i_r}\colon\overline{\mathbbst{D}}\hookrightarrow \mathbbst{R}^n\big)_{i_r}$ 
such that:
\begin{itemize}
    \item $f_{i_r}(\overline{\mathbbst{D}})$ is a closed disc with radius strictly bigger than $R$ for any $i_r$.
    \item $f_{i_r}(\overline{\mathbbst{D}})\cap f_{j_r}(\overline{\mathbbst{D}})=\diameter$ for any pair of distinct indices $1_r\leq i_r, j_r\leq m_r$ (equivalently $(f_{1_r},\dots,f_{m_r})\colon \overline{\mathbbst{D}}_{1_r}\sqcup \dots \sqcup \overline{\mathbbst{D}}_{m_r}\hookrightarrow \mathbbst{R}^n$ is an element of $\Drect^{R}\Conf_{m_r}(\mathbbst{R}^n)$)
    \item $f_{i_{r-1}}(\overline{\mathbbst{D}})\subseteq f_{j_r}(\overline{\mathbbst{D}})$ for any $i_{r-1}\in \upalpha_r^{-1}(j_r)$ (equivalently $(f_{\upalpha_r^{-1}(j_r)})\colon \overline{\mathbbst{D}}\sqcup \dots\sqcup \overline{\mathbbst{D}}\hookrightarrow \mathbbst{R}^n$ factors through $f_{j_r}(\overline{\mathbbst{D}})$). Moreover, either the equality is met, $f_{i_{r-1}}(\overline{\mathbbst{D}})=f_{j_r}(\overline{\mathbbst{D}})$, in which case $\upalpha_r^{-1}(j_r)=\{i_{r-1}\}$, or $f_{i_{r-1}}(\overline{\mathbbst{D}})\subset f_{j_r}(\mathbbst{D})$ for any $i_{r-1}\in\upalpha_r^{-1}(j_r)$.
\end{itemize}
\end{defn}

See Picture \ref{fig:TreeAndPointDRConf} to visualize a point in this space. The conditions  are just the natural nesting rules associated to the $k$-chain $\upalpha$ plus mild point-set impositions that we need in the proof of Theorem \ref{thm_LocalizationForTruncatedDisj}.

\begin{figure}[htp]
	\centering
	\includegraphics[width=16cm]
	{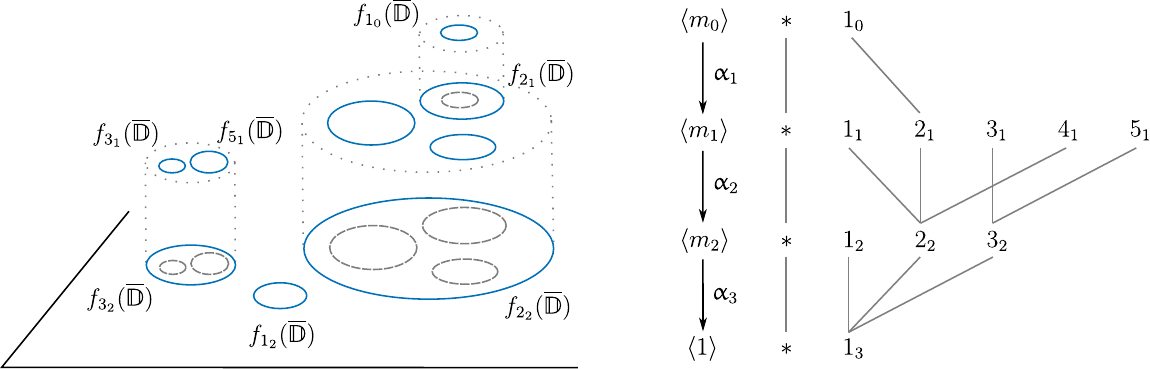}
	\caption{On the right hand side, a $k=3$ chain $\upalpha\colon [3]\to \Fin_* $ of active maps. On the left hand side, a point in $\Drect^R\Conf_{\upalpha}(\mathbbst{R}^n)$ for the previous $ \upalpha$.}\label{fig:TreeAndPointDRConf}
\end{figure}

If we consider more generally that $\upalpha$ is a $k$-chain of active maps, but $\langle m_k\rangle$ is arbitrary, the definition is the same, but replacing the second condition by $f_{i_r}(\overline{\mathbbst{D}})\cap f_{j_r}(\overline{\mathbbst{D}})=\diameter$ for any pair of distinct indices $1_r\leq i_r, j_r\leq m_r$ such that $\upalpha(i_r)=\upalpha(j_r)$, i.e.\ $\upalpha_k\cdots \upalpha_{r+1}(i_r)=\upalpha_k\cdots \upalpha_{r+1}(j_r)$ in $\langle m_r\rangle^{\circ}$. 

In the general case, i.e.\ $\upalpha\in \Fun([k],\Fin_*)$ is arbitrary, we have to also modify the amount of discs conforming the configuration, since the functions $\upalpha_r\colon \langle m_{r-1}\rangle \to \langle m_r \rangle$ might throw away some of them.
\begin{defn}\label{defn_NestedConfigurations} Let $\upalpha\in \Fun([k],\Fin_*)$ be an arbitrary $k$-chain of pointed functions. We define the space of $\upalpha$\emph{-nested configurations of $R$-fattened points in} 
$\mathbbst{R}^n$ to be the subspace
	\[
	\Drect^R\Conf_{\upalpha}(\mathbbst{R}^n)\subseteq\prod_{1\leq i_r\leq m}\Emb^{\mathsf{rect}}(\overline{\mathbbst{D}};\mathbbst{R}^n),
	\]
with $m=m_0-\vert\upalpha^{-1}_{1}(*)\vert+\dots +m_{k-1}-\vert \upalpha^{-1}_{k}(*)\vert$ and where we use labels $i_r$ for the factors corresponding to $\langle m_r\rangle \backslash \upalpha_{r+1}^{-1}(*)$ in the product, spanned by collections of rectilinear embeddings $\big(f_{i_r}\colon\overline{\mathbbst{D}}\hookrightarrow \mathbbst{R}^n\big)_{i_r}$ 
such that:
\begin{itemize}
    \item $f_{i_r}(\overline{\mathbbst{D}})$ is a closed disc with radius strictly bigger than $R$ for any $i_r$.
    \item $f_{i_r}(\overline{\mathbbst{D}})\cap f_{j_r}(\overline{\mathbbst{D}})=\diameter$ for any pair of indices $1_r\leq i_r, j_r\leq m_r$ such that $\upalpha(i_r)=\upalpha(j_r)$, i.e.\ $\upalpha_{t}\cdots \upalpha_{r+1}(i_r)=\upalpha_{t}\cdots \upalpha_{r+1}(j_r)\neq *$ for some $t\leq k$.
    \item $f_{i_{r-1}}(\overline{\mathbbst{D}})\subseteq f_{j_r}(\overline{\mathbbst{D}})$ for any $i_{r-1}\in \upalpha_r^{-1}(j_r)$. Moreover, either the equality is met, $f_{i_{r-1}}(\overline{\mathbbst{D}})=f_{j_r}(\overline{\mathbbst{D}})$, in which case $\upalpha_r^{-1}(j_r)=\{i_{r-1}\}$, or $f_{i_{r-1}}(\overline{\mathbbst{D}})\subset f_{j_r}(\mathbbst{D})$ for any $i_{r-1}\in\upalpha_r^{-1}(j_r)$.
\end{itemize}
\end{defn}

\begin{lem}\label{lem:InflationInducesHomotopyEquivalenceNestedVersion} Let $\upalpha\in \Fun([k],\Fin_*)$ be any $k$-chain of pointed functions. Then, the canonical inclusion of spaces of $\upalpha$-nested configurations 
	\[
        \Drect^R\Conf_{\upalpha}(\mathbbst{R}^n) \xhookrightarrow{\quad\quad} \Drect^0\Conf_{\upalpha}(\mathbbst{R}^n)
    \]
is a homotopy equivalence.
\end{lem}
\begin{proof} As in the proof of Lemma \ref{lem:InflationInducesHomotopyEquivalence}, we construct a homotopy inverse to the inclusion map
\[
\begin{tikzcd}[ampersand replacement=\&]
     \Drect^R\Conf_{\upalpha}(\mathbbst{R}^n) \ar[r,hookrightarrow,"\mathsf{inc}", shift left=1] \ar[r,"\mathsf{infl}"', leftarrow, dashed, shift right=1] \& \Drect^0\Conf_{\upalpha}(\mathbbst{R}^n)
\end{tikzcd}
\]
using appropriate dilations. 
In this case, fix a number $\upepsilon>R$ and define $\mathsf{infl}$ as follows:
\[
\big(f_{i_r}\colon \overline{\mathbbst{D}}\hookrightarrow \mathbbst{R}^n\big)_{i_r}\;\longmapsto\; \big(\,\overline{\mathbbst{D}}\xrightarrow{\;\;\;f_{i_r}\;\;\;} \mathbbst{R} ^n\xrightarrow{\;\;\;\uplambda(\underline{f})\cdot\;\;\;}\mathbbst{R}^n\big)_{i_r},
\]
where $\uplambda(\underline{f})\cdot\colon \mathbbst{R}^n\to \mathbbst{R}^n$ is the dilation by $\uplambda(\underline{f}):=\frac{\upepsilon}{\mathsf{min}(\upepsilon,r_1,\dots,r_m)}$ with $r_i$ the radius of the $i^{\text{th}}$-disc $f_i(\overline{\mathbbst{D}})$ and $m=m_0-\vert\upalpha^{-1}_{1}(*)\vert+\dots +m_{k-1}-\vert \upalpha^{-1}_{k}(*)\vert$. This map is well-defined due to Remark \ref{rem:dilations}.

Let us check that $\mathsf{infl}$ is a homotopy inverse of $\mathsf{inc}$. The  composition $\mathsf{inc}\circ\mathsf{infl}\simeq \id$ is homotopic to the identity via the following $h\colon \Drect^0\Conf_{\upalpha}(\mathbbst{R}^n)\times [0,1]\to \Drect^0\Conf_{\upalpha}(\mathbbst{R}^n)$:
\[
\big(f_{i_r}\colon \overline{\mathbbst{D}}\hookrightarrow \mathbbst{R}^n,t\big)_{i_r}\;\longmapsto\; \big(\,\overline{\mathbbst{D}}\xrightarrow{\;\;\;f_{i_r}\;\;\;} \mathbbst{R} ^n\xrightarrow{\;\;\;\uplambda(\underline{f},t)\cdot\;\;\;}\mathbbst{R}^n\big)_{i_r},
\]
with $\uplambda(\underline{f},t)=1-t+\uplambda(\underline{f})t$. The other composition $\mathsf{infl}\circ\mathsf{inc}\simeq \id $ is analogous.
\end{proof}

We also need to compare $\upalpha$-nested configurations in the open unit $n$-disc $\Drect^0\Conf_{\upalpha}(\mathbbst{D})$\footnote{The definition of this space is an obvious modification of Definition \ref{defn_NestedConfigurations}.} with nested configurations in euclidean space $\mathbbst{R}^n$.

\begin{lem}\label{lem:NestedD0ConfigurationsInDisc} Postcomposition with the canonical inclusion $\mathbbst{D}\hookrightarrow \mathbbst{R}^n$ induces a homotopy equivalence
\[
\Drect^0\Conf_{\upalpha}(\mathbbst{D})\xhookrightarrow{\quad\sim \quad}\Drect^0\Conf_{\upalpha}(\mathbbst{R}^n).
\]
\end{lem}
\begin{proof} Let us construct a homotopy inverse for this map between nested configurations. Fix an arbitrary number $\upepsilon>0$ and 
let $(f_i\colon \overline{\mathbbst{D}}\hookrightarrow \mathbbst{R}^n)_i$ be a point in $\Drect^0\Conf_{\upalpha}(\mathbbst{R}^n)$. Define 
\[
\upd_i=\vert\vert z(f_i(\overline{\mathbbst{D}}))\vert\vert + r_i=\vert\vert f_i(0)\vert\vert + \vert\vert f_i(1,0,\dots,0)-f_i(0)\vert\vert 
\]
as the sum of the distance of the center of $f_i(\overline{\mathbbst{D}})$ to the origin and the radius of  $f_i(\overline{\mathbbst{D}})$. This measures the maximum distance from $f_i(\overline{\mathbbst{D}})$ to the origin. Using these numbers, we can set
\[
\uplambda(\underline{f})=\frac{1}{\mathsf{max}(\upd_1,\dots,\upd_m)+\upepsilon}
\]
and consider the composition $\overline{\mathbbst{D}}\xrightarrow{\;\;\;f_{i}\;\;\;} \mathbbst{R} ^n\xrightarrow{\;\;\;\uplambda(\underline{f})\cdot\;\;\;}\mathbbst{R}^n$. By our choices, this composite factors through the canonical inclusion $\mathbbst{D}\hookrightarrow \mathbbst{R}^n$. Thanks to Remark \ref{rem:dilations}, we obtain a (dashed) continuous map
\[
\begin{tikzcd}[ampersand replacement=\&]
     \Drect^0\Conf_{\upalpha}(\mathbbst{D}) \ar[r,hookrightarrow, shift left=1] \ar[r, leftarrow, dashed, shift right=1] \& \Drect^0\Conf_{\upalpha}(\mathbbst{R}^n)
\end{tikzcd}
\]
Using homotopies similar to those of Lemma \ref{lem:InflationInducesHomotopyEquivalenceNestedVersion}, one checks that these two maps are homotopy inverses.
\end{proof}

\begin{rem} Notice that, since there is no rectilinear isomorphism $\mathbbst{D}\cong \mathbbst{R}^n$, we cannot directly find a homeomorphism $\Drect^0\Conf_{\upalpha}(\mathbbst{D})\cong \Drect^0\Conf_{\upalpha}(\mathbbst{R}^n)$. In the same spirit, one cannot use the diffeomorphism
\[
\begin{tikzcd}[ampersand replacement=\&]
		\upvarphi\colon \&[-1.3cm] \mathbbst{D}\ar[rr,"\simeq"] \&\& \mathbbst{R}^n\\[-20pt]
		 \& x\ar[rr,mapsto] \&\& \frac{1}{\sqrt{1-\vert\vert x\vert\vert^2}}x
	\end{tikzcd}
\]
in an obvious way to produce a homeomorphism $\Drect^0\Conf_{\upalpha}(\mathbbst{D})\cong \Drect^0\Conf_{\upalpha}(\mathbbst{R}^n)$ for a general $k$-chain $\upalpha$. For instance, postcomposition with $\upvarphi$ does not preserve euclidean discs. On the other hand, it is possible to produce a homeomorphism as above for $1$-chains using $\upvarphi$, but the construction fails for general $k$-chains.

\end{rem}



\subsection{Operads of closed discs}\label{subsect:ClosedLittleDiscs}
Due to point-set considerations we need to address later on, it will be more convenient for us to work with closed variants of $\Drect^R_n$ and $\mathbbst{E}_n$.

\begin{defn}\label{defn:ClosedTruncatedDisj} Let $R>0$. The $R$\emph{-truncated closed $n$-discs operad}, denoted  $\overline{\Drect}\!\,^{R}_n$, is the operad with colors closed euclidean bounded discs in $\mathbbst{R}^n$ of radii strictly bigger than $R$, and multimorphisms:
\[
\overline{\Drect}\!\,^R_n\brbinom{\;\overline{D}_1,\dots,\overline{D}_m\;}{\overline{D}}=\left\{\begin{array}{cr}
	\mathsf{pt} & \qquad\text{ if }\bigsqcup_i\overline{D}_i\subset \mathsf{int}(\overline{D}),\\\\ 
    \lbrace\id\rbrace & \qquad \text{ if }m=1\text{ and }\overline{D}_1=\overline{D},
    \\\\
	\diameter & \qquad\text{ otherwise.}
\end{array}\right.
\]
\end{defn}

\begin{defn}\label{defn:ClosedLittleDiscs} The \textit{closed little} $n$\textit{-discs operad} $\overline{\mathbbst{E}}_n$ is the uncolored operad whose space of $m$-ary operations is the subspace 
$$
\overline{\mathbbst{E}}_n(m)\subseteq\Emb^{\mathsf{rect}}\big(\overline{\mathbbst{D}}\,^{\sqcup m};\overline{\mathbbst{D}}\big)
$$
spanned by rectilinear embeddings $f\colon \overline{\mathbbst{D}}\,^{\sqcup m}\hookrightarrow \overline{\mathbbst{D}}$ that do not touch the boundary of the target disc $\partial \overline{\mathbbst{D}}$, except for the identity embedding $\id\colon \overline{\mathbbst{D}}\hookrightarrow \overline{\mathbbst{D}}$ for $m=1$. Composition is given by composition of rectilinear embeddings.
\end{defn}

\begin{rem}
    The closed variant of the little $n$-disc operad that we use is slightly different from the classical one, since we do not allow for intersections of boundaries (except for identity operations). This stems from a point-set issue we have to handle in the proof of Theorem \ref{thm_LocalizationForTruncatedDisj}.
\end{rem}

There are natural maps of operads $ \overline{\Drect}\!\,^R_n\to \Drect^R_n$  and $\overline{\mathbbst{E}}_n\to\mathbbst{E}_n$ determined respectively by
\[
\overline{D}\longmapsto \mathsf{int}(\overline{D}) \qquad \text{and} \qquad (f\colon \overline{\mathbbst{D}}\,^{\sqcup m}\hookrightarrow \overline{\mathbbst{D}})\longmapsto (f\vert_{\mathsf{int}}\colon \mathbbst{D}^{\sqcup m}\hookrightarrow \mathbbst{D}).
\]
There is also a morphism of operads $\overline{\upgamma}\colon \overline{\Drect}\!\,^R_n\to \overline{\mathbbst{E}}_n$ constructed in the same way as its \textit{open} analogue $\upgamma \colon \Drect^R_n\to \mathbbst{E}_n$. All these morphisms together fit into a commutative square
\begin{equation}\label{eqt:CommutativeSquareOfOperads}
    \begin{tikzcd}[ampersand replacement=\&]
    \Drect^R_n\ar[r,"\upgamma"] \& \mathbbst{E}_n\\
    \overline{\Drect}\!\,^R_n \ar[u] \ar[r,"\overline{\upgamma}"'] \& \overline{\mathbbst{E}}_n \ar[u]
\end{tikzcd}.
\end{equation}

Before further exploration of these morphisms, let us introduce a bit of notation that will be used later on. First, recall that given a (simplicial) operad $\Op$, we can consider its fibrational presentation $\Op^{\otimes}\to \Fin_*$, also called (simplicial) operator category (see discussion above \cite[Proposition 4.1.13]{harpaz_little_nodate}). Additionally, we will use the following notation taken from \cite[\textsection 3]{low_fractions_2015}.
\begin{notat} Let $\mathdutchcal{I}$ and $\mathdutchcal{C}=(\Ctt,\Wtt)$ be (small) relative categories. 
By $\left[\mathdutchcal{I}, \mathdutchcal{C}\right]$ we will refer to the category of relative functors and by $\mathsf{weq}\left[\mathdutchcal{I}, \mathdutchcal{C}\right]$ to its wide subcategory whose morphisms are the $\Wtt$-natural transformations. Equipping the $k^{\text{th}}$-simplex category $[k]=\{0\leq 1\leq\dots\leq k\}$ with its minimal relative structure, we have in particular $\mathsf{weq}\left[[k],\mathdutchcal{C}\right]$ for any $k\geq 0$.


\end{notat}

We are ready to come back to our square of operads (\ref{eqt:CommutativeSquareOfOperads}).

\begin{lem}\label{lem:EnvsClosedEn} The map $\overline{\mathbbst{E}}_n\to \mathbbst{E}_n$ is an equivalence of operads.
\end{lem}
\begin{proof} By a simple shrinking argument, one observes that evaluation at the center of discs induces homotopy equivalences
$
\overline{\mathbbst{E}}_n(m)\tilde\longrightarrow \Conf_m(\mathbbst{D})\tilde\longleftarrow \mathbbst{E}_n(m)
$ 
making the obvious triangle commute for any $m\geq 0$.
\end{proof}

\begin{lem}\label{lem:DRvsClosedDR} The map $\overline{\Drect}\!\,^R_n\to \Drect^R_n$ induces an equivalence of $\infty$-operads between the $\infty$-localizations at all unary operations of both operads. In particular, the forgetful functor $
\EuScript{A}\mathsf{lg}^{\mathsf{lc}}_{\Drect^R_n}(\EuScript{V})\longrightarrow\EuScript{A}\mathsf{lg}^{\mathsf{lc}}_{\overline{\Drect}\!\,^R_n}(\EuScript{V})$
is an equivalence of $\infty$-categories for any symmetric monoidal $\infty$-category $\EuScript{V}$.
\end{lem}
\begin{proof} Let us consider the associated map between operator categories $\overline{\Drect}\!\,^{R,\otimes}_n\to \Drect^{R,\otimes}_n$. Also, let us denote $\mathdutchcal{D}$ (resp.\ $\overline{\mathdutchcal{D}}$) the relative category $(\Drect^{R,\otimes}_n,\Wrect_{\bullet})$ (resp.\ $\overline{\Drect}\!\,^{R,\otimes}_n$) obtained by considering the wide subcategory $\Wrect_{\bullet}$ of $\Drect^{R,\otimes}_n$ on the morphisms
\[
\left\{(\upvarphi,\underline{w})\colon \big(\langle m\rangle, \underline{U}\big) \to \big(\langle m\rangle,\underline{U}'\big):\quad \upvarphi\colon \langle m\rangle\xrightarrow{\;\simeq\;} \langle m\rangle\quad \text{and}\quad w_{i}\colon U_{\upvarphi^{-1}(i)}\hookrightarrow U'_i\;\text{ in }\Drect^R_n
\right\}.
\]
Working at the level of quasioperads (or operator categories) allows us to combine Corollary \ref{cor:InfinityLocalizationOfOperatorCats} and Proposition \ref{prop_HinichCriteria} to reduce the claim to: for any $k\geq 0$ and any $\upsigma\in \mathsf{weq}\left[[k],\mathdutchcal{D}\right]$, the slice category $\mathsf{weq}\left[[k],\overline{\mathdutchcal{D}}\right]_{/\upsigma}$ is weakly contractible.

Let us start with the case $k=0$. Then, $\upsigma\in \mathsf{weq}\left[[0],\mathdutchcal{D}\right]$ is the same as choosing an object $(\langle m\rangle,\underline{U})\in \Drect^{R,\otimes}_n$ and the associated slice category $\mathsf{weq}\left[[k],\overline{\mathdutchcal{D}}\right]_{/\upsigma}$ can be identified with the poset of $m$-tuples $(\overline{D}_1,\dots,\overline{D}_m)$ of discs in $\overline{\Drect}\!\,^{R}_n$ such that $\mathsf{int}(\overline{D}_i)\subseteq U_i$ for any $i$. Thus, it is immediate to see that this poset is weakly contractible. 

For $k\geq 1$, the situation is similar, but one has to be more careful since the operad $\Drect^R_n$ contains a multimorphism $(U_1,\dots, U_m)\to V$ even if $\partial U_i\cap \partial U_j\neq \diameter$ for a pair of distinct indices or if $\partial U_i\cap \partial V\neq \diameter$ for some $i$ and $U_i\subsetneq V$, in contrast to $\overline{\Drect}\!\,^{R}_n$. Now, $\upsigma\in \mathsf{weq}\left[[k],\mathdutchcal{D}\right]$ corresponds to a $k$-chain of maps in $\Drect^{R,\otimes}_n$  
\[
\upsigma\colon \left[\big(\langle m_0\rangle,\underline{U}^0\big) \xrightarrow{\;\;(\upalpha_1,\underline{\upmu}^1)\;\;}\big(\langle m_1\rangle,\underline{U}^1\big)  \xrightarrow{\;\;(\upalpha_2,\underline{\upmu}^2)\;\;} \cdots  \xrightarrow{\;\;(\upalpha_k,\underline{\upmu}^k)\;\;}\big(\langle m_k\rangle,\underline{U}^k\big) \right].
\]
Hence, the additional datum with respect to the case $k=0$ is that the families of discs come with inclusions $\bigsqcup_{i\in\upalpha_r^{-1}(j)}U^{r-1}_{i}\subseteq U^r_j$ whenever $j\in \langle m_r\rangle^{\circ}$ and $1\leq r\leq k$. In this case, the slice category $\mathsf{weq}\left[[k],\overline{\mathdutchcal{D}}\right]_{/\upsigma}$ is equivalent to the poset of tuples $(\overline{D}\!\,^0_1,\dots,\overline{D}\!\,^{0}_{m_0},\dots,\overline{D}\!\,^{k}_{1},\dots,\overline{D}\!\,^{k}_{m_k})$ of discs in $\overline{\Drect}\!\,^R_n$, denoted $\overline{D}\!\,^{*}_{\bullet}$, subject to:
\begin{itemize}
    \item $\mathsf{int}(\overline{D}\!\,^r_i)\subseteq  U^r_i$ for any $i\in \langle m_r\rangle^{\circ}$ and $0\leq r\leq k$,
    \item $\bigsqcup_{i\in\upalpha_r^{-1}(j)}\overline{D}\!\,^{r-1}_{i}\subset \mathsf{int}(\overline{D}\!\,^r_j)$ whenever $j\in \langle m_r\rangle^{\circ}$ or $\overline{D}\!\,^{r-1}_i=\overline{D}\!\,^{r}_j$ and $\upalpha_r^{-1}(j)=\{i\}$.
\end{itemize}
We argue that this poset is weakly contractible by showing that for any finite poset $\Krect$ and any functor $\Krect\to  \mathsf{weq}\left[[k],\overline{\mathdutchcal{D}}\right]_{/\upsigma}$, there is a diagonal filler
    \[
    \begin{tikzcd}[ampersand replacement=\&]
		\Krect\ar[rr] \ar[d]\&\&\mathsf{weq}\left[[k],\overline{\mathdutchcal{D}}\right]_{/\upsigma}\\ 
		 \mathsf{pt} \ar[rru, dashed] 
	\end{tikzcd}
    \]
 up to a zigzag of natural transformations. The idea is to gently shrink the discs $U^r_{i}$ so that the additional impositions for an object in $\mathsf{weq}\left[[k],\overline{\mathdutchcal{D}}\right]_{/\upsigma}$ are met; for instance, the shrunken discs should have radii strictly bigger than $R $ and they should satisfy $\partial V^{r}_{i}\cap\partial V^{r}_{j}=\diameter$ whenever $i,j$ are distinct elements of $\langle m_r\rangle^{\circ}$ such that $\upalpha_r(i)=\upalpha_{r}(j)$.

Let us set some notation by observing that the functor $\upphi\colon \Krect\to  \mathsf{weq}\left[[k],\overline{\mathdutchcal{D}}\right]_{/\upsigma}$ assigns tuples of discs $\overline{D}\!\,^*_{\bullet}(p)$ to elements $p\in \Krect$ and, on morphisms, it sends
\[
q\leq p \;\xmapsto{\quad\quad} \; \upphi_{qp}\colon \overline{D}\!\,^{*}_{\bullet}(q)\to \overline{D}\!\,^{*}_{\bullet}(p),
\]
where $\upphi_{qp}$ is comprised of unary morphisms $\overline{D}\!\,^r_i(q)\to \overline{D}\!\,^{r}_{i}(p)$ in $\overline{\Drect}\!\,^{R}_{n}$. We begin by considering a natural transformation $\upphi'\Rightarrow \upphi$ obtained by shrinking the discs conforming $\upphi$ (i.e.\ by considering discs with the same center, but smaller or equal radii) such that the resulting discs satisfy $\overline{D}\!\,^{\prime,r}_i(p)\subset U^r_{i}$ for any triple $(p,r,i)$, instead of the open condition $\mathsf{int}\big(\overline{D}\!\,^{\prime,r}_i(p)\big)\subseteq U^r_{i}$. For each triple $(p,r,i)$, we have to ensure:
\begin{itemize}
    \item $\displaystyle{\bigsqcup_{i\in \upalpha_r^{-1}(j)}}\overline{D}\!\,^{\prime,r-1}_i(p)\subset D^{\prime,r}_j(p)$ when $j\in \langle m_r\rangle^{\circ}$ or $\overline{D}\!\,^{\prime,r-1}_{i}(p)=\overline{D}\!\,^{\prime, r}_{j}(p)
    $ where $\upalpha_r^{-1}(j)=\{i\}$,
    \item $\overline{D}\!\,^{\prime,r}_i(q)\subset D^{\prime,r}_i(p)$ or $\overline{D}\!\,^{\prime,r}_i(q)=\overline{D}\!\,^{\prime,r}_i(p)$ whenever $q\leq p$,
\end{itemize}
and these conditions will be satisfied exactly when they hold for the components of $\upphi$. For example, $\overline{D}\!\,^{\prime,r}_i(q)=\overline{D}\!\,^{\prime,r}_i(p)$ when $\overline{D}\!\,^{r}_i(q)=\overline{D}\!\,^{r}_i(p)$.

To justify that such a natural transformation exists, we use a double induction over $(p,r)$. Being pedantic, the induction over $p\in \Krect$ requires choosing a well-ordering on $\Krect$ extending its partial ordering $\leq$ or, as we prefer, noting that we can work inductively over the levels of the Hasse diagram of $\Krect$; the elements in level $0$ are the minimal elements in $\Krect$, level $1$ is comprised of elements $p\in\Krect$ not in level $0$ such that $\nexists\ q\in \Krect:$ $p_0\lneq q\lneq p$ for any $p_0$ in level $0$, and so on.    

Start the induction at $p\in \Krect$ in level $0$ (i.e.\ $p$ is minimal in $K$) and $r=0$. This case is easy because most of the conditions are void, but it uses the fact that discs $\overline{D}\!\,^0_i(p)$ have radii strictly bigger than $R$. We just need to shrink the discs whose boundary touches the boundary of the relevant components of $\upsigma$. We continue by keeping $p$ fixed and running an induction over $r$. Having constructed the shrunken discs $\overline{D}\!\,^{\prime,t}_i(p)$ for $i\in \langle m_t\rangle^{\circ}$ and $t<r$, one notices that it is possible to shrink $\overline{D}\!\,^{r}_{j}(p)$ to meet the previous conditions. There are two cases: $(a)$ we can set $\overline{D}\!\,^{\prime, r-1}_i(p)=\overline{D}\!\,^{\prime,r}_j(p)$ when $\overline{D}\!\,^{r-1}_i(p)=\overline{D}\!\,^{r}_j(p)$, or $(b)$ we use that radii are strictly bigger than $R$ and that at most we have to make sure that a finite number of disjoint closed balls are contained in $D^{\prime,r}_j(p)\subseteq D^r_j(p)$ (knowing that they sit inside $D^r_{j}(p)$). The remaining steps in the induction over $p$ follow the same lines, since the worst case requires to shrink $\overline{D}\!\,^{r}_i(p)$ so that $\overline{D}\!\,^{\prime,r}_i(q)\subset D^{\prime,r}_i(p)$ for $q<p$ in lower levels.   

Using the new functor $\upphi'$ just constructed, we build a natural transformation $\upphi'\Rightarrow \mathsf{ct}_{\overline{V}\!\,^{*}_{\bullet}}$, where $\mathsf{ct}_{\overline{V}\!\,^{*}_{\bullet}}$ is the constant functor $\Krect\to \mathsf{weq}\left[[k],\overline{\mathdutchcal{D}}\right]_{/\upsigma}$ associated to a tuple $\overline{V}\!\,^*_{\bullet}$. The components $V^r_i$ of such a tuple are constructed by shrinking $U^r_i$. This is now possible since $\overline{D}\!\,^{\prime,r}_i(p)\subset U^r_{i}$ for any triple $(p,r,i)$. The formal construction applies an induction over $r$ similar to the one discussed above.\footnote{The case $\Krect=\diameter$ is also covered by the argument.}

Recapping, for any functor $\upphi\colon\Krect\to  \mathsf{weq}\left[[k],\overline{\mathdutchcal{D}}\right]_{/\upsigma}$, we have constructed a diagram
   \[
    \begin{tikzcd}[ampersand replacement=\&]
		\Krect\ar[rrrr,"\upphi"]\ar[rrrr,phantom, ""{name=A, below}]\ar[rrrr,"\upphi'"{name=B, description}, bend right=21] \ar[ddrr,"!"'] \&\&\&\& \mathsf{weq}\left[[k],\overline{\mathdutchcal{D}}\right]_{/\upsigma}\\\\ 
		\&\& \mathsf{pt} \ar[rruu, dashed,"\overline{V}\!\,^*_{\bullet}"']  \ar[Rightarrow, from=B] \ar[Rightarrow, from=B, to=A]
	\end{tikzcd}.
    \]
This shows that $\mathsf{weq}\left[[k],\overline{\mathdutchcal{D}}\right]_{/\upsigma}$ is weakly contractible, concluding the proof.
\end{proof}

\begin{rem} Using the notation from the proof of Lemma \ref{lem:DRvsClosedDR}, notice that the slice category $\mathsf{weq}\left[[k],\overline{\mathdutchcal{D}}\right]_{\upsigma/}$ might be empty and that $\mathsf{weq}\left[[k],\overline{\mathdutchcal{D}}\right]_{/\upsigma}$ is not filtered nor cofiltered for some $\upsigma\in \mathsf{weq}\left[[k],\mathdutchcal{D}\right]$. See Figure  \ref{fig:PathologicalCofilteredness}. Therefore, our strategy to show weak contractibility of the slices in the cited proof cannot be avoided easily.
\end{rem}

\begin{figure}[htp]
	\centering
	\includegraphics
	{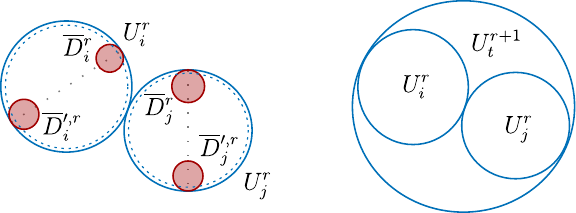}
	\caption{On the left, components of two elements $\overline{D}\!\,^{*}_{\bullet},\, \overline{D}\!\,^{\prime,*}_{\bullet}\in \mathsf{weq}\left[[k],\overline{\mathdutchcal{D}}\right]_{/\upsigma}$ which admit no upper nor lower bound in that poset, i.e.\ the slice is not filtered nor cofiltered. On the right, components of an element $\upsigma\in \mathsf{weq}\left[[k],\mathdutchcal{D}\right]$ such that $\mathsf{weq}\left[[k],\overline{\mathdutchcal{D}}\right]_{\upsigma/}$ is empty.}\label{fig:PathologicalCofilteredness}
\end{figure}


\subsection{Proof of Theorem \ref{thm_LocalizationForTruncatedDisj}}\label{subsect:ProofLocalizationTHM}

 First notice that the commutative square (\ref{eqt:CommutativeSquareOfOperads}) together with Lemmas \ref{lem:EnvsClosedEn} and \ref{lem:DRvsClosedDR} reduce the question about $\upgamma\colon \Drect^R_n\to \mathbbst{E}_n$ to its analogue for $\overline{\upgamma}\colon\overline{\Drect}\!\,_n^{R}\to \overline{\mathbbst{E}}_n$. To prove this claim, we work with quasioperads and we show the following equivalent statement (see Corollary \ref{cor:InfinityLocalizationOfOperatorCats}): $\Nrect(\overline{\Drect}\!\,_n^{R,\otimes})\to \Nrect_{\infty}(\overline{\mathbbst{E}}\!\,_n^{\otimes})$ exhibits $\Nrect_{\infty}(\overline{\mathbbst{E}}\!\,_n^{\otimes})$ as the $\infty$-localization of $\Nrect(\overline{\Drect}\!\,_n^{R,\otimes})$ at $\Nrect(\Wrect_{\bullet})$ as an $\infty$\emph{-category}. We apply Proposition \ref{prop_HinichCriteria} and Remark \ref{rem:HinichCriteriaPLUS} to see that this is the case. 

 Recycling the notation $\overline{\mathdutchcal{D}}=(\overline{\Drect}\!\,^{R,\otimes}_n,\Wrect_{\bullet})$ from the proof of Lemma \ref{lem:DRvsClosedDR}, we must show that for any $k\geq 0$
 \[
\overline{\upgamma}_*\colon \mathsf{weq}\left[[k],\overline{\mathdutchcal{D}}\right]\xrightarrow{\qquad\qquad} \mathsf{weq}\left[[k],\overline{\mathbbst{E}}\!\,_n^{\otimes,\sharp}\right]\equiv \Fun([k],\overline{\mathbbst{E}}\!\,_n^{\otimes})^{\simeq}
 \]
is a weak homotopy equivalence. Making use of the commutative triangle
\[
		\begin{tikzcd}[ampersand replacement=\&]
		\mathsf{weq}\left[[k],\overline{\mathdutchcal{D}}\right]\ar[rr]\ar[rd] \&\& \mathsf{weq}\left[[k],\overline{\mathbbst{E}}\!\,_n^{\otimes,\sharp}\right]\ar[ld]\\
			\& \mathsf{weq}\left[[k],\Fin_*^{\sharp}\right]
		\end{tikzcd}\quad,
\]	
we can further reduce the claim to show that the horizontal map induces weak homotopy equivalences between homotopy fibers of the diagonal arrows over $k$-chains $\upalpha\in \Fun([k],\Fin_{*})$. Hence, let us fix $\upalpha$ and look at the mentioned homotopy fibers, denoted $\mathsf{weq}\left[[k],\overline{\mathdutchcal{D}}\right]_{\upalpha}$ and $\mathsf{weq}\left[[k],\overline{\mathbbst{E}}\!\,_n^{\otimes,\sharp}\right]_{\upalpha}$ respectively. 

For $k=0$, the right diagonal map becomes an equivalence and hence, we must check that for any $\langle m\rangle \in \Fin_*$, the category
$
\mathsf{weq}\left[[0],\overline{\mathdutchcal{D}}\right]_{\langle m\rangle}
$
is weakly contractible. It is clearly so because it is filtered.

For $k\geq 1$, since $\mathsf{weq}\left[[k],\overline{\mathbbst{E}}\!\,_n^{\otimes,\sharp}\right]$ is the $k^{\text{th}}$-level of the complete Segal space $\Nrect^{\Rrect}_{\infty}(\overline{\mathbbst{E}}\!\,^{\otimes,\sharp}_n)_{\bullet}$,
\[
\mathsf{weq}\left[[k],\overline{\mathbbst{E}}\!\,_n^{\otimes,\sharp}\right]\xrightarrow{\quad \simeq\quad} \coprod_{\langle m_0\rangle,\dots,\langle m_k\rangle} \Map_{\overline{\mathbbst{E}}\!\,_n^{\otimes}}(\langle m_{k-1}\rangle,\langle m_{k}\rangle) \times \cdots \times\Map_{\overline{\mathbbst{E}}\!\,_n^{\otimes}}(\langle m_0\rangle,\langle m_1\rangle).
\]
Taking the homotopy fiber over $\upalpha$, and using the definition of $\overline{\mathbbst{E}}_n$ (where we only allow for rectilinear embeddings whose image do not intersect the boundary, except for identity embeddings), we get an identification $\mathsf{weq}\left[[k],\overline{\mathbbst{E}}\!\,_n^{\otimes,\sharp}\right]_{\upalpha}\cong \Drect^0\Conf_{\upalpha}(\mathbbst{D})$. Notice that we use the obvious variant of Definition \ref{defn_NestedConfigurations} where the $\upalpha$-nested configurations live in the open $n$-dimensional unit disc $\mathbbst{D}$. On the other hand, observe that $\mathsf{weq}\left[[k],\overline{\mathdutchcal{D}}\right]_{\upalpha}$ can be identified with the poset of tuples $\underline{U}^{*}=\big(\underline{U}^{0},\dots, \underline{U}^{k}\big)$, where each $\underline{U}^{r}=\big(U^{r}_{i_r}\big)_{i_r\in \langle m_r\rangle^{\circ}}$ is again a tuple of euclidean open discs in $\mathbbst{R}^n$ with radii strictly bigger than $R$, subject to:
\begin{itemize}
    \item $\overline{U}\!\,_{i_r}^{r}\cap\overline{U}\!\,_{j_r}^{r}=\diameter$ if $\upalpha_{t}\cdots \upalpha_{r+1}(i_r)=\upalpha_{t}\cdots \upalpha_{r+1}(j_r)\neq *$ for some $t\leq k$,
    \item either $\overline{U}\!\,_{i_{r-1}}^{r-1}\subset U_{j_{r}}^{r}$ or $U^{r-1}_{i_{r-1}}=U^{r}_{j_r}$ if $i_{r-1}\in \upalpha_{r}^{-1}(j_r)$.
\end{itemize}

The rest of the proof exploits the following (homotopy) commutative diagram
\[
\begin{tikzcd}[ampersand replacement=\&]
\mathsf{weq}\left[[k],\overline{\mathdutchcal{D}}\right]_{\upalpha}\ar[d,"(5)"']\ar[rr] \&\& \mathsf{weq}\left[[k],\overline{\mathbbst{E}}\!\,_n^{\otimes,\sharp}\right]_\upalpha\ar[d,"(1)"]\\
\widetilde{\mathsf{weq}}\left[[k],\overline{\mathdutchcal{D}}\right]_{\upalpha}\ar[d,"(4)"'] \&\& \Drect^0\Conf_{\upalpha}(\mathbbst{D})\ar[d,"(2)"]\\
\Drect^R\Conf_{\upalpha}(\mathbbst{R}^n) \ar[rr,"(3)"'] \&\& \Drect^0\Conf_{\upalpha}(\mathbbst{R}^n)
\end{tikzcd},
\]
where $\widetilde{\mathsf{weq}}\left[[k],\overline{\mathdutchcal{D}}\right]_{\upalpha}$ is the poset of tuples $\underline{U}^*$ defined as above, with the only difference that the components $\underline{U}^r=\big(U^r_{i_r}\big)$ are indexed by $i_r\in \langle m_r\rangle \backslash \upalpha_{r+1}^{-1}(*)$ and $r$ runs through $0\leq r< k$. By 2-out of-3, we will be done if we show that $(1)-(5)$ are weak homotopy equivalences.

The map $(1)$ is the previous identification and $(2)-(3)$ were shown to be homotopy equivalences in Lemmas \ref{lem:InflationInducesHomotopyEquivalenceNestedVersion} and \ref{lem:NestedD0ConfigurationsInDisc}.  An application of Quillen's Theorem A shows that the projection functor $(5)$ is a weak homotopy equivalence. More concretely, one can observe that for any $\underline{U}^*\in \widetilde{\mathsf{weq}}\left[[k],\overline{\mathdutchcal{D}}\right]_{\upalpha}$, the slice $\big(\mathsf{weq}\left[[k],\overline{\mathdutchcal{D}}\right]_{\upalpha}\big)_{/\underline{U}^*}$ is filtered, essentially because the components indexed by $i_r\in \upalpha_{r+1}^{-1}(*)\cap\langle m_r\rangle^{\circ}$ can be taken as big as needed.

It remains to see that $(4)$ is a weak homotopy equivalence to conclude the proof. The idea is to apply Lurie--Seifert--van Kampen's Theorem \cite[A.3.1]{lurie_higher_nodate} and for that purpose, we must interpret elements in $\Prect:=\widetilde{\mathsf{weq}}\left[[k],\overline{\mathdutchcal{D}}\right]_{\upalpha}$ as open subsets of $\Drect^R\Conf_{\upalpha}(\mathbbst{R}^n)$. Let us define a functor
$\upzeta\colon\Prect\to\mathsf{Open}\big(\Drect^R\Conf_{\upalpha}(\mathbbst{R}^n)\big)$ by sending each $\underline{ U}^{*}$ to the subspace $\upzeta(\underline{ U}^{*})$ of $\Drect^R\Conf_{\upalpha}(\mathbbst{R}^n)$ spanned by embeddings whose image is contained in $\underline{ U}^{*}$, i.e.
		\[
		\upzeta(\underline{ U}^{*})=\Drect^R\Conf_{\upalpha}(\mathbbst{R}^n)\cap\prod_{1\leq i_r\leq m}\Emb^{\mathsf{rect}}\big(\overline{\mathbbst{D}}, U_{i_r}^{r}\big),
		\] 
with $m=m_0-\vert\upalpha^{-1}_{1}(*)\vert+\dots +m_{k-1}-\vert \upalpha^{-1}_{k}(*)\vert$ and where we use labels $i_r$ for the factors corresponding to $\langle m_r\rangle \backslash \upalpha_{r+1}^{-1}(*)$ in the product. With this functor at hand, the map $(4)$ can be seen as 
\[
		\widetilde{\mathsf{weq}}\left[[k],\overline{\mathdutchcal{D}}\right]_{\upalpha}\simeq \hocolim\big(\Prect\overset{\upzeta}{\longrightarrow}\EuScript{S}\mathsf{pc}\big)\longrightarrow \Drect^R\Conf_{\upalpha}(\mathbbst{R}^n),
\]
		where the map on the left is an equivalence since $\upzeta(\underline{ U}^{*})$ is contractible for any $ \underline{ U}^{*}\in \Prect$. To see that the right map is an equivalence we apply Lurie--Seifert--vanKampen's Theorem. It suffices to check that for every $(f_{i_r})_{i_r}\in\Drect^R\Conf_{\upalpha}(\mathbbst{R}^n)$ the subposet 
		$
		\left\lbrace \underline{ U}^{*}\in\Prect:\; (f_{i_r})_{i_r}\in\upzeta(\underline{ U}^{*})\right\rbrace\subseteq \Prect
		  $
		is cofiltered. It is non-empty since we assumed that embeddings in $\Drect^R\Conf_{\upalpha}(\mathbbst{R}^n)$ do not allow intersections of boundaries (see Figure \ref{fig:NonEmptynessGeneralSlice}). It remains to see that for any pair of distinct elements of the poset $\upzeta(\underline{ U}^*)\ni (f_{i_r})_{i_r}\in\upzeta(\underline{ V}^*)$, one can find a lower bound $(f_{i_r})_{i_r}\in\upzeta(\underline{ W}^*)$, i.e. $\underline{ U}^*\leftarrow \underline{ W}^*\rightarrow \underline{ V}^*$.

        Constructing such a $\underline{W}^*$ amounts to build tuples of discs $\underline{W}^r=(W^r_{i_r})_{i_r\in \langle m_r\rangle\backslash\upalpha_{r+1}^{-1}(*)}$ for $0\leq r<k$ subject to the following conditions:
        \begin{itemize}
            \item[$(a)$] $f_{i_r}(\overline{\mathbbst{D}})\subset W^r_{i_r}$ for all $i_r\in \langle m_r\rangle \backslash\upalpha_{r+1}^{-1}(*)$ and $0\leq r<k$,
           
            \item[$(b)$] $\overline{W}\!\,^r_{i_r}\subset U^r_{i_r}\cap V^r_{i_r}$ for all pairs $(r,i_r)$,
            

            \item[$(c)$] either $\bigsqcup_{t\in\upalpha_r^{-1}(i_r)}\overline{W}\!\,^{r-1}_{t}\subset W^r_{i_r}$ or $\overline{W}\!\,^{r-1}_{i_{r-1}}=\overline{W}\!\,^{r}_{i_r}$ with $\upalpha_r^{-1}(i_r)=\{i_{r-1}\}$ for all pairs $(r,i_r)$.
        \end{itemize}
        Notice that condition $(a)$ implies that the radii of the discs are always strictly bigger than $R$. Also, observe that $U^r_{i_r}\cap V^r_{i_r}$ is a convex open subset of $\mathbbst{R}^n$ containing the closed disc $f_{i_r}(\overline{\mathbbst{D}})$. We argue that this family can be built by working inductively over $0\leq r<k$.
        
        Starting at $r=0$, we find the disc $W^0_{i_0}$ for each $i_0\in \langle m_0\rangle \backslash \upalpha_1^{-1}(*)$ by ``expanding" $f_{i_0}(\mathbbst{D})$ inside $U^0_{i_0}\cap V^0_{i_0}$. This means that $W^0_{i_0}$ shares its center with $f_{i_0}(\mathbbst{D})$, i.e.\ the point $f_{i_0}(0)$, and its radius is strictly bigger than that of $f_{i_0}(\mathbbst{D})$. At this level, this seems sufficient since condition $(c)$ is void for $r=0$. However, to ensure that next steps in the construction are doable, we additionally impose: $\overline{W}\!\,^0_{i_0}\subset f_{i_t}(\mathbbst{D})$ if $f_{i_0}(\overline{\mathbbst{D}})\subset f_{i_t}(\mathbbst{D})$, where $i_t$ is the first index of the form $i_t=\upalpha_t\cdots\upalpha_1(i_0)\in \langle m_t\rangle\backslash\upalpha_{t+1}^{-1}(*)$ for which such an inclusion occurs. If there is no such index, i.e.
        \[
        f_{i_0}(\overline{\mathbbst{D}})=f_{i_1}(\overline{\mathbbst{D}})=\cdots =f_{i_l}(\overline{\mathbbst{D}})
        \]
        with $i_s=\upalpha_s\cdots\upalpha_1(i_0)$ for any $1\leq s\leq l$ and either $l=k-1$ or $\upalpha_{l+2}\upalpha_{l+1}(i_l)=*$, we impose no further condition. The previous expansion can be performed because $f_{i_0}(\overline{\mathbbst{D}})$ is a closed disc sitting inside the convex open subset $U^0_{i_0}\cap V^0_{i_0}$ and the open disc $f_{i_t}(\mathbbst{D})$.

        For $r=1$ and any $i_1\in \langle m_1\rangle\backslash\upalpha_2^{-1}(*)$, we want to do a similar process: find $W^1_{i_1}$ by expanding $f_{i_1}(\mathbbst{D})$ inside $U^1_{i_1}\cap V^1_{i_1}$ and $f_{i_t}(\mathbbst{D})$, where $i_t$ is the first index obtained by applying components of $\upalpha$ to $i_0$ such that $f_{i_1}(\overline{\mathbbst{D}})\subset f_{i_t}(\mathbbst{D})$.  Now, we need to make sure that condition $(c)$ holds (which depends on the previous step $r=0$). There are two cases: either $\bigsqcup_{t\in \upalpha_{1}^{-1}(i_1)}f_{t}(\overline{\mathbbst{D}})\subset f_{i_1}(\mathbbst{D})$ or $f_{i_0}(\overline{\mathbbst{D}})=f_{i_1}(\overline{\mathbbst{D}})$ for the unique $i_0\in \upalpha_1^{-1}(i_1)$. In the second case, it suffices to set $W^1_{i_1}:=W^0_{i_0}$. Otherwise, we are allowed to perform the expansion of $f_{i_1}(\mathbbst{D})$ because $\overline{W}\!\,^0_{t} $ with $t\in \upalpha_1^{-1}(i_1)$ was selected so that it is contained in $f_{i_1}(\mathbbst{D})$ and because of the inclusion $f_{i_1}(\overline{\mathbbst{D}})\subset U^1_{i_1}\cap V^1_{i_1}\cap f_{i_t}(\mathbbst{D})$. 

        The general case follows by the same argument given for $r=1$; just make the appropriate index-modifications. Recapping, we have constructed a lower bound $\underline{W}^*$ in the subposet of $\Prect$ for the pair of elements $\underline{U}^*,\underline{V}^*$, showing that it is cofiltered, therefore weakly contractible. 
        
        \hfill$\Box$

\begin{figure}[htp]
	\centering
	\includegraphics[width=16cm]
	{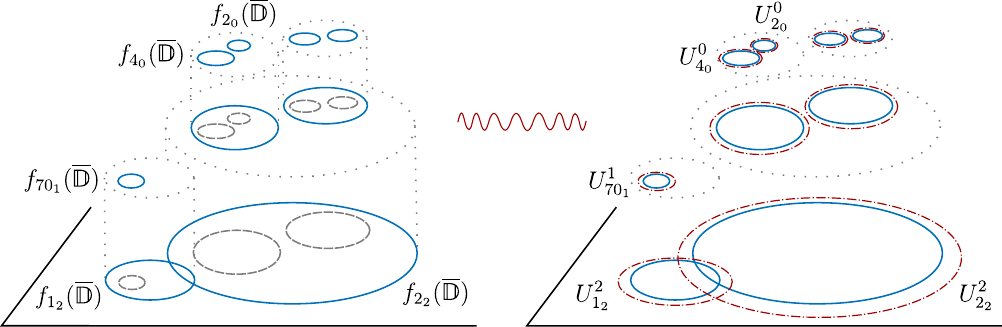}
	\caption{A point $(f_{i_r})_{i_r}\in \Drect^R\Conf_{\upalpha} (\mathbbst{R}^n)$ and $\underline{U}^*\in \Prect$ such that $(f_{i_r})_{i_r}\in \upzeta(\underline{U}^*) $.}\label{fig:NonEmptynessGeneralSlice}
\end{figure}

\begin{rem} The introduction of the modified poset $\widetilde{\mathsf{weq}}\left[[k],\overline{\mathdutchcal{D}}\right]_{\upalpha}$ in the previous proof is fundamental. If one just tries to apply Lurie--Seifert--van Kampen's Theorem to the obvious variation $\upzeta'\colon \mathsf{weq}\left[[k],\overline{\mathdutchcal{D}}\right]_{\upalpha}\longrightarrow\mathsf{Open}\big(\Drect^R\Conf_{\upalpha}(\mathbbst{R}^n)\big)$ of the functor $\upzeta$, one faces the problem that the subposets
\[
\left\lbrace \underline{ U}^{*}\in\mathsf{weq}\left[[k],\overline{\mathdutchcal{D}}\right]_{\upalpha}:\; (f_{i_r})_{i_r}\in\upzeta'(\underline{ U}^{*})\right\rbrace\subseteq \mathsf{weq}\left[[k],\overline{\mathdutchcal{D}}\right]_{\upalpha}
\]
are not always cofiltered. The issue arises when trying to find the $i_r^{\text{th}}$-component of a ``lower bound'' for a pair of distinct elements $\underline{U}^*$, $\underline{V}^*$ where the index belongs to $i_r\in \upalpha_{r+1}^{-1}(*)\cap \langle m_r\rangle^{\circ}$. It can happen that $U^r_{i_r}\cap V^r_{i_r}$ does not contain any disc of radius strictly bigger than $R$.
\end{rem}


\subsection{A cubical version}\label{subsect:LargeScaleCubes}

So far, we have chosen to work with euclidean discs in $\mathbbst{R}^n$, but the arguments do not really use anything special from the euclidean norm (except for convexity of intersections of balls). An inspection of our proofs reveal that other norms on $\mathbbst{R}^n$ yield operads for which our main result hold. We will focus on the $\infty$-norm
$
\vert\vert x\vert\vert_{\infty}=\mathsf{max} \{\vert x_i\vert:\ 1\leq i\leq n\}.
$
The practical difference is that balls become cubes for $\vert\vert\,\text{-}\,\vert\vert_{\infty}$, e.g.\ 
\[
\square:=(-1,1)^n=\{y\in \mathbbst{R}^n:\ \vert\vert y \vert\vert_{\infty}< 1\}=:\mathsf{B}_{\infty}(0,1),
\]
and sometimes this has some advantages over working with euclidean discs.

It is fairly simple to modify our definitions replacing euclidean discs by cubes to get $R$-\emph{truncated/little} $n$-\emph{cubes operads}:
\[
\begin{tikzcd}
 \Drect^R_n \ar[d,rightsquigarrow]& \overline{\Drect}\!\,^R_n \ar[d,rightsquigarrow] & \mathbbst{E}_n \ar[d,rightsquigarrow] & \overline{\mathbbst{E}}_n \ar[d,rightsquigarrow]  \\
 \Crect^R_n & \overline{\Crect}\!\,^{R}_n & \mathbbst{E}^{\square}_{n} & \overline{\mathbbst{E}}\!\,^{\square}_n
\end{tikzcd}\quad.
\]

Notice that all $\mathbbst{E}_n$, $\mathbbst{E}^{\square}_n$, $\overline{\mathbbst{E}}_n$ and $\overline{\mathbbst{E}}\,\!^{\square}_n$ are equivalent operads. Moreover, locally constant algebras over $\Crect^R_n$ and $\overline{\Crect}\!\,^R_n$ can be defined in an obvious way (see Definition \ref{defn:PrefactAtUnitaryScale}).

Following the proof strategy we wrote for Theorem \ref{thm_LocalizationForTruncatedDisjStratifiedI}, which summarizes our proof of Theorem \ref{thm_LocalizationForTruncatedDisj}, one can check that the arguments we applied for euclidean discs work with cubes. The final output is:
\begin{thm} For any symmmetric monoidal $\infty$-category $\EuScript{V}$, restriction along the operad maps $\overline{\Crect}\!\,^R_n\to \Crect^R_n\to \mathbbst{E}^{\square}_n$ establish equivalences of $\infty$-categories
\[
\EuScript{A}\mathsf{lg}_{\mathbbst{E}_n}(\EuScript{V})\simeq\EuScript{A}\mathsf{lg}_{\mathbbst{E}_n^{\square}}(\EuScript{V})\xrightarrow{\;\;\sim\;\;}\EuScript{A}\mathsf{lg}^{\mathsf{lc}}_{\Crect_n^R}(\EuScript{V}) \xrightarrow{\;\;\sim\;\;} \EuScript{A}\mathsf{lg}^{\mathsf{lc}}_{\overline{\Crect}\!\,_n^R}(\EuScript{V})\;. 
\]
\end{thm}

To illustrate the potential of working with cubes, we refer to the recent work of Benini-Fern\'andez-Schenkel \cite{benini_derived_2024} where they have constructed a ``rectangular prefactorization algebra on $\mathbbst{Z}^2$'' modeling a version of lattice $2d$ Yang--Mills theory.  

Let us reproduce their definition of rectangular prefactorization algebra. \begin{defn} The \emph{rectangular prefactorization operad} $\mathcalpxtx{P}_{\mathbbst{Z}^n}$ on the lattice $\mathbbst{Z}^n$ is the colored operad with:
\begin{itemize}
        \item[\textbf{Obj:}] its colors are rectangular subsets $\overline{V}=\prod_i[a_i,b_i]\subseteq \mathbbst{Z}^n$ where each side has length $\geq 2$;\footnote{Notice that their scale-restriction is closed, while our operads impose crucially a restriction of the form $>R$. Of course, their choice of size limit is not fundamental.} 
        \item[\textbf{Mor:}] it has a (unique) multimorphism $(\,\overline{V}_1,\dots,\overline{V}_m)\to \overline{V}$ if and only if $\overline{V}_1\sqcup \dots\sqcup \overline{V}_k\subseteq \overline{V}$. 
    \end{itemize}

 A $\mathcalpxtx{P}_{\mathbbst{Z}^n}$-algebra $\mathcalpxtx{F}$ is \emph{locally constant} if $\mathcalpxtx{F}\big(\overline{V}\big)\to \mathcalpxtx{F}\big(\overline{V}\!\,'\big) $ is an equivalence for every inclusion of rectangular subsets $\overline{V}\hookrightarrow \overline{V}\!\,'$.
\end{defn}

Let us fix $R\geq\frac32$. A ball of radius $R\geq\frac32$ for the $\infty$-norm is a cube with sides having length bigger than $2R\geq3$; hence each side contains an integer interval of length $\geq 2$. 
\begin{prop}\label{prop P vs C}
Any $\mathcalpxtx{P}_{\mathbbst{Z}^n}$-algebra determines a $\overline{\Crect}\!\,^R_n$-algebra. Furthermore, the last one is locally constant if so is the original.
\end{prop}
\begin{proof}
    There is an obvious morphism of operads $\overline{\Crect}\!\,^R_n\to \mathcalpxtx{P}_{\mathbbst{Z}^n}$ sending $\overline{U}$ to $\overline{U}\cap \mathbbst{Z}^n$. Pulling back along this morphism provides the desired functor 
    $\EuScript{A}\mathsf{lg}_{\mathcalpxtx{P}_{\mathbbst{Z}^n}}(\EuScript{V})\to \EuScript{A}\mathsf{lg}_{\overline{\Crect}\!\,^R_n}(\EuScript{V})$, and clearly preserves local constancy. 
\end{proof}


The consequence of this result is that the \emph{classical lattice Yang--Mills model on} $\mathbbst{Z}^2$ from \cite[Theorem 5.2]{benini_derived_2024} yields an $\mathbbst{E}_2$-monoidal dg-category. Thus, we obtain an explicit $2d$ example to complement the $1d$ examples from \textsection\ref{sect:Quantization} below and \cite{calaque_not_2024}.

\section{Defects}\label{sect:Defects}
The discussion in Section \ref{sect:Renormalization} aims to identify locally constant prefactorization algebras at a fixed scale over $\mathbbst{R}^n$, corresponding to plain topological field theories (TFTs), but the methods are general enough to cover \emph{theories with defects}. For such a purpose, we are going to introduce a family of constructible prefactorization algebras at a fixed scale, controlled by suitable stratified spaces, which encode TFTs with defects. We are not pursuing the maximum level of generality here. Rather, we want to discuss two scenarios (working on $\mathbbst{R}^n$ and on $\mathbbst{R}^p\times \mathbbst{R}_{\geq 0}^{q}$) which hopefully illustrate the ideas. 

\subsection{Linear defects}\label{subsect:LinearDefects} 
Let us first focus on euclidean spaces and ``linear defects'' on them. 

\begin{defn}\label{defn:LinearStratificationsRn} A \emph{linear stratification} on $\mathbbst{R}^n$ consists of a finite filtration 
	$$
	 X_{0}\subsetneq\dots\subsetneq X_{d}=\mathbbst{R}^n
	 $$
	  such that $ X_{j}$ is a linear subspace of $\mathbbst{R}^n$ for all $0\leq j\leq d$. We will denote by  $\mathbbst{R}^n_{\upchi}$ or simply $\upchi$ the data of such a linear stratification on $\mathbbst{R}^n$. By convention we will consider $ X_{-1}=\diameter$.
\end{defn}

It will be useful in the sequel to denote $z( U)$ the center of a disc $ U$.

\begin{exs} The four linear stratifications on $\mathbbst{R}^2$:
$$
\begin{matrix}
    \mathbbst{R}^2_{\diameter}:\big(\mathbbst{R}^2\big), && 
    \mathbbst{R}^2_{\boxdot}:\big(\left\{0\right\}\subsetneq\mathbbst{R}^2\big), && 
    \mathbbst{R}^2_{\boxminus}:\big( L\subsetneq \mathbbst{R}^2\big),  &&
    \mathbbst{R}^2_{\boxdottedline}:\big(\left\{0\right\}\subsetneq  L\subsetneq \mathbbst{R}^2\big),
\end{matrix}
$$
where $ L$ is a line in $\mathbbst{R}^2$ passing through the origin.
\end{exs}

Given a linear stratification $\mathbbst{R}^n_{\upchi}$, we can modify the operad $\Drect^R_{n}$ to encode defects:
\begin{defn}\label{defn:TruncatedDiscsWithLinearStratifications}
The $R$-truncated $\upchi$-discs operad, denoted $\Drect^{R}_{\upchi}$, is the full suboperad of $\Drect^{R}_{n}$ spanned by open (euclidean) discs $ U\subseteq \mathbbst{R}^n$ satisfying:
$$
    \text{if }z( U)\text{ belongs to the }j^{\text{th}}\text{-stratum, i.e.\ }z( U)\in X_{j}\backslash X_{j-1}\text{, then } U\cap X_{j-1}= \diameter.
$$
\end{defn}

Using the $R$-truncated closed $n$-discs operad $\overline{\Drect}\!\,^R_n$ (see Definition \ref{defn:ClosedTruncatedDisj}) instead of $\Drect^R_n$ in the previous definition,
we find the $R$\emph{-truncated closed} $\upchi$\emph{-discs operad} $\overline{\Drect}\!\,^R_{\upchi}$.

\begin{defn}\label{defn:ConstructiblePrefactRn} Let  $\mathbbst{R}^{n}_{\upchi}$ be a linear stratification of $\mathbbst{R}^n$ and $ \EuScript{V}$ a sm-$\infty$-category. 
\begin{itemize}
    \item A \emph{prefactorization algebra at scale} $R$ over $\mathbbst{R}^n_{\upchi}$ is a $\Drect_{\upchi}^{R}$-algebra (in $\EuScript{V}$). 
    
    \item A prefactorization algebra $\mathdutchcal{F}$ at scale $R$ over $\mathbbst{R}^n_{\upchi}$ is \emph{constructible} if for any inclusion of discs $ U\subseteq V$ such that  $z( U)$, $z( V)$ lie in the same stratum $ X_{j}\backslash  X_{j-1}$, the induced morphism $\mathdutchcal{F}( U)\to \mathdutchcal{F}( V)$ is an equivalence in $ \EuScript{V}$.
\end{itemize}
\end{defn}

\begin{rem} If one forgets about scales, it is possible to adapt the former definitions to produce \emph{constructible prefactorization algebras} over $\mathbbst{R}^n_{\upchi}$ (for simplicity only defined on open discs of $\mathbbst{R}^n$): they are algebras over the operad $\Disc(\mathbbst{R}^n_{\upchi})\coloneqq \Drect_{\upchi}^{0}$. 
\end{rem}

In order to introduce the analog of $\mathbbst{E}_n$-algebras when defects are allowed, we have to identify the objects/colors that the corresponding operad should have. For this purpose, consider the minimal equivalence relation on $\mathsf{ob}(\Drect^{R}_{\upchi})$ generated by
$$
 U\underset{\text{pre}}{\sim} V \text{ if } U\subseteq  V \text{ and both }z( U)\text{ and }z( V)\text{ lie in the same stratum } X_{j}\backslash X_{j-1}.
$$

One readily deduces:
\begin{lem}\label{lem:ObjectsOfLittleUpchiDiscsOperad} There is a one-to-one correspondence between $\mathsf{ob}(\Drect^{R}_{\upchi})/\!\!\sim$ and the set of pairs  $\upchi'_{\upvarrho}\equiv(\upchi',\upvarrho)$ where:
\begin{itemize}
    \item  $\upchi'$ is a linear stratification on $\mathbbst{R}^n$ obtained from $\mathbbst{R}^n_{\upchi}$ by forgetting some interval of steps $[0, t]$ for $t\in\left\{-1,\dots,d\right\}$ in the filtration, and
    \item $\upvarrho$ represents a connected component of $ X_{t+1}\backslash X_{t}$ for the previous $t$.
\end{itemize} 
\end{lem}

\begin{ex} Consider the linear stratification  $\mathbbst{R}^2_{\boxdottedline}:\big(\left\{0\right\}\subsetneq  L\subsetneq \mathbbst{R}^2\big)$. Since $ L\backslash\left\{0\right\}$ has two connected components, say $l$ and $r$, as well as $\mathbbst{R}^2\backslash L$, namely $u$ and $d$,
the previous lemma yields $\mathsf{ob}(\Drect^{R}_{\boxdottedline})\slash\!\!\sim\;=\left\{\square_{u},\square_{d},\boxminus_{l}, \boxminus_{r}, \boxdottedline \right\}$, 
where
$$
\begin{matrix}
	\square_u=\left[\big(\mathbbst{R}^2\big)_{u}\right], && && \square_d=\left[\big(\mathbbst{R}^2\big)_{d}\right], \\[4mm]
    \boxminus_{l}=\left[\big( L\subsetneq \mathbbst{R}^2\big)_{l}\right], && &&
    \boxminus_{r}=\left[\big( L\subsetneq \mathbbst{R}^2\big)_{r}\right],
     \\[4mm]
     &&
    \boxdottedline=\left[\big(\left\{0\right\}\subsetneq  L\subsetneq \mathbbst{R}^2\big)\right].
\end{matrix}
$$
Note that the linear filtration $\mathbbst{R}^2_{\boxdot}:\big(\left\{0\right\}\subsetneq \mathbbst{R}^2\big)$ does not appear in this quotient set. The reason is that it corresponds to forgetting the middle step in the linear filtration $\mathbbst{R}^2_{\boxdottedline}$ and not an interval of steps of the form $[0,t]$. 
\end{ex}

Now, we must define the spaces of admissible embeddings when defects are allowed. For that purpose, we use the following partial order on $\mathsf{ob}(\Drect^{R}_{\upchi})/\!\!\sim$ to control the admissibility of embeddings: $\upchi_{\upvarrho}'\leq \upchi_{\uppsi}''$ if $(i)$ $\upchi'$ can be obtained from $\upchi''$ by forgetting some interval of steps $[0,t]$ in the filtration and $(ii)$ $\uppsi$ is contained in the closure of $\upvarrho$ in $\mathbbst{R}^n$. 

\begin{defn}\label{defn:LittleStratifiedDiscsOperad} Let $\mathbbst{R}^n_{\upchi}:\big( X_{0}\subsetneq \dots\subsetneq  X_{d}=\mathbbst{R}^n\big)$ be a linear stratification and $\upchi'_{\upvarrho}$, $\upchi''_{\uppsi}$ be two equivalence classes in $\mathsf{ob}(\Drect^{R}_{\upchi})\slash\!\!\sim$.
\begin{itemize}
    \item If $\upchi'_{\upvarrho}\leq \upchi''_{\uppsi}$, we define a $\upchi$-\emph{admissible rectilinear embedding} $\upchi'_{\upvarrho}\hookrightarrow \upchi''_{\uppsi}$ to be a rectilinear embedding $\mathbbst{D}\hookrightarrow \mathbbst{D}$ which preserves the induced stratifications. We denote by  $\Emb_{\upchi}^{\mathsf{rect}}(\upchi'_{\upvarrho},\upchi''_{\uppsi})\subseteq \Emb^{\mathsf{rect}}(\mathbbst{D},\mathbbst{D})$ the subspace of $\upchi$-admissible rectilinear embeddings. By convention, $\Emb_{\upchi}^{\mathsf{rect}}(\upchi'_{\upvarrho},\upchi''_{\uppsi})=\diameter$ if $\upchi'_{\upvarrho}\leq \upchi''_{\uppsi}$ does not hold, i.e.\ if $\upchi'$ cannot be obtained from $\upchi''$ by forgetting some interval of steps $[0,t]$ in the filtration or $\uppsi\nsubseteq\overline{\upvarrho}$.
    Admissible rectilinear embeddings can be composed and contain the identity map.
    
    \item The \emph{little $(\upchi,n)$-discs operad}, denoted $\mathbbst{E}_{\upchi}$, is given by:
    \begin{itemize}
        \item[\textbf{Obj:}] its set of objects/colors is $\mathsf{ob}(\mathbbst{E}_{\upchi})=\mathsf{ob}(\Drect^{R}_{\upchi})\slash\!\!\sim$; 
        \item[\textbf{Mor:}] its spaces of multimorphisms are defined by 
        $$
        \mathbbst{E}_{\upchi}\brbinom{\underline{\upchi_{\upvarrho}'}}{\upchi''_{\uppsi}}=
        \prod_i \Emb_{\upchi}^{\mathsf{rect}}\big(\upchi'_{i,\upvarrho(i)},\upchi''_{\uppsi}\big)\,\cap \,\Emb\Big(\bigsqcup_i\mathbbst{D}_i,\mathbbst{D}\Big)
        $$
        with the obvious composition of rectilinear embeddings and identities. See Fig.\ \ref{fig:OperationInDefects}.
    \end{itemize}
\end{itemize}
\end{defn}

\begin{figure}[htp]
	\centering
	\includegraphics[height=5cm]
	{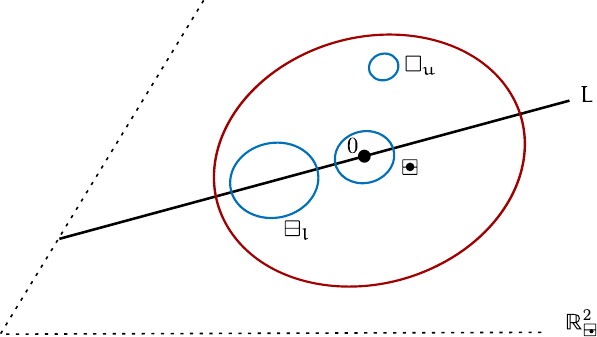}
	\caption{An operation in $\mathbbst{E}_{\boxdottedline}\brbinom{\{\boxdottedline,\square_{u},\boxminus_{l}\}}{\boxdottedline}$, where $\mathbbst{E}_{\boxdottedline}$ represents the little $(\upchi,2)$-discs operad  associated to the linear stratification $\mathbbst{R}^2_{\boxdottedline}$.}\label{fig:OperationInDefects}
\end{figure}
\begin{rem} One can modify the definition of $\mathbbst{E}_{\upchi}$ to get the \emph{closed little} $(\upchi,n)$\emph{-discs operad} $\overline{\mathbbst{E}}\!\,_{\upchi}$ as done in Definition \ref{defn:ClosedLittleDiscs} for the little $n$-discs operad $\mathbbst{E}_n$. For instance, the operation depicted in Figure \ref{fig:OperationInDefects} also belongs to $\overline{\mathbbst{E}}\!\,_{\boxdottedline}$.
\end{rem}

There is an obvious map $\upgamma\colon\Drect_{\upchi}^{R}\to \mathbbst{E}_{\upchi}$ from the $R$-truncated $\upchi$-discs operad into the little $(\upchi,n)$-discs operad, sending  $\mathbbst{E}_{\upchi}$-algebras to constructible prefactorization algebras at scale $R$ over $\mathbbst{R}^n_{\upchi}$. The triviality of renormalization for TFTs with defects in this case corresponds to:
\begin{thm}\label{thm:ConstructiblePrefactAlgsScaleRonRn} 
   Constructible prefactorization algebras at scale $R$ over $\mathbbst{R}^n_{\upchi}$ are the same as $\mathbbst{E}_{\upchi}$-algebras. More concretely, the functor 
	$$
	\upgamma^*\colon\EuScript{A}\mathsf{lg}_{\mathbbst{E}_{\upchi}}( \EuScript{V})\to  
	\begin{Bmatrix}
\textup{constructible }\EuScript{V}\text{-prefactorization}\\
\textup{algebras at scale }R\textup{ over }\mathbbst{R}^n_{\upchi}
\end{Bmatrix}
	$$
	is an equivalence of $\infty$-categories for any sm-$\infty$-category $ \EuScript{V}$.
\end{thm}

\begin{rem} If one defines constructible factorization algebras over $\mathbbst{R}^n_{\upchi}$ in the obvious way, Theorem \ref{thm:ConstructiblePrefactAlgsScaleRonRn} may be used to provide an equivalence of $\infty$-categories
$$
\begin{Bmatrix}
\textup{constructible }\EuScript{V}\text{-factorization}\\ 
\textup{algebras over }\mathbbst{R}^n_{\upchi}
\end{Bmatrix} \simeq  
\begin{Bmatrix}
\textup{constructible }\EuScript{V}\text{-prefactorization}\\
\textup{algebras at scale }R\textup{ over }\mathbbst{R}^n_{\upchi}
\end{Bmatrix}
$$
as in the locally constant case (no stratification on $\mathbbst{R}^n$).
\end{rem}

As expected, Theorem \ref{thm:ConstructiblePrefactAlgsScaleRonRn} follows from the analog of Theorem \ref{thm_LocalizationForTruncatedDisj}.

\begin{thm}\label{thm_LocalizationForTruncatedDisjStratifiedI}
	The morphism of operads $\upgamma\colon \Drect_{\upchi}^{R}\to\mathbbst{E}_{\upchi}$ exhibits $\mathbbst{E}_{\upchi}$ as the $\infty$-localization (as an $\infty$-operad) of $\Drect_{\upchi}^R$ at the set of unary maps 
	$$
	\Wrect=\left\{
 U\subseteq  V \text{ where both }z( U)\text{ and }z( V)\text{ lie in the same stratum } X_j\backslash X_{j-1} \text{ for some }j
	\right\}.
	$$
\end{thm}

The strategy to show Theorem \ref{thm_LocalizationForTruncatedDisjStratifiedI} parallels that of Theorem \ref{thm_LocalizationForTruncatedDisj}. In fact:
\begin{enumerate}
    \item Reduce the problem about the operad map $\Drect^R_{\upchi}\rightarrow\mathbbst{E}_{\upchi}$ to its closed cousin $\overline{\Drect}\!\,^{R}_{\upchi}\rightarrow \overline{\mathbbst{E}}\!\,_{\upchi}$, which is technically more convenient. Observe that the relevant results in \textsection\ref{subsect:ClosedLittleDiscs}, i.e.\ Lemmas \ref{lem:EnvsClosedEn} and \ref{lem:DRvsClosedDR}, generalize to this more general situation with a similar proof. In essence, the shrinking arguments we use only require inclusions of discs $U\subseteq V$ where both $z(U)$ and $z(V)$ belong to the same stratum in the present case.

    \item The formal arguments in \textsection\ref{subsect:ProofLocalizationTHM} apply almost verbatim to our constructible case (note that now $\overline{\mathbbst{E}}\!\,_{\upchi}$ is a honest colored operad, unlike $\overline{\mathbbst{E}}\!\,_n$). The output is a (homotopy) commutative diagram
    \[
\begin{tikzcd}[ampersand replacement=\&]
\mathsf{weq}\left[[k],\overline{\mathdutchcal{D}}\!\,_{\upchi}\right]_{\upalpha}\ar[d,"(5)"']\ar[rr] \&\& \mathsf{weq}\left[[k],\overline{\mathbbst{E}}\!\,_{\upchi}^{\otimes,\sharp}\right]_\upalpha\ar[d,"(1)"]\\
\widetilde{\mathsf{weq}}\left[[k],\overline{\mathdutchcal{D}}\!\,_{\upchi}\right]_{\upalpha}\ar[d,"(4)"'] \&\& \Drect^0\Conf_{\upalpha}(\mathbbst{D},{\upchi})\ar[d,"(2)"]\\
\Drect^R\Conf_{\upalpha}(\mathbbst{R}^n,{\upchi}) \ar[rr,"(3)"'] \&\& \Drect^0\Conf_{\upalpha}(\mathbbst{R}^n,{\upchi})
\end{tikzcd},
\]
for which we should prove that $(1)-(5)$ are weak homotopy equivalences. Again, $(1)$ and $(5)$ are easy once we define the poset $\widetilde{\mathsf{weq}}\left[[k],\overline{\mathdutchcal{D}}\!\,_{\upchi}\right]_{\upalpha}$ and $\upalpha$-nested configurations for this situation. We will do this right after concluding this sketch of strategy. 

\item The maps $(2)$ and $(3)$ are shown to be homotopy equivalences as in \textsection\ref{subsect:DConfAndEquivalences}. See Lemmas \ref{lem:InflationInducesHomotopyEquivalenceNestedVersionStratifiedI} and \ref{lem:NestedD0ConfigurationsInDiscStratifiedI}. 

\item For the remaining map, $(4)$, we want to apply Lurie--Seifert--van Kampen's Theorem \cite[A.3.1]{lurie_higher_nodate} as in the proof of Theorem \ref{thm_LocalizationForTruncatedDisj}. For that purpose, denote by $\Prect_{\upchi}$ the poset $\widetilde{\mathsf{weq}}\left[[k],\overline{\mathdutchcal{D}}\!\,_{\upchi}\right]_{\upalpha}$ and define the functor $\upzeta\colon \Prect_{\upchi}\to\mathsf{Open}\big(\Drect^R\Conf_{\upalpha}(\mathbbst{R}^n,{\upchi})\big)$ by
\[
\underline{U}^*\longmapsto \upzeta(\underline{ U}^{*})=\Drect^R\Conf_{\upalpha}(\mathbbst{R}^n,{\upchi})\cap\prod_{(r,i_r)}\Emb^{\mathsf{rect}}\big(\overline{\mathbbst{D}}, U_{i_r}^{r}\big)\;.\footnote{This expression of $\upzeta(\underline{U}^*)$ is ambiguous, it is written in this form to draw the analogy with the original proof. For the precise definition of the functor $\upzeta$, see the discussion after Definition \ref{defn:PosetForLSVKwithDefects}.}
\]
We are reduced to check: $(a)$ the space $\upzeta(\underline{U}^*)$ is weakly contractible for any $\underline{U}^*\in \Prect_{\upchi}$, and $(b)$ the subposet $\left\lbrace \underline{ U}^{*}\in\Prect_{\upchi}:\; (f_{i_r})_{i_r}\in\upzeta(\underline{ U}^{*})\right\rbrace\subseteq\Prect_{\upchi}$ is weakly contractible for any $(f_{i_r})_{i_r}\in \Drect^R\Conf_{\upalpha}(\mathbbst{R}^n_{\upchi})$. The one which is not automatic is $(b)$ and it is tackled in Proposition \ref{prop:CofilteredSubposetLinearDefects}. This concludes the proof of Theorem \ref{thm_LocalizationForTruncatedDisjStratifiedI}.
\end{enumerate}

\paragraph{Fattened configurations in the presence of linear defects.} Let us first define ordinary configuration spaces of points in $\mathbbst{R}^n_{\upchi}:\big( X_{0}\subsetneq\dots\subsetneq X_{d}=\mathbbst{R}^n\big)$ and later their fattened variants.

\begin{defn}\label{defn:ConfWithStratif}
Let $\upupsilon\colon I\to \left\{0,\dots,d\right\}$ be an ordered finite set with labels in the set of levels of the linear stratification $\mathbbst{R}^n_{\upchi}$. Define the space of \emph{configurations of $\upupsilon$-points in} $\mathbbst{R}^n_{\upchi}$ as
$$
\Conf_{\upupsilon}(\mathbbst{R}^n_{\upchi})=\prod_{0\leq j\leq d}\Conf_{\vert\upupsilon^{-1}(j)\vert}( X_{j}\backslash X_{j-1}).
$$
That is, the space of configurations of $\vert I\vert$-points in $\mathbbst{R}^n$ where $i\in I$ belongs to the  $\upupsilon(i)^{\text{th}}$-stratum. Adding a decoration $\upvarrho$ to $\upchi$, accounting for a connected component as in Lemma \ref{lem:ObjectsOfLittleUpchiDiscsOperad}, forces configurations of points to lie in that connected component.
\end{defn}

\begin{defn}\label{defn:FattConfWithStratif}
Let $\upupsilon\colon I\to \left\{0,\dots,d\right\}$ be an ordered finite set with labels in the set of levels of the linear stratification $\mathbbst{R}^n_{\upchi}$. We define the space of \emph{configurations of $R$-fattened $\upupsilon$-points in }$\mathbbst{R}^n_{\upchi}$ to be the subspace of rectilinear embeddings of closed discs
	$$
	\Drect^R\Conf_{\upupsilon}(\mathbbst{R}^n_{\upchi})\subseteq\Emb^{\mathsf{rect}}\Big(\bigsqcup_{i\in I}\overline{\mathbbst{D}}_i;\mathbbst{R}^n\Big)
	$$
	spanned by rectilinear embeddings $f\colon\bigsqcup_i\overline{\mathbbst{D}}_i\hookrightarrow \mathbbst{R}^n$ such that: 
	\begin{itemize}
	    \item  $f(\overline{\mathbbst{D}}_i)$ is a closed disc with radius stricly  bigger than $R$ for any $i\in I$, and 
	    \item $z(f(\mathbbst{D}_i))=f(z(\mathbbst{D}_i))$ belongs to $ X_{\upupsilon(i)}\backslash X_{\upupsilon(i)-1}$ and $f(\overline{\mathbbst{D}}_i)\cap X_{\upupsilon(i)-1}=\diameter$.
	\end{itemize}
\end{defn}

As expected, both configuration spaces are homotopy equivalent since we are working on the euclidean space, and this fact is again key to prove Theorem \ref{thm_LocalizationForTruncatedDisjStratifiedI}.
\begin{lem}\label{lem:InflationStratificationInducesHomotopyEquivalence} Let $\upupsilon\colon I\to \left\{0,\dots,d\right\}$ be an ordered finite set with labels in the set of levels of the linear stratification $\mathbbst{R}^n_{\upchi}$. Then, evaluation at centers of discs determines homotopy equivalences 
	\[
        \Drect^R\Conf_{\upupsilon}(\mathbbst{R}^n_{\upchi}) \simeq \Drect^0\Conf_{\upupsilon}(\mathbbst{R}^n_{\upchi}) \simeq \Conf_{\upupsilon}(\mathbbst{R}^n_{\upchi})\,,
    \]
\end{lem}
\begin{proof} 
Analogous to the proof of Lemma \ref{lem:InflationInducesHomotopyEquivalence}. Note that dilations by positive scalars preserve the strata of any linear stratification (see Definition \ref{defn:LinearStratificationsRn}).
\end{proof}

\paragraph{Nested configurations with linear defects.}
To define the space of $\upalpha$\emph{-nested configurations of $R$-fattened points}, denoted $\Drect^R\Conf_{\upalpha}(\mathbbst{R}^n,\upchi)$, we first consider a preliminary notion. Given $\upchi'_{\uppsi}\in \ob(\mathbbst{E}_{\upchi})$, the space $\Drect^R\widetilde{\Conf}_{\upalpha}\big(\mathbbst{R}^{n}_{\upchi};\upchi'_{\uppsi}\big)$ is obtained by slightly modifying Definition \ref{defn_NestedConfigurations}; $\Drect^R\widetilde{\Conf}_{\upalpha}\big(\mathbbst{R}^{n}_{\upchi};\upchi'_{\uppsi}\big)$ is now a subspace of $\Emb^{\mathsf{rect}}(\overline{\mathbbst{D}};\mathbbst{R}^n)\times\prod_{i_r}\Emb^{\mathsf{rect}}(\overline{\mathbbst{D}};\mathbbst{R}^n)$ given by collections of rectilinear embeddings $(e,(f_{i_r})_{i_r})$ subject to the restrictions in Definition \ref{defn_NestedConfigurations} over the components $(f_{i_r})_{i_r}$ plus: 
\begin{itemize}
    \item $f_{i_r}(\overline{\mathbbst{D}})\cap X_{s-1}=\diameter$ if $z(f_{i_r}(\mathbbst{D}))$ belongs to $ X_{s}\backslash X_{s-1}$, and 
    \item $z(e(\mathbbst{D}))\in \uppsi\subseteq X_{t+1}\backslash X_{t}$ and $e(\overline{\mathbbst{D}})\cap X_t=\diameter$, where $t$ accounts for the interval of steps $[0,t]$ to pass from $\upchi$ to $\upchi'$, together with $\bigcup_{i_r}f_{i_r}(\overline{\mathbbst{D}})\subseteq e(\overline{\mathbbst{D}})$.
\end{itemize}

Then, we define $\Drect^R\Conf_{\upalpha}(\mathbbst{R}^n,\upchi)$
by taking together all possible choices as follows
\[
\Drect^R\Conf_{\upalpha}(\mathbbst{R}^n,\upchi):= \coprod_{\underline{\upchi^k_{\uppsi}}}\!\!\prod_{\quad i\in \langle m_k\rangle^{\circ}}\Drect^{R}\widetilde{\Conf}_{\upalpha\vert_i}\big(\mathbbst{R}^n_{\upchi};\upchi^k_{i,\uppsi(i)}\big)\;,
\]
with $\underline{\upchi^k_{\uppsi}}=(\upchi^k_{1,\uppsi(1)},\dots ,\upchi^k_{m_k,\uppsi(m_k)})$ running over $\mathsf{ob}(\mathbbst{E}_{\upchi})^{\times m_k}$ and $\upalpha\vert_i$ denoting the $k$-chain obtained from $\upalpha$ by taking successive preimages of $i\in \langle m_k\rangle^{\circ}$, i.e.\
\[
\upalpha\vert_i=\left[(\upalpha_k\cdots\upalpha_1)^{-1}(\{*,i\})\xrightarrow{\;\; \upalpha_1\vert\;\;}\cdots \xrightarrow{\;\; \upalpha_{k-1}\vert\;\;}\upalpha_k^{-1}(\{*,i\})\xrightarrow{\;\; \upalpha_k\vert\;\;} \{*,i\} \right]\;.
\]

\begin{rem} The first additional restriction on $(f_{i_r})_{i_r}$ accounts for the admissibility of rectilinear embeddings with respect to the linear stratification $\upchi$, while the second one imposes that everything happens somehow in the connected component $\uppsi\subseteq X_{t+1}\backslash X_{t}$. The technical reason why we take disjoint unions and products to define $\Drect^R\Conf_{\upalpha}(\mathbbst{R}^n,\upchi)$ stems from the discussion about $\mathsf{weq}\left[[k],\overline{\mathbbst{E}}\!\,^{\otimes,\sharp}_{\upchi}\right]_{\upalpha}$ below. 
\end{rem}

Similarly to Lemma \ref{lem:InflationStratificationInducesHomotopyEquivalence}, an inspection of the proof of Lemma \ref{lem:InflationInducesHomotopyEquivalenceNestedVersion} yields:
\begin{lem}\label{lem:InflationInducesHomotopyEquivalenceNestedVersionStratifiedI} Let $\upalpha\in \Fun([k],\Fin_*)$ be a $k$-chain of pointed functions. Then, the canonical inclusion of spaces of $\upalpha$-nested configurations 
	\[
        \Drect^R\Conf_{\upalpha}(\mathbbst{R}^n,{\upchi}) \xhookrightarrow{\quad\quad} \Drect^0\Conf_{\upalpha}(\mathbbst{R}^n,{\upchi})
    \]
is a homotopy equivalence.
\end{lem}

Let us now move to the definition of $\Drect^0\Conf_{\upalpha}(\mathbbst{D},\upchi)$. 
The goal is to identify the space $\mathsf{weq}\left[[k],\overline{\mathbbst{E}}\!\,^{\otimes,\sharp}_{\upchi}\right]_{\upalpha}$ in a way that its connection to $\Drect^0\Conf_{\upalpha}(\mathbbst{R}^n_{\upchi})$ is clearer. For example, for $k=1$, we have an identification 
\[
\mathsf{weq}\left[[k],\overline{\mathbbst{E}}\!\,^{\otimes,\sharp}_{\upchi}\right]\cong \coprod_{(\langle m_0\rangle,\underline{\upchi'_{\uprho}}),\;(\langle m_1\rangle,\underline{\upchi''_{\uppsi}})} \Map_{\overline{\mathbbst{E}}\!\,^{\otimes}_{\upchi}} \big((\langle m_0\rangle,\underline{\upchi'_{\uprho}}),(\langle m_1\rangle,\underline{\upchi''_{\uppsi}})\big)\,.
\]
Fixing the pointed function $\upalpha\colon \langle 3\rangle \to \langle 3\rangle$ given by $\upalpha(1)=1=\upalpha(2)$, $\upalpha(3)=2$, it restricts to
\[
\mathsf{weq}\left[[k],\overline{\mathbbst{E}}\!\,^{\otimes,\sharp}_{\upchi}\right]_{\upalpha}\cong \coprod_{\underline{\upchi'_{\uprho}},\;\underline{\upchi''_{\uppsi}}} \overline{\mathbbst{E}}_{\upchi}\brbinom{\{\upchi'_{1,\uprho(1)},\upchi'_{2,\uprho(2)}\}}{\upchi''_{1,\uppsi(1)}}\times \overline{\mathbbst{E}}_{\upchi}\brbinom{\{\upchi'_{3,\uprho(3)}\}}{\upchi''_{2,\uppsi(2)}} \times \overline{\mathbbst{E}}_{\upchi}\brbinom{\diameter}{\upchi''_{3,\uppsi(3)}}\,.
\]
For an arbitrary $k$, the Segal condition yields similar descriptions. Thus, one can observe that it is mandatory to take a coproduct running over all possible targets. For this reason, given a $k$-chain $\upalpha: [\langle m_0\rangle\to \cdots\to \langle m_k\rangle]$, we define
\[
\Drect^0\Conf_{\upalpha}(\mathbbst{D},\upchi):=\coprod_{\underline{\upchi^k_{\uppsi}}}\!\! \prod_{\quad i\in \langle m_k\rangle^{\circ}} \Drect^0\widetilde{\Conf}_{\upalpha\vert_i}\big(\mathbbst{D}_{\upchi^k_{i}}\big)\;,
\]
where the space $\Drect^0\widetilde{\Conf}_{\upalpha\vert_i}\big(\mathbbst{D}_{\upchi^k_{i}}\big)$ is the obvious modification of $\Drect^0\widetilde{\Conf}_{\upalpha\vert_i}\big(\mathbbst{R}_{\upchi^k_{i}}\big)$ with rectilinear embeddings landing in the unit $n$-disc $\mathbbst{D}$ equipped with the linear stratification $\upchi^k_i$. Notice that there is no decoration with respect to connected components; this means that the additional datum of $e\colon \overline{\mathbbst{D}}\hookrightarrow \mathbbst{D}$ is not considered.

To homogenize the arguments in the sequel, let us fix a rectilinear embedding $\mathbbst{D}_{\upchi'}\hookrightarrow \mathbbst{R}^n_{\upchi}$ for any $\upchi'_{\uppsi}\in \ob(\mathbbst{E}_{\upchi})$. We will call these choices \emph{canonical} in the next result, but essentially any choice works equally well. If $\upchi'$ is obtained from $\upchi$ by forgetting the interval of steps $[0,t]$ in the filtration, we choose a point $z(\upchi'_{\uppsi})\in \uppsi\subseteq X_{t+1}\backslash X_{t}\subseteq \mathbbst{R}^n$ whose distance to the subspace $X_t$ is, let us say, $2$. Therefore, we can consider the canonical inclusion of the unit disc around $z(\upchi'_{\uppsi})$. For $\upchi'=\upchi$ (equiv.\ $t=-1$), we take $z(\upchi)=0$. 

\begin{lem}\label{lem:NestedD0ConfigurationsInDiscStratifiedI} Postcomposition with the canonical inclusions $\{\mathbbst{D}_{\upchi'}\hookrightarrow \mathbbst{R}^n_{\upchi}:\;\upchi'_{\uppsi}\in \ob(\mathbbst{E}_{\upchi})\}$ induces a homotopy equivalence
\[
\Drect^0\Conf_{\upalpha}(\mathbbst{D},{\upchi})\xhookrightarrow{\quad\sim \quad}\Drect^0\Conf_{\upalpha}(\mathbbst{R}^n,{\upchi}).
\]
\end{lem}
\begin{proof} The proof is similar to that of Lemma \ref{lem:NestedD0ConfigurationsInDisc}. In essence, the only difference is that, at the component indexed by $\upchi'_{\uppsi}\equiv\upchi^k_{i, \uppsi(i)}$, we have to perform an additional translation to place the $\upalpha\vert_i$-nested configuration in the image of $\mathbbst{D}_{\upchi'}\hookrightarrow \mathbbst{R}^n_{\upchi}$ associated to $\upchi'_{\uppsi}$. More concretely, the homotopy inverse $\Drect^0\widetilde{\Conf}_{\upalpha\vert_i}\big(\mathbbst{R}^n_{\upchi};\upchi'_{\uppsi}\big)\longrightarrow \Drect^0\widetilde{\Conf}_{\upalpha\vert_i}\big(\mathbbst{D}_{\upchi'}\big)$ is given postcomposing each rectilinear embedding in $(e,\underline{f})$ with the rectilinear map 
\[
 \mathbbst{R} ^n\xrightarrow{\;\;\;\uplambda(e,\underline{f})\cdot\;\;\;}\mathbbst{R}^n\xrightarrow{\;\;\;v(e,\underline{f})+\;\;\;}\mathbbst{R}^n,
\]
where
\[
\uplambda(e,\underline{f}):=\frac{1}{\vert\vert e(0)\vert\vert+\vert\vert e(1,0,\dots,0)-e(0)\vert\vert+\upepsilon} \qquad \text{and} \qquad v(e,\underline{f}):= z(\upchi'_{\uppsi})-\uplambda(e,\underline{f})\cdot e(0)
\]
with $\upepsilon>0$ a fixed number. The relevant homotopies are just convex interpolations.
\end{proof}

To close this subsection, let us justify the last step in the proof of Theorem \ref{thm_LocalizationForTruncatedDisjStratifiedI} described above. We start by defining the poset $\Prect_{\upchi}=\widetilde{\mathsf{weq}}\left[[k],\overline{\mathdutchcal{D}}\!\,_{\upchi}\right]_{\upalpha}$. 

\begin{defn}\label{defn:PosetForLSVKwithDefects} We denote by $\Prect_{\upchi}$ the poset whose objects are  tuples $\underline{U}^{*}=\big(\underline{U}^{0},\dots, \underline{U}^{k}\big)$, where each $\underline{U}^{r}=\big(U^{r}_{i_r}\big)_{i_r}$, indexed by $i_r\in \langle m_r\rangle \backslash \upalpha_{r+1}^{-1}(*)$\footnote{By convention, $\upalpha_{k+1}$ is the unique active map $\langle m_k\rangle \to \langle 1\rangle$.} and $0\leq r\leq k$, is again a tuple of euclidean open discs in $\mathbbst{R}^n$ with radii strictly bigger than $R$, subject to:
\begin{itemize}
    \item $\overline{U}\!\,^{r}_{i_r}\cap X_{s-1}=\diameter$ if $z(U^{r}_{i_r})$ belongs to $ X_{s}\backslash X_{s-1}$, 
    \item $\overline{U}\!\,_{i_r}^{r}\cap\overline{U}\!\,_{j_r}^{r}=\diameter$ if $\upalpha_{t}\cdots \upalpha_{r+1}(i_r)=\upalpha_{t}\cdots \upalpha_{r+1}(j_r)\neq *$ for some $t\leq k$,
    \item either $\overline{U}\!\,_{i_{r-1}}^{r-1}\subset U_{j_{r}}^{r}$ or $U^{r-1}_{i_{r-1}}=U^{r}_{j_r}$ if $i_{r-1}\in \upalpha_{r}^{-1}(j_r)$.
\end{itemize}
There is a unique morphism $\underline{U}^*\to \underline{V}^*$ between elements of $\Prect_{\upchi}$ if:
\begin{itemize}
    \item $[U^k_{i_k}]=[V^k_{i_k}]$ as elements in $\ob(\mathbbst{E}_{\upchi})$ for any $i_k\in \langle m_k\rangle^{\circ}$, 
    \item $\overline{U}\!\,^{r}_{i_r}\subset V^{r}_{i_r}$ or $U^{r}_{i_r}= V^{r}_{i_r}$ for any $i_r\in \langle m_r\rangle\backslash\upalpha_{r+1}^{-1}(*) $ and $r<k$, and
    \item $z(U^r_{i_r})$, $z(V^r_{i_r})$ lie in the same stratum $X_{s}\backslash X_{s-1}$ for any pair $(r,i_r)$ with $r<k$.
\end{itemize}
\end{defn}

At this point, it remains to specify which functor $\upzeta$ is employed to apply Lurie--Seifert--van Kampen's theorem. We consider the functor $\upzeta\colon \Prect_{\upchi}\to\mathsf{Open}\big(\Drect^R\Conf_{\upalpha}(\mathbbst{R}^n,{\upchi})\big)$ which sends $\underline{U}^*$ to the subspace of
\[
\prod_{i\in \langle m_k\rangle^{\circ}}\Drect^{R}\widetilde{\Conf}_{\upalpha\vert_i}\big(\mathbbst{R}^n_{\upchi};[U^k_{i}]\big) \xhookrightarrow{\quad\quad} \Drect^R\Conf_{\upalpha}(\mathbbst{R}^n,\upchi)
\]
determined by the points whose factors $(e,\underline{f})$ satisfy: $(a)$ $f_{j_r}(\overline{\mathbbst{D}})\subset U^{r}_{j_r}$, and $(b)$ $z(f_{j_r}(\overline{\mathbbst{D}}))$, $z(U^{r}_{j_r})$ live in the same stratum $X_{s}\backslash X_{s-1}$, for any $j_r\in \langle m_r\rangle\backslash\upalpha_{r+1}^{-1}(*)$ and $0\leq r<k$. In other words, we consider the open subset $\upzeta(\underline{U}^*)$ of the component of $\Drect^R\Conf_{\upalpha}(\mathbbst{R}^n,\upchi)$ indexed by $[\underline{U}^k]\in \ob(\mathbbst{E}_{\upchi})^{\times m_k}$ given by embeddings whose image is appropriately contained in $(\underline{U}^0,\dots, \underline{U}^{k-1})$.

\begin{prop}\label{prop:CofilteredSubposetLinearDefects} Let $p\in \Drect^R\Conf_{\upalpha}(\mathbbst{R}^n,{\upchi})$ be an $\upalpha$-nested configuration. Then, the subposet $\left\lbrace \underline{ U}^{*}\in\Prect_{\upchi}:\; p\in\upzeta(\underline{ U}^{*})\right\rbrace \subseteq \Prect_{\upchi}$ is cofiltered and hence weakly contractible.
\end{prop}
\begin{proof} Rearranging the factors of $p$, we can see the $\upalpha$-nested configuration as a list of rectilinear embeddings
$
\big((e_{i_k})_{i_k\in \langle m_k\rangle^{\circ}},(f_{i_r})_{1\leq i_r\leq m}\big)
$
with $m=m_0-\vert\upalpha^{-1}_{1}(*)\vert+\dots +m_{k-1}-\vert \upalpha^{-1}_{k}(*)\vert$. Since the space $\Drect^R\Conf_{\upalpha}(\mathbbst{R}^n,\upchi)$ is defined as a disjoint union, $p$ lives in one of the components, let us say indexed by $\underline{\upchi'_{\uppsi}}$. This determines $[\underline{U}^k]=\underline{\upchi'}_{\uppsi}$ and the strata where $z(e_{i_k}(\mathbbst{D}))$ lies on. From here, to check that the subposet is cofiltered, we use the same argument we used in the proof of Theorem \ref{thm_LocalizationForTruncatedDisj}. Just note that the constructive proof there handles the components of $\underline{U}^*=(\underline{U}^1,\dots,\underline{U}^k)$ indexed by $r<k$ with respect to the list $(f_{i_r})_{i_r}$ and that types of discs with respect to $\upchi$ (equiv.\ their class in $\ob(\mathbbst{E}_{\upchi})$) are preserved along the construction.
\end{proof}


\subsection{Corner defects}

For field theories defined over half euclidean spaces or other local models for manifolds with corners, i.e.\ $\mathbbst{R}^{p}\times \mathbbst{R}_{\geq 0}^{q}$, we may define defects for which our techniques apply. We will be more sketchy in this part, since the ideas are simple adaptations of those in Subsection \ref{subsect:LinearDefects}. 

Let us fix once and for all in this subsection the space $\mathbbst{R}^{p,q}=\mathbbst{R}^p\times \mathbbst{R}^q_{\geq 0}$.

\begin{defn}\label{defn:CornerStratifications} The \emph{corner stratification of} $\mathbbst{R}^{p,q}$ is the finite  filtration  
$
 Z_{0}\subsetneq \cdots \subsetneq  Z_{q}=\mathbbst{R}^{p,q},  
$
given by the subspaces 
$$
Z_{k}=\left\{(x,y)\in \mathbbst{R}^{p,q}\;\text{ s.t. at least }(q-k)\text{-coordinates of }y\text{ are null}\right\}.$$
\end{defn}

Definitions \ref{defn:TruncatedDiscsWithLinearStratifications}
 and \ref{defn:ConstructiblePrefactRn} can be straightforwardly adapted to this situation, yielding the   \emph{$R$-truncated $(p,q)$-discs operad} $\Drect^{R}_{p,q}$ and the notion of \emph{constructible prefactorization algebras at scale }$R$ over $\mathbbst{R}^{p,q}$. Also, the corresponding little disc operad (see Definition \ref{defn:LittleStratifiedDiscsOperad}), called \emph{little $(p,q)$-discs operad} and denoted  $\mathbbst{E}_{p,q}$, is a generalization of the Swiss cheese operad (corresponding to the case $q=1$); see also \cite[\S3.1.1]{calaque_around_nodate}. For concreteness, let us be more specific about what $\mathbbst{E}_{p,q}$ is.
 
We obtain its set of objects $\ob(\mathbbst{E}_{p,q})$ as the quotient of $\ob(\Drect^{R}_{p,q})$ by the minimal equivalence relation  generated by
$$
 U\underset{\text{pre}}{\sim} V \text{ if } U\subseteq  V \text{ and both }z( U)\text{ and }z( V)\text{ lie in the same stratum } Z_{j}\backslash Z_{j-1}.
$$

\begin{lem} There is a bijection between $\mathsf{ob}(\Drect^{R}_{p,q})/\!\!\sim$ and the set of pairs  $t_{\upvarrho}\equiv(t,\upvarrho)$ where:
\begin{itemize}
    \item  $t$ is a  non-negative integer smaller or equal than $q$, and  
    \item $\upvarrho$ represents a connected component of $ Z_{t}\backslash Z_{t-1}$ for the previous $t$.
\end{itemize} 
\end{lem}
\begin{rem} We should think about $t_{\upvarrho}$ as a $(r,t)$-disc with $r+t=p+q$, i.e.\  $\mathbbst{D}^{r,t}=\mathbbst{D}\cap \mathbbst{R}^{r,t}$ within $\mathbbst{R}^{r+t}$, for each index $\upvarrho$.  We use the notation $\mathbbst{D}^{r,t}_{\upvarrho}$ to refer to this situation.
\end{rem}

\begin{defn}\label{defn:LittleStratifiedDiscsOperadCorner} Let $(r,t)$ and $(s,l)$ be two pairs of non-negative integers satisfying $t,l\leq q$ and $r+t=p+q=s+l$.
\begin{itemize}
    \item A \emph{c-admissible rectilinear embedding} $\mathbbst{D}^{r,t}\hookrightarrow \mathbbst{D}^{s,l}$ is a rectilinear embedding which preserves the induced corner stratifications\footnote{In particular, there is no c-admissible rectilinear embedding if $t>l$.}. We denote by  $\Emb_{\upc}^{\mathsf{rect}}(\mathbbst{D}^{r,t},\mathbbst{D}^{s,l})$ the subspace of c-admissible rectilinear embeddings. For $\upvarrho$ connected component of $ Z_{t}\backslash Z_{t-1} $ and $\uppsi$  connected component of $ Z_{l}\backslash Z_{l-1}$, set
    $$
    \Emb_{\upc}^{\mathsf{rect}}(\mathbbst{D}^{r,t}_{\upvarrho},\mathbbst{D}^{s,l}_{\uppsi})= \left\lbrace
	\begin{array}{cr}
	\Emb_{\upc}^{\text{rect}}(\mathbbst{D}^{r,t},\mathbbst{D}^{s,l}) & \text{ if }\;\uppsi\text{ is contained in }\overline{\upvarrho},\\[5mm]
	\diameter  & \text{ otherwise,}
	\end{array}
	\right.
    $$
    where $\overline{\upvarrho}$ denotes the closure of $\upvarrho$ as a subspace of $\mathbbst{R}^{p,q}$. 

    \item The \emph{little $(p,q)$-discs operad}, denoted $\mathbbst{E}_{p,q}$, is given by:
    \begin{itemize}
        \item[\textbf{Obj:}] its set of objects/colors is $\mathsf{ob}(\mathbbst{E}_{p,q})=\mathsf{ob}(\Drect^{R}_{p,q})\slash\!\!\sim$; 
        \item[\textbf{Mor:}] its spaces of multimorphisms are defined by 
        $$
        \mathbbst{E}_{p,q}\brbinom{\underline{t_{\upvarrho}}}{l_{\uppsi}}= \prod_i \Emb_{\upc}^{\mathsf{rect}}(\mathbbst{D}^{r_i,t_i}_{\upvarrho(i)},\mathbbst{D}^{s,l}_{\uppsi})\cap \Emb\Big(\bigsqcup_i\mathbbst{D}^{r_i,t_i},\mathbbst{D}^{s,l}\Big)
        $$
        with the obvious composition of rectilinear embeddings and identities.
    \end{itemize}
\end{itemize}
\end{defn}

There is a canonical map of operads $\upgamma\colon\Drect_{p,q}^{R}\to \mathbbst{E}_{p,q}$, sending  $\mathbbst{E}_{p,q}$-algebras to constructible prefactorization algebras at scale $R$ over $\mathbbst{R}^{p,q}$. As in the previous subsection, one shows:
\begin{thm}\label{thm:ConstructiblePrefactAlgsScaleRonRpq} 
   Constructible prefactorization algebras at scale $R$ over $\mathbbst{R}^n_{p,q}$ are the same as $\mathbbst{E}_{p,q}$-algebras. More concretely, the functor 
	$$
	\upgamma^*\colon\EuScript{A}\mathsf{lg}_{\mathbbst{E}_{p,q}}( \EuScript{V})\to 
	\begin{Bmatrix}
\textup{constructible }\EuScript{V}\text{-prefactorization}\\
\textup{algebras at scale }R\textup{ over }\mathbbst{R}^{p,q}
\end{Bmatrix}
	$$
	is an equivalence of $\infty$-categories for any sm-$\infty$-category $ \EuScript{V}$.
\end{thm}

\begin{rem}
The interested reader may note that the stratifications isolated by Definition \ref{defn:LinearStratificationsRn} and \ref{defn:CornerStratifications} are just simple choices to illustrate how our main result can be generalized. An important requirement to look for generalizations is that one should consider stratifications which are \emph{stable by inflations}, e.g.\ those filtrations of $\mathbbst{R}^n$ or $\mathbbst{R}^{p,q}$ where $p+q=n$  such that for any dilation $\uplambda\colon \mathbbst{R}^n\to \mathbbst{R}^n$ with $\uplambda>1$, one has $\uplambda( X_j\backslash  X_{j-1})\subseteq  X_j\backslash X_{j-1}$ for any index $j$ of the filtration. This assumption is unavoidable for our purposes and it already implies that the levels of the stratification must contain rays. In particular, the only point that may appear isolated in a stratum is the origin.
\end{rem}

\section{Application: quantization of constant Poisson structures}\label{sect:Quantization}
From now on, we fix a field $\mathbbst{k}$ of characteristic $0$. Let $\Vrect$ be a $\mathbbst{k}$-vector space, or more generally a cochain complex thereof, equipped with an antisymmetric pairing
$$
\langle\text{-},\text{-}\rangle\colon \Uplambda^2(\Vrect)\to\mathbbst{k}.
$$ 
This is equivalent to the data of a \emph{constant Poisson structure} 
on $\Sym(\Vrect)$ determined by 
$$
\lbrace v,w\rbrace:=\langle v,w\rangle\in\mathbbst{k}. 
$$
It is well-known that such a constant Poisson structure admits a quantization given by a kind of ``Weyl algebra''
$$
\mathsf{Weyl}_\hbar(\Vrect_{\langle\text{-},\text{-}\rangle}):= \Trect(\Vrect)\llbracket\hbar\rrbracket/(v\otimes w-(-1)^{\vert v\vert\vert w\vert}w\otimes v-\hbar\langle v,w\rangle),
$$
whose deformed product has no higher terms over linear generators and also the deformation does not modify the underlying $\mathbbst{k}\llbracket\hbar\rrbracket$-module, i.e.\ one has an isomorphism of $\mathbbst{k}\llbracket\hbar\rrbracket$-modules
$$
\mathsf{Weyl}_\hbar(\Vrect_{\langle\text{-},\text{-}\rangle})\cong \Sym(\Vrect)\llbracket\hbar\rrbracket.
$$

We are going to recover this algebra  $\mathsf{Weyl}_\hbar(\Vrect_{\langle\text{-},\text{-}\rangle})$ as global sections of a \emph{locally constant} $\Drect_1^R$-algebra with $R=\frac12$. Our approach shall be seen as a discrete (as opposed to continuous) version of Costello--Gwilliam's discussion of quantum mechanics in \cite[\textsection 4.3]{costello_factorization_2017}.


\subsection{Poisson additivity for constant Poisson structures}

Let us start by recalling Safronov's result \cite{Safronov-additivity} stating an equivalence of $\infty$-categories 
\[
\EuScript{A}\mathsf{lg}_{\mathbbst{P}_1}(\EuScript{M}\mathsf{od}_\mathbbst{k}^\otimes)\simeq \EuScript{A}\mathsf{lg}_{\mathbbst{E}_1}\big(\EuScript{A}\mathsf{lg}_{\mathbbst{P}_0}(\EuScript{M}\mathsf{od}_\mathbbst{k}^\otimes)\big)
\]
that commutes with the forgetful functors to the $\infty$-category $\EuScript{M}\mathsf{od}_\mathbbst{k}$ of cochain complexes of $\mathbbst{k}$-modules. Here $\mathbbst{P}_n$ denotes the algebraic operad encoding commutative algebras equipped with a Poisson bracket of degree $1-n$, that is a $(n-1)$-shifted biderivation satifying the Jacobi identity (also known as a $(n-1)$-shifted Poisson structure\footnote{Here we follow the terminology from \cite{melani,CPTVV,Safronov-additivity}. Beware that this terminology is unfortunately opposite to the one from \cite{costello_factorization_2021,calaque_not_2024}: an $(n-1)$-shifted Poisson structure on a commutative (differential graded) algebra $A$ in the sense of \cite{melani,CPTVV,Safronov-additivity} is a $(1-n)$-shifted Poisson bracket on $A$ in the sense of \cite{costello_factorization_2021,calaque_not_2024}. }). 

\begin{rem}\label{remark-on-units}
In \cite{Safronov-additivity} there are two versions of the above statement: for unital and non-unital Poisson algebras ($\mathbbst{E}_1$-algebras are unital by definition). All along Section \ref{sect:Quantization}, all algebras over operads will be unital by convention. 
\end{rem}

\medskip

In this subsection we aim at giving explicit models realizing Poisson additivity in the case of constant Poisson brackets. 
Recall the usual identifications $\Omega^1_{\Sym(\Vrect)}\overset{}{\cong}\Sym(\Vrect)\otimes\Vrect$ and
\[
\mathsf{BiDer}\big(\Sym(\Vrect)\big):=\Hom_{\Sym(\Vrect)}\big(\Uplambda^2_{\Sym(\Vrect)}(\Upomega^1_{\Sym(\Vrect)}),\Sym(\Vrect)\big)
\overset{}{\cong}
\Hom_{\mathbbst{k}}\big(\Uplambda^2_{\mathbbst{k}}(\Vrect),\Sym(\Vrect)\big)\,.
\]
provided by the Leibniz rule and base change. A \emph{constant} ($0$-shifted) Poisson structure on $\Sym(\Vrect)$ is thus precisely determined by an antisymmetric pairing 
$\langle\text{-},\text{-}\rangle:\Uplambda^2_{\mathbbst{k}}(\Vrect)\to \mathbbst{k}$ as above. 
We also have the $(-1)$-shifted biderivations
\begin{align*}
\mathsf{BiDer}\big(\Sym(\Vrect),-1\big)
& :=\Hom_{\Sym(\Vrect)}\big(\Uplambda^2_{\Sym(\Vrect)}(\Upomega^1_{\Sym(\Vrect)}[-1]),\Sym(\Vrect)\big)[-1] \\
& \overset{}{\cong}
\Hom_{\mathbbst{k}}\big(\Uplambda^2_{\mathbbst{k}}(\Vrect[-1]),\Sym(\Vrect)\big)[-1]
\overset{}{\cong}\Hom_{\mathbbst{k}}\big(\Sym^2(\Vrect),\Sym(\Vrect)\big)[1]
\,.
\end{align*}
A constant $(-1)$-shifted Poisson structure on $\Sym(\Vrect)$ is therefore determined by a degree $+1$ symmetric pairing $\langle\!\langle\text{-},\text{-}\rangle\!\rangle\colon\Sym^2(\Vrect)\to\mathbbst{k}[1]$. 

\medskip

Poisson additivity for constant Poisson structures therefore goes as follows: 
\begin{enumerate}
\item Using that $\Sym^2(\Vrect[1])[-2]\cong \Uplambda^2(\Vrect)$, we get that the degree $0$ antisymmetric pairing $\langle\text{-},\text{-}\rangle$ on $\Vrect$ determines a degree $+2$ symmetric pairing on $\Vrect[1]$. 
\item Now, $\Vrect\cong 0\times^{\uph}_{\Vrect[1]}0$ is an $\mathbbst{E}_1$-algebra in $\EuScript{M}\mathsf{od}_{\mathbbst{k}}^{\oplus}$, on which the pull-back of the degree $+2$ symmetric pairing vanishes up to homotopy in two different ways, leading to a degree $+1$ symmetric pairing that is compatible with the $\mathbbst{E}_1$-algebra structure. 
\end{enumerate}


\subsubsection{The Costello--Gwilliam model}\label{sssection-Costello-Gwilliam-model}

One can compute the homotopy fiber product $\Vrect\cong 0\times^{\uph}_{\Vrect[1]}0$ by considering the resolution 
\[
0\,\tilde\longrightarrow\,\Upomega_0^\bullet\big((0,1),\Vrect\!\big)[1]\,\overset{\mathsf{ev}_1}{\longrightarrow}\, \Vrect[1]\,,
\]
where ``$\Upomega^\bullet_0$'' indicates that $1$-forms are compactly supported while functions vanish near $0$ and are constant near $1$. 
The nullhomotopy for the pull-back of the degree $+2$ symmetric pairing along $\mathsf{ev}_1$ is given by 
\begin{equation}\label{homotopy-bracket}
\upalpha\cdot\upbeta\longmapsto \int_{[0,1]}\langle\upalpha,\upbeta\rangle
\end{equation}
(the homotopy property is Stokes formula). 
Taking the fiber product with zero we get the $\mathbbst{E}_1$-algebra $\Upomega^\bullet_{\upc}\big((0,1),\Vrect\!\big)[1]$, in $\EuScript{M}\mathsf{od}_{\mathbbst{k}}^{\oplus}$, of $\Vrect[1]$-valued compactly supported forms on the open interval $(0,1)$, that carries the degree $+1$ symmetric pairing given by \eqref{homotopy-bracket}. Applying $\Sym$, one thus gets an $\mathbbst{E}_1$-algebra, in $\EuScript{M}\mathsf{od}_{\mathbbst{k}}^{\otimes}$, carrying a compatible (constant) $(-1)$-shifted Poisson structure. 

\medskip

Note that if one wants a version of the above that uses the language of factorization algebras, one shall simply compose with the functor 
$\uppi^*\colon\EuScript{A}\mathsf{lg}_{\mathbbst{E}_1}(\EuScript{V})\to \EuScript{A}\mathsf{lg}_{\Disc(\mathbbst{R})}^{\mathsf{lc}}(\EuScript{V})$. The idea is then to apply some BD-quantization procedure (see \S\ref{ssec-BD} below) in a way that is compatible with the factorization algebra structure. 

\medskip

Our aim is to do this with a smaller/discrete model. 


\subsubsection{The discrete model}

From now, and until the end of Section \ref{sect:Quantization}, we assume for simplicity that $V$ is concentrated in degree $0$ (this just makes the exposition clearer, especially dealing with signs -- we let the reader check that everything still works in general). 

\medskip

Another way to compute the homotopy fiber product $0\times^{\uph}_{\Vrect[1]}0$ is to consider another resolution 
\[
0\xrightarrow{\quad \sim\quad} \mathsf{cone}(\id_{\Vrect})\xrightarrow{\quad\phantom{\sim}\quad} \Vrect[1]\,.
\]
The nullhomotopy for the pull-back of the degree $+2$ symmetric pairing along the obvious map $\mathsf{cone}(\id_{\Vrect})\,\to \, \Vrect[1]$ is given by 
\[
(a_{-1},a_0)\cdot(b_{-1},b_0)\longmapsto \frac12\big(\langle a_{-1},b_0\rangle+\langle b_{-1},a_0\rangle\big)\,.
\]

Then $0\times^{\uph}_{\Vrect[1]}0 \cong \mathsf{cone}(\mathsf{id}_{\Vrect})\times_{\Vrect[1]}\mathsf{cone}(\mathsf{id}_{\Vrect})$ can be identified with the cone of the anti-diagonal map 
\[
Q\colon \Vrect  \longrightarrow  \Vrect^{\oplus 2}\,,\,
v \longmapsto (-v,v)
\]
and the degree $+1$ symmetric pairing is given by 
\[
(a_{-1},a_0,a_0')\cdot(b_{-1},b_0,b_0')\longmapsto \frac12\big(\langle a_{-1},b_0'+b_0\rangle + \langle b_{-1},a_0'+a_0\rangle\big) \,.
\]
 
\medskip

The above can be interpreted as a $\Vrect[1]$-valued discrete de Rham complex with finite support, and thus generalized to define a $\Disc(\mathbbst{R})$-algebra with values in the symmetric monoidal category of complexes equipped with a pairing. 

\begin{defn}\label{defn-V}
For each bounded open interval $(a,b)\subseteq\mathbbst{R}$ we let $\mathbbst{V}(a,b)$ be the cone of
\[
\Map\big((a,b-1)\cap\,\mathbbst{Z},\Vrect\big)\overset{ Q}{\longrightarrow} \Map\big((a,b)\cap\,\mathbbst{Z},\Vrect\big) \]
given by the finite difference function
$(Qg)(x)=g(x-1)-g(x)$. By convention, $g(x):=0$ whenever $x\notin (a,b-1)$, and we set $\mathbbst{V}(\mathbbst{R}):=\underset{r\to \infty}{\mathsf{colim}}\, \mathbbst{V}(-r,r)$.
\end{defn}
Any cochain in $\mathbbst{V}(a,b)$ can be viewed as a function $\mathbbst{Z}\to \Vrect$ with finite support (or the suspension of such a function). 
The addition of functions with disjoint support turns $\mathbbst{V}$ into a $\Disc(\mathbbst{R})$-algebra with values in $\EuScript{M}\mathsf{od}_\mathbbst{k}^\oplus$. 
It is not locally constant whenever $\Vrect\neq0$: the map $0=\mathbbst{V}(0,1)\to\mathbbst{V}(0,2)=\Vrect$ is not a quasi-isomorphism. The  following proposition shows that $\mathbbst{V}$ becomes locally constant if one does not include intervals of size $\leq1$; in other words, $\mathbbst{V}$ defines a locally constant $\Drect_1^{\frac12}$-algebra (and thus an $\mathbbst{E}_1$-algebra) in $\EuScript{M}\mathsf{od}_\mathbbst{k}^\oplus$. 

\begin{prop}\label{Proposition-locally-constant-V}
For every $t\in(a,b)\cap\mathbbst{Z}$, the map $\Vrect\longrightarrow \mathbbst{V}(a,b)$ sending $v$ to $\delta_tv$ is a quasi-isomorphism. 
\end{prop}
\begin{proof}
We just observe that the sequence 
\[
0\longrightarrow\Map\big((a,b-1)\cap\,\mathbbst{Z},\Vrect\big)\overset{ Q}{\longrightarrow} \Map\big((a,b)\cap\,\mathbbst{Z},\Vrect\big)
\overset{\int}{\longrightarrow} \Vrect \longrightarrow0\,,
\]
where $\int f :=\sum_n f(n)$, is exact whenever $(a,b)\,\cap\,\mathbb{Z}\neq\diameter$, and thus induces a quasi-isomorphism $\mathbbst{V}(a,b)\to\Vrect$. The result follows since the map $v\mapsto \delta_t v$ is a section of $\int$. 
\end{proof}
It follows from the proof that different choices of $t\in (a,b)\cap \mathbbst{Z}$ will lead to homotopic quasi-isomorphisms. 

\medskip

For clarity, we use overlined symbols $\overline{g}:\mathbbst{Z}\to V[1]$ to denote the suspension of functions $g:\mathbbst{Z}\to V$ (for which we use plain symbols). 
Using the pairing $\langle\text{-,-}\rangle$ on $\Vrect$, we can define a degree $+1$ symmetric pairing $\langle\!\langle\text{-,-}\rangle\!\rangle$ on $\mathbbst{V}(a,b)$: 
\[
\langle\!\langle\overline{f}_{-1}+f_0\,,\,\overline{g}_{-1}+g_0\rangle\!\rangle:=\frac12\sum_{x\in\mathbbst{Z}}\big(\langle f_{-1}(x-1)+f_{-1}(x),g_0(x)\rangle+\langle g_{-1}(x-1)+g_{-1}(x),f_0(x)\rangle\big)\,.
\]
Note that this bracket is completely determined by its symmetry and the formula 
\[
\langle\!\langle\overline{f}\,,\,g\rangle\!\rangle:=\frac12\sum_{x\in\mathbbst{Z}}\langle f(x-1)+f(x),g(x)\rangle\,.
\]
The compatibility of $\langle\!\langle\text{-,-}\rangle\!\rangle\colon \mathbbst{V}(a,b)^{\otimes 2}\to \mathbbst{k}[1]$ with the differential follows from a simple calculation (recall that the differential of $\overline{f}$ is $Qf$):
\begin{eqnarray*}
\langle\!\langle Qf\,,\,\overline{g}\rangle\!\rangle-\langle\!\langle \overline{f}\,,\,Qg\rangle\!\rangle
&=&
\frac12\sum_{x\in\mathbbst{Z}}\big(\langle g(x-1)+g(x),f(x-1)-f(x)\rangle \\
&&
-\langle f(x-1)+f(x),g(x-1)-g(x)\rangle\big) \\
&=& 
\frac12\sum_{x\in\mathbbst{Z}}\big(
\langle g(x),f(x-1)\rangle
- \langle g(x-1),f(x)\rangle \\
&&
+\langle f(x-1),g(x)\rangle
- \langle f(x),g(x-1)\rangle\big)=0
\,.
\end{eqnarray*}
This upgrades $\mathbbst{V}$ to a locally constant $\Drect_1^{\frac12}$-algebra (and thus an $\mathbbst{E}_1$-algebra) in cochain complexes equipped with a degree $+1$ symmetric pairing. 

Recalling that $\Sym$ defines a symmetric monoidal functor 
$$
\Sym\colon \left\{ \begin{matrix} \text{complexes }
\Wrect\in \EuScript{M}\mathsf{od}_{\mathbbst{k}}^{\oplus} \text{  with}\\
\text{degree +1 symmetric pairing}\\
\Sym^2(\Wrect)\longrightarrow \mathbbst{k}[1]
\end{matrix}\right\} \xrightarrow{\quad\quad\quad\quad} \EuScript{A}\mathsf{lg}_{\mathbbst{P}_0}(\EuScript{M}\mathsf{od}_{\mathbbst{k}}^{\otimes}),
$$
we get that $\Sym(\mathbbst{V})$ is a locally constant $\Drect_1^{\frac12}$-algebra (and thus an $\mathbbst{E}_1$-algebra) in $\mathbbst{P}_0$-algebras. This is our ``discrete'' model for the image of the $\mathbbst{P}_1$-algebra $\Sym(\Vrect)$ through the Poisson additivity equivalence. 

\begin{rem}
Notice that Definition \ref{defn-V} obviously admits a version leading to a $\mathcalpxtx{P}_{\mathbbst{Z}}$-algebra $\mathbbst{V}_{\mathbbst{Z}}$ in $\EuScript{M}\mathsf{od}_\mathbbst{k}^\oplus$, that is locally constant thanks to Proposition \ref{Proposition-locally-constant-V}. 
By definition, the pullback of $\mathbbst{V}_{\mathbbst{Z}}$ along the morphism $\overline{\Drect}\!\,_1^{\frac32}=\overline{\Crect}\!\,_1^{\frac32}\to \mathcalpxtx{P}_{\mathbbst{Z}}$ from the proof of Proposition \ref{prop P vs C} is equal to the pullback of 
$\mathbbst{V}$ along $\overline{\Drect}\!\,_1^{\frac32}\rightarrow\Drect_1^{\frac12}$. 
\end{rem}


\subsection{Deformation quantization}\label{ssec-BD}

Given a ($0$-shifted) Poisson algebra $(A_0,\cdot,\{\text{-},\text{-}\})$, recall from \cite{BFFLS} that a deformation quantization of $A$ is an (unital) associative $\mathbbst{k}\llbracket
\hbar\rrbracket$-algebra such that   
\begin{enumerate}
\item $A$ is a $\hbar$-adically complete flat formal ($1$-parameter) deformation of $(A_0,\cdot)$;  
\item For every $f,g\in A$, 
\[
\left(\frac{f\cdot g-g\cdot f}{\hbar}\right)\vline_{\underset{\hbar=0}{\phantom{T}}}=\left\{\!\!\!\!\phantom{\Big(}f\vert_{\hbar=0},g\vert_{\hbar=0}\phantom{\Big)}\!\!\!\!\right\}\,.
\]
\end{enumerate}
\begin{rem}
The above can be rephrased by saying that $A$ is an algebra over the Beilinson--Drinfeld operad $\mathbbst{BD}_1$ (within the symmetric monoidal category of complete topological $k\llbracket\hbar\rrbracket$-modules) such that (i) the underlying $\mathbbst{k}\llbracket
\hbar\rrbracket$-module of $A$ is flat, and (ii) $A\otimes_{\mathbbst{k}\llbracket
\hbar\rrbracket} \mathbbst{k}\cong A_0$ as $\mathbbst{BD}_1\otimes_{\mathbbst{k}\llbracket
\hbar
\rrbracket} \mathbbst{k}\cong\mathbbst{P}_1$-algebras. See \cite[\S.5.1]{ safronov-coisotropic}.

\end{rem}

The Weyl type algebra $\mathsf{Weyl}_\hbar(\Vrect_{\langle\text{-},\text{-}\rangle})$ introduced in the beginning of this section is a deformation quantization of the algebra $\Sym(\Vrect)$ equipped with the constant Poisson structure induced by the pairing $\langle\text{-},\text{-}\rangle$ (which is given on generators 
$v,w\in \Vrect$ by $\{v,w\}=\langle v,w\rangle$). 


\subsubsection{BD quantization of $(-1)$-shifted Poisson structures}

We now recall another Beilinson--Drinfeld operad, $\mathbbst{BD}_0$, that is relevant for the quantization of $\mathbbst{P}_0$-algebras (see \cite[\S A.3.2]{ costello_factorization_2017}, or \cite[\S5.1]{safronov-coisotropic}): 
\begin{itemize}
\item As a graded $\mathbbst{k}\llbracket
\hbar\rrbracket$-linear operad $\mathbbst{BD}_0=\mathbbst{P}_0\llbracket
\hbar\rrbracket$ (as such it is generated by a degree $0$ binary operation $\cdot$, a degree $0$ constant $\mathbbst{1}$ and a degree $1$ binary operation $\{\text{-},\text{-}\}$); 
\item The differential is given on generating operations by $d(\cdot)=\hbar\{\text{-},\text{-}\}$ and $d(\mathbbst{1})=0$.  
\end{itemize}
In more concrete terms, a (unital) $\mathbbst{BD}_0$-algebra is the following data: 
\begin{itemize}
    \item A cochain complex $(A,d)$ of $\mathbbst{k}\llbracket\hbar\rrbracket$-modules;
	\item A (degree $0$, $\mathbbst{k}\llbracket\hbar\rrbracket$-linear, unital) commutative and associative product $A\otimes A\xrightarrow{\;\;\cdot\;\;} A$;
	\item A ($\mathbbst{k}\llbracket\hbar\rrbracket$-linear) degree $1$ Poisson bracket $A\otimes A\xrightarrow{\;\,\{\text{-,-}\}\,\;} A[1]$ for the product;
	\item For every $f,g\in A$,  
$d(f\cdot g)- d f\cdot g -(-1)^{\vert f\vert} f\cdot d g= \hbar\{f,g\}$. 
\end{itemize}

\begin{defn}\label{definition-BD-quantization}
A \emph{deformation quantization} of a $\mathbbst{P}_0$-algebra $A_0$ is a $\mathbbst{BD}_0$-algebra $A$ (within the symmetric monoidal category of complete topological $k\llbracket\hbar\rrbracket$-modules) that is flat as a $\mathbbst{k}\llbracket\hbar\rrbracket$-module, and such that $A\otimes_{\mathbbst{k}\llbracket
\hbar\rrbracket} \mathbbst{k}\cong A_0$ as a $\mathbbst{BD}_0\otimes_{\mathbbst{k}\llbracket
\hbar\rrbracket} \mathbbst{k}\cong\mathbbst{P}_0$-algebra. 
\end{defn}

A strategy for the quantization of ordinary Poisson structures is as follows: 
\begin{enumerate}
    \item First use Poisson additivity to view an unshifted Poisson algebra as an $\mathbbst{E}_1$-algebra in $\mathbbst{P}_0$-algebras. 
    \item Then quantize it as a $\mathbbst{BD}_0$-algebra in a way that is compatible with the $\mathbbst{E}_1$-action. 
    \item Finally just remember the $\mathbbst{E}_1$-algebra structure: it should provide the quantization of the Poisson structure one started with.
\end{enumerate}
This strategy was successfully applied to the quantization of  symplectic manifolds in \cite{grady_batalin-vilkovisky_2017} (where the relationship with  Fedosov's approach \cite{FedosovDQ} through formal geometry was provided as well). The following conjecture explains why this strategy must always work (recall from Remark \ref{remark-on-units} that all algebras are unital by convention): 
\begin{conj}[BD additivity]
There is an equivalence of $\infty$-categories 
\[
\EuScript{A}\mathsf{lg}_{\mathbbst{BD}_1}(\EuScript{M}\mathsf{od}_{\mathbbst{k}\llbracket\hbar\rrbracket}^\otimes)\simeq \EuScript{A}\mathsf{lg}_{\mathbbst{E}_1}\big(\EuScript{A}\mathsf{lg}_{\mathbbst{BD}_0}(\EuScript{M}\mathsf{od}_{\mathbbst{k}\llbracket\hbar\rrbracket}^\otimes)\big)
\]
that gives back Safronov's Poisson additivity equivalence when tensoring with $\mathbbst{k}$ (i.e.~when taking $\hbar=0$). 
\end{conj}


\subsubsection{Quantization of constant $(-1)$-shifted Poisson structures}

Recall that a constant $(-1)$-shifted Poisson structure on a cochain complex $\Wrect$ is completely determined (thanks to the Leibniz rule) by $\langle\!\langle\text{-},\text{-}\rangle\!\rangle\colon\Sym^2(\Wrect)\to\mathbbst{k}[1]$ a degree $1$ symmetric pairing, which can also be seen as a degree $1$ differential operator of order $2$ with constant coefficients, on $\Sym(\Wrect)$. We call it the \emph{odd laplacian} $\Updelta$ associated with $\langle\!\langle\text{-},\text{-}\rangle\!\rangle$. It terms of formula, on monomials it is given by 
\[
\Updelta(w_1\cdots w_p)=\sum_{1\leq i<j\leq p}(-1)^\upepsilon\langle\!\langle w_i,w_j\rangle\!\rangle w_1\cdots\hat{w_i}\cdots\hat{w_j}\cdots w_p\,,
\]
with $\upepsilon=\vert w_i\vert\underset{1\leq s<i}{\sum}\vert w_s\vert+\vert w_j\vert \underset{1\leq t\neq i<j}{\sum}\vert w_t\vert$. 
\begin{rem}\label{rem:CoordinatesForBVLaplacian}
Let us unravel the previous discussion in coordinates for a finite dimensional $\Wrect$. Assume that $\Wrect$ has coordinates $\{x_i\}_i$ (this is a basis of $\Wrect^*$, whose dual basis of $\Wrect$ is denoted $\{\frac{\partial}{\partial x_i}\}_i$). Then the pairing $\langle\!\langle\text{-},\text{-}\rangle\!\rangle$ is given by the $\mathbbst{k}$-linear combination $\sum_{i\leq j}\upc_{ij}\frac{\partial}{\partial x_{i}}\cdot\frac{\partial}{\partial x_{j}}$. The induced degree $1$ Poisson bracket is 
	$$
	\displaystyle
\{f,g\}=\frac{1}{2}\sum_{i\leq j}\upc_{ij}\left((-1)^{\vert f\vert \vert x_j\vert}\frac{\partial f}{\partial x_i}\frac{\partial g}{\partial x_j}+(-1)^{\vert x_i\vert (\vert x_j\vert+\vert f\vert)}\frac{\partial f}{\partial x_j}\frac{\partial g}{\partial x_i}\right)\,,
	$$
the associated differential operator is $\Updelta=\frac{1}{2}\sum_{i\leq j}\upc_{ij}\frac{\partial^2}{\partial x_i\partial x_j}$, or in other words, 
$$
\Updelta(f)=\frac{1}{2}\sum_{i\leq  j}\upc_{ij}\frac{\partial^2f}{\partial x_i\partial x_j}\,.
$$
Note that $\upc_{ij}$ is non zero only if $|x_i|+|x_j|=-1$. 
\end{rem}
\begin{prop}\label{Prop-quantized-observables}
The tuple $\big(\Sym(\Wrect)\llbracket\hbar\rrbracket,d_{\Wrect}+\hbar\Updelta,\cdot,\{\text{-},\text{-}\}\big)$ defines a deformation quantization of the (differential graded) $\mathbbst{P}_0$-algebra $\big(\Sym(\Wrect),d_{\Wrect},\cdot,\{\text{-},\text{-}\}\big)$ in the sense of Definition \ref{definition-BD-quantization}. 
\end{prop}
\begin{proof}
First of all, the pairing $\langle\!\langle\text{-},\text{-}\rangle\!\rangle$ is cochain map, hence so is $\Updelta$, i.e.\ $\Updelta d_{\Wrect}+d_{\Wrect}\Updelta=0$. 

Then we prove the following: for every $f,g\in \Sym(\Wrect)$, 
\[
\Updelta(f\cdot g)-\Updelta(f)\cdot g-(-1)^{|f|}f\cdot\Updelta(g)=\{f,g\}\,.
\]
Since $\Updelta$ is of order $2$ it is sufficient to prove this equality on generators, where this is obvious: by definition, for $w_1,w_2\in\Wrect$, 
\[
\Updelta(w_1\cdot w_2)=\langle\!\langle w_1,w_2\rangle\!\rangle=\{w_1,w_2\}\,.
\]

Finally, let us prove that $\Updelta$ squares to zero. This follows from the fact that the degree $1$ symmetric pairing is scalar valued. Indeed, since $\Delta^2$ is order four, it is sufficient to check that it vanishes on monomials of order four. For $w_1,w_2,w_3,w_4\in\Wrect$, 
\begin{eqnarray*}
\Updelta^2(w_1\cdot w_2\cdot w_3\cdot w_4)
& = & 
\langle\!\langle w_1,w_2\rangle\!\rangle\langle\!\langle w_3, w_4\rangle\!\rangle+(-1)^{\vert w_2\vert\vert w_3\vert}\langle\!\langle w_1,w_3\rangle\!\rangle\langle\!\langle w_2,w_4\rangle\!\rangle \\
&& 
+(-1)^{\vert w_4\vert(\vert w_2\vert+\vert w_3\vert)}\langle\!\langle w_1,w_4\rangle\!\rangle\langle\!\langle w_2,w_3\rangle\!\rangle
\\
&&
+(-1)^{\vert w_1\vert(\vert w_2\vert+\vert w_3\vert)}\langle\!\langle w_2,w_3\rangle\!\rangle\langle\!\langle w_1,w_4\rangle\!\rangle\\
&&
+(-1)^{\vert w_1\vert\vert w_2\vert+\vert w_4\vert(\vert w_1\vert+\vert w_3\vert)}\langle\!\langle w_2,w_4\rangle\!\rangle\langle\!\langle w_1,w_3\rangle\!\rangle\\
&&
+(-1)^{(\vert w_1\vert+\vert w_2\vert)(\vert w_3\vert+\vert w_4\vert)}\langle\!\langle w_3,w_4\rangle\!\rangle\langle\!\langle w_1,w_2\rangle\!\rangle \,,
\end{eqnarray*}
and one can see that the terms cancel pairwise. For instance, $\langle\!\langle w_1,w_2\rangle\!\rangle\langle\!\langle w_3, w_4\rangle\!\rangle$ is non-zero only if $(|w_1|+|w_2|)(|w_3|+w_4|)$ is odd, so that the first and last terms cancel (we let the reader check the two other cancellations). 
\end{proof}

We denote by 
\begin{itemize}
    \item $\mathcal O^q(\Wrect^*)$ the $\mathbbst{BD}_0$-algebra $\big(\Sym(\Wrect)\llbracket\hbar\rrbracket,d_{\Wrect}+\hbar\Updelta,\cdot,\{\text{-},\text{-}\}\big)$ appearing in Proposition \ref{Prop-quantized-observables};
    \item $\mathcal O^{c\ell}(\Wrect^*)$ the $\mathbbst{P}_0$-algebra $\mathcal O^q(\Wrect^*)\otimes_{\mathbbst{k}\llbracket\hbar\rrbracket}\mathbbst{k}=\big(\Sym(\Wrect),d_{\Wrect},\cdot,\{\text{-},\text{-}\}\big)$. 
\end{itemize} 
\begin{rem}
    Following \cite[\S4.2.6]{costello_factorization_2017}, the $\mathbbst{BD}_0$ algebra $\mathcal O^q(\Wrect^*)$ is in fact the Chevalley--Eilenberg chain complex of a certain Heisenberg Lie algebra. This Heisenberg Lie algebra is obtained as the central extension $\Wrect[-1]\oplus \mathbbst{k}[-1]$ of the abelian Lie algebra $\Wrect[-1]$ associated with the degree $-1$ anti-symmetric pairing on $\Wrect[-1]$ induced by degree $1$ symmetric pairing $\langle\!\langle\text{-},\text{-}\rangle\!\rangle$ on $\Wrect$. The (degree $0$) parameter $\hbar$ is the suspension of the (degree $1$) central generator. 
\end{rem}
We mention three crucial facts (we leave the details to the reader): 
\begin{enumerate}   
\item The assignment $\Wrect\mapsto \mathcal O^q(\Wrect^*)$ is clearly functorial (in the $1$-categorical sense), more precisely it yields a functor from the category of cochain complexes with a degree $1$ symmetric pairing to the category of $\mathbbst{BD}_0$-algebras. 
\item It is also symmetric monoidal. The symmetric monoidal structure on the category of cochain complexes with a pairing is the direct sum (of the underlying complexes, equipped with the sum of the pairings). The symmetric monoidal structure on the category of $\mathbbst{BD}_0$-algebras comes from the fact that the operad $\mathbbst{BD}_0$ is a Hopf operad: the coproduct on generating operations is given by 
\[
\mathbbst{1}\longmapsto \mathbbst{1}\otimes \mathbbst{1},\quad \cdot\longmapsto\cdot\otimes\cdot\quad\text{and}\quad \{\text{-},\text{-}\}\longmapsto\{\text{-},\text{-}\}\otimes\cdot+\cdot\otimes \{\text{-},\text{-}\}\,.
\]
\item It finally preserves quasi-isomorphisms. Indeed, this can be checked after applying $\otimes_{\mathbbst{k}\llbracket\hbar\rrbracket}\mathbbst{k}$, and thus this reduces to the fact that $\Sym$ preserves quasi-isomorphisms (because $\mathbbst{k}$ is a field of characteristic 0). 
\end{enumerate}
As a result, we get a locally constant $\Drect_1^{\frac12}$-algebra (and thus an $\mathbbst{E}_1$-algebra) in $\mathbbst{BD}_0$-algebras  $\mathcal O^q(\mathbbst{V}^*)$, because $\mathbbst{V}$ is a locally constant $\Drect_1^{\frac12}$-algebra in complexes with a pairing. 
\begin{rem}
    The $\Drect_1^{\frac12}$-algebra in $\mathbbst{BD}_0$-algebra $\mathcal O^q(\mathbbst{V}^*)$ is strict. Indeed, our quantization procedure for constant $1$-shifted Poisson structures is purely $1$-categorical. Note finally that even though the algebraic structure on $\mathcal O^q(\mathbbst{V}^*)$ is strict, it is locally constant only in the cohomological sense (unary operations are quasi-isomorphisms, not isomorphisms). 
\end{rem}

\begin{thm}\label{thm: Quantum Observables as Weyl algebra}
There is an equivalence of (locally constant) $\Drect_1^{\frac12}$-algebras 
$$
\mathcal O^q(\mathbbst{V}^*)\simeq \upgamma^*\mathsf{Weyl}_\hbar(\Vrect_{\langle\text{-},\text{-}\rangle}).
$$ 
\end{thm}
We first show that the cohomology of $\mathcal O^{q}(\mathbbst{V}^*)$ is concentrated in degree $0$: 
\begin{lem}
The map 
$\mathcal O^{q}(\mathbbst{V}^*)\longrightarrow \Hrect^0\big(\mathcal O^q(\mathbbst{V}^*)\big)$ is a quasi-isomorphism and,  moreover, 
$\Hrect^0\big(\mathcal O^{q}(\mathbbst{V}^*)(a,b)\big)\cong \Sym(\Vrect)\llbracket\hbar\rrbracket$ as $\mathbbst{k}\llbracket\hbar\rrbracket$-modules. 
\end{lem}
\begin{proof} The idea is to apply the homological perturbation lemma (see e.g.\ \cite{crainic_perturbation_2004}). Let $(a,b)\subseteq \mathbbst{R}$ be an interval of size $>1$. First, applying Proposition \ref{Proposition-locally-constant-V} together with the fact that $\Sym$ preserves quasi-isomorphisms, we obtain a deformation retract
$$
\begin{tikzcd}
\mathcal{O}^{c\ell}(\mathbbst{V}^*)(a,b)= \big(\Sym(\mathbbst{V}(a,b)),Q\big) \arrow[out=183,in=177,loop,looseness=3, "h"] \ar[r, shift left=1,"p"]   & \big(\Sym(\Vrect),0\big)\cong \Hrect^0\mathcal{O}^{c\ell}(\mathbbst{V}^*)(a,b) \ar[l, shift left=1,"i"] 
\end{tikzcd}.
$$
Recall that we are working over a field $\mathbbst{k}$ of characteristic zero. Then, we extend scalars to $\mathbbst{k}\llbracket \hbar\rrbracket$ and perturb the differential on the left hand side with $\hbar\Delta$. Note that this deformation is small since $(1-\hbar \Delta h)$ is invertible over $\mathbbst{k}\llbracket\hbar\rrbracket$, as it is of the form $1+\hbar(\cdots)$. The resulting perturbed ``homotopy equivalence data"
$$
\begin{tikzcd}
\mathcal{O}^{q}(\mathbbst{V}^*)(a,b)= \big(\Sym(\mathbbst{V}(a,b))\llbracket\hbar\rrbracket,Q+\hbar\Delta\big) \arrow[out=183,in=177,loop,looseness=3, "h'"] \ar[r, shift left=1,"p'"]   & \big(\Sym(\Vrect)\llbracket\hbar\rrbracket,b'\big) \ar[l, shift left=1,"i'"] 
\end{tikzcd}
$$
satisfies: $\Sym(\Vrect)\llbracket\hbar\rrbracket$ is concentrated in degree $0$, and $\Sym(\mathbbst{V}(a,b))\llbracket\hbar\rrbracket$ is concentrated in non-positive degrees. Hence, $b'=p(1-\hbar\Delta h)^{-1}\hbar\Delta i$ must vanish, as $\Delta i=0$ by degree reasons.
\end{proof}

\begin{proof}[Proof of Theorem \ref{thm: Quantum Observables as Weyl algebra}]
It remains to prove that there is an equivalence of $\Drect_1^{\frac12}$-algebras 
\[
\Hrect^0\big(\mathcal O^q(\mathbbst{V}^*)\big)\cong \upgamma^*\mathsf{Weyl}_\hbar(\Vrect_{\langle\text{-},\text{-}\rangle})\,.
\]
Note that both are concentrated in degree $0$ and strictly locally constant (meaning that unary operations are isomorphisms). Hence they can be regarded as plain associative algebras: 
\begin{itemize}
\item The Weyl algebra $\mathsf{Weyl}_\hbar(\Vrect_{\langle\text{-},\text{-}\rangle})$ on the one hand; 
\item $\Hrect^0\Big(\mathcal O^q\big(\mathbbst{V}(-1,2)^*\big)\Big)$ on the other hand, where the product is described as follows: 
\[
\begin{tikzcd}[ampersand replacement=\&]
\Hrect^0\Big(\mathcal O^q\big(\mathbbst{V}(-1,2)^*\big)\Big)\otimes \Hrect^0\Big(\mathcal O^q\big(\mathbbst{V}(-1,2)^*\big)\Big)\ar[d, leftarrow, "\cong"',"\text{local constancy}"]\\
\Hrect^0\Big(\mathcal O^q\big(\mathbbst{V}(-1,\frac12)^*\big)\Big)\otimes \Hrect^0\Big(\mathcal O^q\big(\mathbbst{V}(\frac12,2)^*\big)\Big)
\ar[d,"\text{factorization product}"]
\\
\Hrect^0\Big(\mathcal O^q\big(\mathbbst{V}(-1,2)^*\big)\Big)
\end{tikzcd}
\]
\end{itemize}

We then introduce the algebra map 
\[
\upvarphi\colon\Trect(\Vrect)\llbracket\hbar\rrbracket\longrightarrow \Hrect^0\Big(\mathcal O^q\big(\mathbbst{V}(-1,2)^*\big)\Big)
\]
defined by $\Vrect\ni v\longmapsto [\delta_0v]\in \Hrect^0\Big(\mathcal O^q\big(\mathbbst{V}(-1,2)^*\big)\Big)$, and prove that for any $v,w\in\Vrect$, $\upvarphi(v)\upvarphi(w)-\upvarphi(w)\upvarphi(v)=\hbar\langle v,w\rangle$: 
\begin{itemize}
    \item Recall that $\delta_{1}v-\delta_0v=d(\delta_{0}\overline{v})$, where $d=Q+\hbar\Updelta$. 
    \item Therefore 
    \begin{eqnarray*}
    [\delta_0v][\delta_0w]-[\delta_0w][\delta_0v] & = & [\delta_0v][\delta_{1}w]-[\delta_0w][\delta_{1}v]=[\delta_0v\cdot \delta_{1}w-\delta_0w\cdot \delta_{1}v] \\
    & = & \big[\delta_0v\cdot \big(\delta_0w+d(\delta_0\overline{w})\big)-\delta_0w\cdot \big(\delta_0v+d(\delta_0\overline{v})\big)\big] \\
    & = & \big[\delta_0v\cdot d(\delta_0\overline{w})]-\big[\delta_0w\cdot d(\delta_0\overline{v})]
    \end{eqnarray*}
    \item We observe that $d(\delta_0w\cdot\delta_0\overline{v})-\delta_0w\cdot d(\delta_0\overline{v})=\hbar\{\delta_0w,\delta_0\overline{v}\}=\frac{\hbar}2\langle v,w\rangle$. Similarly, $d(\delta_0v\cdot\delta_0\overline{w})-\delta_0v\cdot d(\delta_0\overline{w})=\hbar\{\delta_0v,\delta_0\overline{w}\}=\frac{\hbar}2\langle w,v\rangle$. Hence
    \[
    [\delta_0v][\delta_0w]-[\delta_0w][\delta_0v]=\hbar\langle v,w\rangle\,.
    \]
\end{itemize}
As a consequence, the map $\upvarphi$ factors through $\mathsf{Weyl}_\hbar(\Vrect_{\langle\text{-},\text{-}\rangle})$. 

Finally, it remains to prove that the induced algebra morphism 
\[
\mathsf{Weyl}_\hbar(\Vrect_{\langle\text{-},\text{-}\rangle})\longrightarrow \Hrect^0\Big(\mathcal O^q\big(\mathbbst{V}(-1,2)^*\big)\Big)
\]
is an isomorphism. This can be checked ``mod $\hbar$'' (i.e.~after applying $\otimes_{\mathbbst{k}\llbracket\hbar\rrbracket}\mathbbst{k}$), where the map becomes the identity on $\Sym(\Vrect)$. 
\end{proof}

\appendix

\section{Weak approximations of $\infty$-operads}\label{sect:CounterexampleWeakApprox}
In this brief appendix, we show that weak approximations (see \cite[Definition 4.2.14]{harpaz_little_nodate} or \cite[Definition 2.3.3.6]{lurie_higher_nodate}) are a no go to prove our main Theorem \ref{thm_LocalizationForTruncatedDisj}. This tool is what J.\ Lurie employed to prove that locally constant factorization algebras over $\mathbbst{R}^n$ are $\mathbbst{E}_n$-algebras, see \cite[Theorem 5.4.5.9]{lurie_higher_nodate}. 
  
\begin{thm}\label{thm:LocalizationOfBigDiscs}
		The morphism $\upgamma\colon \Drect_n^{R}\to\mathbbst{E}_n$ is not a weak approximation.
\end{thm}
  \begin{proof}
 While Condition $(ii)$ in \cite[Definition 4.2.14]{harpaz_little_nodate} holds true for $\upgamma\colon \Drect_n^{R}\to\mathbbst{E}_n$, we show that Condition $(i)$ does not.  
 The assumption $(i)$ asserts that for any disc $ Z\in\Drect_n^{\Rrect}$ 
  	$$
  	(\Drect_n^{R,\otimes})^{\text{act}}_{/ Z}\longrightarrow (\mathbbst{E}_n^{\otimes})_{/\langle 1\rangle}^{\text{act}}
  	$$
  must be a weak homotopy equivalence. As in the proof of \cite[Proposition 5.3.4]{harpaz_little_nodate}, we use the commutative triangle
  	$$
  	\begin{tikzcd}[ampersand replacement=\&]
  	(\Drect_{n}^{R,\otimes})_{/ Z}^{\text{act}}\ar[rr]\ar[rd] \&\& (\mathbbst{E}_{n}^{\otimes})_{/\langle 1\rangle}^{\text{act}}\ar[ld]\\
  	\& (\mathbbst{E}_{n}^{\otimes})^{\text{act}}
  	\end{tikzcd}\quad
  	$$ 
  	to reduce this assertion to see if the homotopy fibers of the diagonal maps are weakly homotopy equivalent. Take $\langle m\rangle\in (\mathbbst{E}_{n}^{\otimes})^{\text{act}}$. On the one hand, the homotopy fiber of the right diagonal map is a Kan complex presenting the weak homotopy type of  
  	$
  	\mathbbst{E}_n(m)=\Emb^{\text{rect}}(\mathbbst{D}^{\sqcup m};\mathbbst{D}).
  	$ 
  	On the other hand, 
  	$(\Drect_{n}^{R,\otimes})_{/ Z}^{\text{act}}
  	$ 
  	can be identified with the subposet of open subsets of $ Z\subseteq \mathbbst{R}^n$ which are disjoint unions of finite collections of discs with radii strictly bigger than $R$, i.e.\ $\Prect(Z)=\bigcup_m\Prect_m(Z)\subset \mathsf{Open}(Z)$ where 
  	$$  
  	\Prect_m( Z)=\begin{Bmatrix}
  	 U_1\sqcup\cdots\sqcup  U_m\overset{\text{open}}{\hookrightarrow}  Z \text{ s.t. } U_i=\mathbbst{D}(x_i,\upepsilon)\text{ with }\upepsilon>R \text{ for }1\leq i\leq m 
  	\end{Bmatrix}.
  	$$ 
  	We describe the homotopy fiber over $\langle m\rangle$ of the left diagonal map  as the total space of a Grothendieck construction. More precisely, the Grothendieck construction corresponds to the functor $\Prect_{m}(Z)\to \EuScript{S}\mathsf{pc}$ given by
  	$$
  	( U_1,\dots, U_m)\longmapsto \prod_{j\in\langle m\rangle^{\text{o}}}\Emb^{\text{rect}}(\mathbbst{D}; U_j)
  	$$
  	and consequently, the homotopy fiber is
    $
  	\underset{( U_1,\dots, U_m)}{\hocolim}\prod_{j}\Emb^{\text{rect}}(\mathbbst{D}; U_j), 
  	$
  	where the homotopy colimit runs over $\Prect_m( Z)$. Thus, we are reduced to check that the map
  	$$
  	\underset{( U_1,\dots, U_m)}{\hocolim}\prod_{j\in\langle m\rangle^{\text{o}}}\Emb^{\text{rect}}(\mathbbst{D}; U_j)\longrightarrow \Emb^{\text{rect}}(\mathbbst{D}^{\sqcup m};\mathbbst{D})
  	$$
  	is not a weak equivalence. Taking $ Z$ small enough, $\Prect_{m}( Z)$ is empty (for $m>1$). Hence, the left hand side of the map is the empty space whereas $\Emb^{\text{rect}}(\mathbbst{D}^{\sqcup m};\mathbbst{D})$ has the homotopy type of $\Conf_{m}(\mathbbst{R}^n)$. 
 \end{proof}

\begin{rem} An alternative criterion to detect $\infty$-localizations
of $\infty$-operads can be found in the work of Karlsson--Scheimbauer--Walde, see \cite[Appendix A]{karlsson_assembly_2024}. Unfortunately, it also doesn't help in our case, i.e.\ to check that $\upgamma\colon \Drect^R_n\to \mathbbst{E}_n$ is an $\infty$-localization (by a counterexample similar to the one observed in Theorem \ref{thm:LocalizationOfBigDiscs}).
\end{rem}

\section{Localizations of operator categories}\label{sect:LocalizationQuasioperads}
Given an $\infty$-operad $\mathdutchcal{O}$ and a subcategory $\mathdutchcal{W}$ of its underlying $\infty$-category, it always exists the $\infty$-localization of $\mathdutchcal{O}$ at $\mathdutchcal{W}$ as an $\infty$-operad (we assume that $\mathdutchcal{W}$ contains all identities for notational convenience later on). Let us denote this device by $\upgamma\colon\mathdutchcal{O}\to\Lfrak_{\mathdutchcal{W}}\mathdutchcal{O}$. Working with quasioperads as our preferred presentation of $\infty$-operads \cite[Definition 2.1.1.10]{lurie_higher_nodate}, 
our task is to check that 
$$
\begin{tikzcd}[ampersand replacement=\&]
	\mathdutchcal{O}^{\otimes}\llbracket\mathdutchcal{W}^{-1}_{\bullet}\rrbracket \ar[rd] \ar[rr, "\sim"] \&\& (\Lfrak_{\mathdutchcal{W}}\mathdutchcal{O})^{\otimes}\ar[ld]\\
	\& \Fin_{*}
\end{tikzcd}
$$ 
are equivalent as quasioperads, where $\mathdutchcal{W}_{\bullet}\subseteq \mathdutchcal{O}^{\otimes}$ is the subcategory 
associated to $\mathdutchcal{W} \subseteq \mathdutchcal{O}^{\otimes}_{\langle 1\rangle}$ via the Segal condition (i.e.\ $\mathdutchcal{W}_{\bullet}$ is the union of its fibers over $\Fin_*$ and its fiber over $\langle n\rangle$ corresponds to $\mathdutchcal{W}^{\times n}$ via the equivalence $\mathdutchcal{O}^{\otimes}_{\langle n\rangle}\simeq (\mathdutchcal{O}^{\otimes}_{\langle 1\rangle})^{\times n}$) and $\mathdutchcal{O}^{\otimes}\llbracket\mathdutchcal{W}^{-1}_{\bullet}\rrbracket$ denotes the $\infty$-localization of $\mathdutchcal{O}^{\otimes}$ at $\mathdutchcal{W}_{\bullet}$ as an $\infty$-category. In other words, the functor $\mathdutchcal{O}^{\otimes}\to \mathdutchcal{O}^{\otimes}\llbracket\mathdutchcal{W}_{\bullet}^{-1}\rrbracket$ over $\Fin_*$ exhibits the $\infty$-localization of $\mathdutchcal{O}^{\otimes}$ at $\mathdutchcal{W}$ as a quasioperad. However, the first problem one encounters is that $\mathdutchcal{O}^{\otimes}\llbracket\mathdutchcal{W}_{\bullet}^{-1}\rrbracket\to \Fin_*$ is not yet known to be a quasioperad. We will see that this is the case by an indirect argument.\footnote{We can slightly modify the subcategory $\mathdutchcal{W}_{\bullet}\subseteq \mathdutchcal{O}^{\otimes}$ without changing the $\infty$-localization by considering that cocartesian lifts of isomorphisms in $\Fin_*$ are also part of $\mathdutchcal{W}_{\bullet}$.}

Recall that the forgetful functor $\Urect\colon\sm\!\Catinfty\to \Opdinfty$ from symmetric monoidal $\infty$-categories, with strong monoidal functors between them, to quasioperads admits a left adjoint $\Env\colon \Opdinfty\to \sm\!\Catinfty$. This sm-envelope was constructed by Lurie (see \cite[Subsection 2.2.4]{lurie_higher_nodate}) as follows: given $\mathdutchcal{O}^{\otimes}\to \Fin_*$, its sm-envelope fits into the pullback of $\infty$-categories 
$$
\begin{tikzcd}
	\Env(\mathdutchcal{O}^{\otimes})\ar[r]\ar[d] & \mathdutchcal{O}^{\otimes}\ar[d]\\
	\mathsf{Act}(\Fin_*)\ar[r,"\mathsf{s}"] & \Fin_*
\end{tikzcd},
$$
where
\begin{itemize}
	\item $\mathsf{Act}(\Fin_*)$ denotes the full subcategory of $\Fun([1],\Fin_*)$ spanned by active morphisms, 
	\item  $\mathsf{s}\colon \mathsf{Act}(\Fin_*)\to \Fin_*$ is the restriction of the source map $\Fun([1],\Fin_*)\to \Fin_*$.   
\end{itemize}
We equip $\Env(\mathdutchcal{O}^{\otimes})$ with the morphism $\Env(\mathdutchcal{O}^{\otimes})\to\Fin_*$ given by composing the functor $\Env(\mathdutchcal{O}^{\otimes})\to\mathsf{Act}(\Fin_*)$  with the target projection $\mathsf{t}\colon\mathsf{Act}(\Fin_*)\to \Fin_*$. Consequently, the previous construction $\Env$ can be consider  as a functor $\Env\colon \Catinfty_{/\Fin_*}\to \Catinfty_{/\Fin_*}$. Actually, in \cite[Corollary 2.1.5]{barkan_envelopes_2022}, this generalized construction was refined to yield a left adjoint functor
\begin{equation}\label{eqt:Generalized envelope}
\Env\colon \Catinfty_{/\Fin_*}\longrightarrow \Catinfty_{/\Fin_*}^{\!\!\!\!\!\!\mathsf{act}\text{-}\mathsf{cocart}},
\end{equation}
where $\Catinfty_{/\Fin_*}^{\!\!\!\!\!\!\mathsf{act}\text{-}\mathsf{cocart}}$ denotes the subcategory of $\Catinfty_{/\Fin_*}$ whose objects have all cocartesian lifts over active maps and whose morphisms preserve these.

Neither of the functors in $\Env\colon \Opdinfty\rightleftarrows\sm\!\Catinfty:\!\Urect$ is fully-faithful. Nevertheless, Haugseng-Kock in \cite{haugseng_infty-operads_2024} noticed that taking into account the image of the unique map of quasioperads $\mathdutchcal{O}^{\otimes}\to \Fin_*$ to the terminal quasioperad, and observing that $\Env(\Fin_*)\simeq \Fin^{\amalg}$, one obtains an adjunction 
$$
\begin{tikzcd}
	\Env\colon \Opdinfty\simeq \Opdinfty_{/\Fin_*} \ar[r, shift left=1]\ar[r, shift right=1, leftarrow] & \sm\!\Catinfty_{/\Fin^{\amalg}}:\!\Urect'  
\end{tikzcd}
$$ 
whose left adjoint is fully-faithful. Actually, they described the essential image of $\Env$  as follows: a strong monoidal functor $\EuScript{C}^{\otimes}\to \Fin^{\amalg}$ comes from a quasioperad if the following two conditions are satisfied
\begin{itemize}
	\item[$(i)$] every object in $\EuScript{C}=\EuScript{C}^{\otimes}_{\langle 1\rangle}$ is equivalent to a tensor product of objects that lie over $\underline{1}:=\{1\}$ in $\Fin=\Fin^{\amalg}_{\langle 1\rangle}$,
	
	\item[$(ii)$] the morphism
	$$
	\coprod_{\upphi\in \Fin(\underline{n},\underline{m})}\,\prod_{j\in \underline{m}}\Map_{\EuScript{C}}\Big(\bigotimes_{i\in \upphi^{-1}(j)}x_i,\,y_j\Big)\longrightarrow \Map_{\EuScript{C}}\Big(\bigotimes_{i\in \underline{n}}x_i,\,\bigotimes_{j\in \underline{m}}y_j\Big),
	$$
	given by tensoring maps, is a weak homotopy equivalence for any collection of objects $\{x_i,y_j\}_{i,j}$ in $\EuScript{C}$ lying over $\underline{1}\in\Fin$. In other words, for any $\upphi\in \Fin(\underline{n},\underline{m})$, the map
	$$
	\prod_{j\in\underline{m}}\Map_{\EuScript{C}}\Big(\bigotimes_{i\in \upphi^{-1}(j)}x_i,\,y_j\Big)\longrightarrow \Map_{\EuScript{C}}^{\upphi}\Big(\bigotimes_{i\in \underline{n}}x_i,\,\bigotimes_{j\in \underline{m}}y_j\Big)
	$$
	is a weak homotopy equivalence, where the right-hand side denotes the union of connected components of the mapping space lying over $\upphi$.
\end{itemize}
See \cite{haugseng_infty-operads_2024} for more details. 

The idea is to show that $\Env(\mathdutchcal{O}^{\otimes}\llbracket\mathdutchcal{W}_{\bullet}^{-1}\rrbracket)$ lies in the essential image of the fully-faithful functor $\Env\colon \Opdinfty\hookrightarrow\sm\!\Catinfty_{/\Fin^{\amalg}}$, by relating this object with an $\infty$-localization of the symmetric monoidal $\infty$-category $\Env(\mathdutchcal{O}^{\otimes})$.

\begin{prop}\label{prop:Localization of sm-envelope} Let $\mathdutchcal{O}^{\otimes}\to \Fin_*$ be a quasioperad and $\mathdutchcal{W}_{\bullet}\subset \mathdutchcal{O}^{\otimes}$ be the subcategory generated by  a subcategory of unary operations $\mathdutchcal{W}\subset\mathdutchcal{O}_{\langle 1\rangle }^{\otimes}$ via the Segal condition. Consider the $\infty$-localization $\Lfrak_{\mathdutchcal{W}}\Env(\mathdutchcal{O}^{\otimes}):=\Env(\mathdutchcal{O}^{\otimes})\llbracket\Env(\mathdutchcal{W}_{\bullet})^{-1}\rrbracket$ of the pair $\Env(\mathdutchcal{W}_{\bullet})\subset \Env(\mathdutchcal{O}^{\otimes})$ in $\Catinfty$. Then, the following statements hold:
	\begin{itemize}
		\item[(a)] The functor $\Lfrak_{\mathdutchcal{W}}\Env(\mathdutchcal{O}^{\otimes})\to \Fin_*$ induced by $\Env(\mathdutchcal{O}^{\otimes})\to \Fin_*$ exhibits $\Lfrak_{\mathdutchcal{W}}\Env(\mathdutchcal{O}^{\otimes})_{\langle 1\rangle}$ as a symmetric monoidal $\infty$-category.
		\item[(b)] The functor $\Lfrak_{\mathdutchcal{W}}\Env(\mathdutchcal{O}^{\otimes})\to \Fin^{\amalg}$ induced by $\Env(\mathdutchcal{O}^{\otimes})\to \Fin^{\amalg}$ is a strong monoidal functor and satisfies conditions $(i)\text{-}(ii)$ of Haugseng-Kock's characterization above.
		\item[(c)] The commutative triangle of strong monoidal functors 
		$$
		\begin{tikzcd}[ampersand replacement=\&]
			\Env(\mathdutchcal{O}^{\otimes})\ar[rr]\ar[rd] \&\& \Lfrak_{\mathdutchcal{W}} \Env(\mathdutchcal{O}^{\otimes}) \ar[ld]\\
			\& \Fin^{\amalg}
		\end{tikzcd}
		$$
		induces an equivalence of $\infty$-categories 
		$$
		\Fun^{\otimes}(\Lfrak_{\mathdutchcal{W}}\Env(\mathdutchcal{O}^{\otimes}), \EuScript{V}^{\otimes})\xrightarrow{\quad\sim\quad} \EuScript{A}\mathsf{lg}_{\mathdutchcal{O},\mathdutchcal{W}}(\EuScript{V}^{\otimes})
		$$
		for any symmetric monoidal $\infty$-category $\EuScript{V}^{\otimes}\to \Fin_*$, where $\EuScript{A}\mathsf{lg}_{\mathdutchcal{O},\mathdutchcal{W}}(\EuScript{V}^{\otimes})$ denotes the full subcategory of $\EuScript{A}\mathsf{lg}_{\mathdutchcal{O}}(\EuScript{V}^{\otimes})$ spanned by algebras which send $\mathdutchcal{W}$ to equivalences.
	\end{itemize}
\end{prop}
\begin{proof}
	\begin{itemize}
		\item[\textit{(a)}] We use \cite[Proposition 3.2.2]{hinich_dwyer-kan_2016} or \cite[Proposition A.5]{nikolaus_topological_2018}; morally speaking, the localization of a sm-$\infty$-category is a sm-$\infty$-category provided the tensor product preserves weak equivalences separately in both variables. In our case, the tensor product is formal $\otimes\colon \Env(\mathdutchcal{O}^{\otimes})_{\langle 1\rangle}^{\times 2}\to \Env(\mathdutchcal{O}^{\otimes})_{\langle 1\rangle}$, i.e.\ it takes two formal tensor products of objects and produces the total formal tensor product of all the objects, and $\Env(\mathdutchcal{W}_{\bullet})$ is by definition closed under formal tensor products. 
		\item[\textit{(b)}] By \cite[Proposition A.5 (v)]{nikolaus_topological_2018} we have such a strong monoidal functor. Also, it is clear that condition $(i)$ holds since it holds for $\Env(\mathdutchcal{O}^{\otimes})\to \Fin^{\amalg}$; e.g.\ one can take the localization to be the identity on objects. For condition $(ii)$, we have to do more work. The idea is to exploit Mazel-Gee's description of mapping spaces in the $\infty$-localization of a relative $\infty$-category $(\EuScript{D},\EuScript{J})$ (see \cite[Section 2]{mazel-gee_hammocks_2018} for notation and definitions)
		$$
		\Map_{\EuScript{D}\llbracket\EuScript{J}^{-1}\rrbracket}(x,y)\simeq \left\|\underset{\underline{\textbf{z}}\in\EuScript{Z}^{\op}}{\hocolim}\,\Nrect_{\infty}\big(\Fun_{**}(\underline{\textbf{z}},\EuScript{D})^{\EuScript{J}}\big)\right\|\equiv \left\|\underset{\underline{\textbf{z}}\in\EuScript{Z}^{\op}}{\hocolim}\,\Nrect_{\infty}\big(\underline{\textbf{z}}_{(\EuScript{D},\EuScript{J})}(x,y)\big)\right\|,
		$$
		where the homotopy colimit (in simplicial spaces) runs over $\EuScript{Z}^{\op}$, the opposite of the category of zigzag types, and the vertical lines are meant to denote (homotopy invariant) geometric realization.
		
		First recall that Haugseng-Kock's conditions for $\EuScript{C}^{\otimes}\to \Fin^{\amalg}$ are the same as asserting that the counit transformation $\Env(\Urect'(\EuScript{C}^{\otimes}))\to\EuScript{C}^{\otimes}$ is an equivalence of sm-$\infty$-categories, or equiv.\ that $\Env(\Urect'(\EuScript{C}^{\otimes}))_{\langle 1\rangle}\to\EuScript{C}^{\otimes}_{\langle 1\rangle}$ is an equivalence of $\infty$-categories (see \cite{haugseng_infty-operads_2024}). Applied to $\Env(\mathdutchcal{O}^{\otimes})$, we actually obtain that the counit transformation yields an equivalence of relative $\infty$-categories, when equipping $\Env(\mathdutchcal{O}^{\otimes})_{\langle 1\rangle}$ with $\Env(\mathdutchcal{W}_{\bullet})_{\langle 1\rangle}$. Let us use the notation
		$$
		\big(\Qfrak(\mathdutchcal{O}^{\otimes}),\Qfrak(\mathdutchcal{W}_{\bullet})\big)\xrightarrow{\quad \sim \quad}  \big( \Env(\mathdutchcal{O}^{\otimes})_{\langle 1\rangle},\Env(\mathdutchcal{W}_{\bullet})_{\langle 1\rangle}\big)
		$$
		for this equivalence of relative $\infty$-categories; i.e.\  $\Qfrak(\mathdutchcal{O}^{\otimes})=\big(\Env\circ\Urect'\circ\Env(\mathdutchcal{O}^{\otimes})\big)_{\langle 1\rangle}$. Again borrowing notation from \cite{mazel-gee_hammocks_2018}, for any pair of objects $(x_i)_{i\in \underline{n}}$, $(y_j)_{j\in\underline{m}}\in \Qfrak(\mathdutchcal{O}^{\otimes})$ and any zigzag type $\underline{\textbf{z}}$, one gets an equivalence of $\infty$-categories
		$$
		\underline{\textbf{z}}_{\Qfrak(\mathdutchcal{O}^{\otimes})}\big((x_i)_{i\in \underline{n}},(y_j)_{j\in\underline{m}} \big) \xrightarrow{\quad\sim\quad} \underline{\textbf{z}}_{\Env(\mathdutchcal{O}^{\otimes})}\big((x_i)_{i\in \underline{n}},(y_j)_{j\in\underline{m}} \big).
		$$
	
		Restricting to zigzags lying over a fixed $\upphi\in \Fin(\underline{n},\underline{m})$ and by definition of $(\Qfrak(\mathdutchcal{O}^{\otimes}),\Qfrak(\mathdutchcal{W}_{\bullet}))$ (see the proof of \cite[Proposition 2.4.6]{haugseng_infty-operads_2024}), we have an equivalence
		$$
		\underline{\textbf{z}}_{\Qfrak(\mathdutchcal{O}^{\otimes})}\big((x_i)_{i\in \underline{n}},(y_j)_{j\in\underline{m}} \big)^{\upphi} \simeq \prod_{j\in \underline{m}} \underline{\textbf{z}}_{\Qfrak(\mathdutchcal{O}^{\otimes})}\big((x_i)_{i\in \upphi^{-1}(j)},y_j \big),
		$$
		because such decomposition holds in $\Qfrak(\mathdutchcal{O}^{\otimes})$ and arrows in $\Qfrak(\mathdutchcal{W}_{\bullet})$ preserve it. Actually, maps $\underline{\textbf{z}}\to \underline{\textbf{z}}'$ of zigzag types, i.e.\ maps in $\EuScript{Z}$, induce functors which are compatible with these decompositions. See Figure \ref{fig:Decomposition of zigzags in localization} for an instance of this phenomenon.
		\begin{figure}[htp]
			\centering
		$$
		\begin{tikzcd}[ampersand replacement=\&]
		(x_i)_{i\in \underline{6}} \ar[r, "	\left(\begin{matrix}
			w_1\vert^{\{1\}}\\
			w_2\vert^{\{2\}}\\
		\vdots\\
			w_6\vert^{\{6\}}
		\end{matrix}\right)", leftarrow] \&[6mm] (x'_i)_{i\in \underline{6}} \ar[r,"\left(\begin{matrix}
		f_1\vert^{\{4\}}\\
		f_2\vert^{\{5,6\}}\\
		f_3\vert^{\{1,2,3\}}
	\end{matrix}\right)"] \&[6mm] (z_k)_{k\in\underline{3}} \ar[r,"\left(\begin{matrix}
	g_1\vert^{\{1,3\}}\\
	g_2\vert^{\{2\}}
\end{matrix}\right)"] \&[6mm] (y'_j)_{j\in \underline{2}} \ar[r,"\left(\begin{matrix}
v_1\vert^{\{1\}}\\
v_2\vert^{\{2\}}
\end{matrix}\right)", leftarrow] \&[6mm] (y_j)_{j\in \underline{2}}\\ 
\&\& \vert\vert \\
		(x_i)_{i\in \underline{4}} \ar[r, "	\left(\begin{matrix}
	w_1\vert^{\{1\}}\\
	\vdots\\
	w_4\vert^{\{4\}}
\end{matrix}\right)", leftarrow] \&[6mm] (x'_i)_{i\in\underline{4}} \ar[r,"\left(\begin{matrix}
	f_1\vert^{\{4\}}\\
	f_3\vert^{\{1,2,3\}}
\end{matrix}\right)"] \&[6mm] (z_k)_{k\in\{1,3\}} \ar[r,"g_1\vert^{\{1,3\}}"] \&[6mm] y'_1 \ar[r,"
	v_1\vert^{\{1\}}", leftarrow] \&[6mm] y_1\\ 
\&\& \times \\
		(x_i)_{i\in\{5,6\}} \ar[r, "	\left(\begin{matrix}
	w_5\vert^{\{5\}}\\
	w_6\vert^{\{6\}}
\end{matrix}\right)", leftarrow] \&[6mm] (x'_i)_{i\in\{5,6\}} \ar[r,"
	f_2\vert^{\{5,6\}}"] \&[6mm] z_2 \ar[r,"
	g_2\vert^{\{2\}}"] \&[6mm] y'_2 \ar[r,"v_2\vert^{\{2\}}", leftarrow] \&[6mm] y_2\\ 
		\end{tikzcd}
		$$
			\caption{A depiction of an object in $\underline{\textbf{z}}_{\Qfrak(\mathdutchcal{O}^{\otimes})}\big((x_i)_{i\in \underline{6}},(y_j)_{j\in \underline{2}}\big)$ for $\underline{\textbf{z}}=[-1;2;-1]$ and the corresponding factors in the decomposition.  More concretely, this object lies over the map $\upphi\colon \underline{6}\to \underline{2}$ characterized by $\upphi^{-1}(1)=\{1,2,3,4\}$ and $\upphi^{-1}(2)=\{5,6\}$. We have  employed a notation which makes explicit the source and target of each component, e.g.\ 
		    $f_3\vert^{\{1,2,3\}}$ is a morphism $(x_1',x_2',x_3')\to z_3$ in $\Qfrak(\mathdutchcal{O}^{\otimes})$.
			}\label{fig:Decomposition of zigzags in localization}
		\end{figure}
		
		Now use: $(1)$ $\Nrect_{\infty}$ preserves (homotopy) coproducts and homotopy limits, and $(2)$ the previous equivalence for zigzags survives after taking homotopy colimits over $\EuScript{Z}^{\op}$ as it does in the mapping spaces of $\EuScript{D}\llbracket\EuScript{J}^{-1}\rrbracket\times \EuScript{D'}\llbracket \EuScript{J}^{'-1}\rrbracket\simeq (\EuScript{D}\times \EuScript{D}')\llbracket(\EuScript{J}\times \EuScript{J}')^{-1}\rrbracket$; to obtain for any $\upphi\in \Fin(\underline{n},\underline{m})$
		$$
		\begin{tikzcd}[ampersand replacement=\&]
			\prod_{j} \Map_{\Lfrak_{\mathdutchcal{W}}\Env(\mathdutchcal{O}^{\otimes})}\big(\bigotimes_{i\in \upphi^{-1}(j)}x_i,y_j\big)\ar[rr] \&[-4.5cm]\&[-4.5cm] \Map_{\Lfrak_{\mathdutchcal{W}}\Env(\mathdutchcal{O}^{\otimes})}^{\upphi}\big(\bigotimes_{i}x_i,\bigotimes_{j}y_j\big)\\
			\&\& \left\|\underset{\underline{\textbf{z}}\in\EuScript{Z}^{\op}}{\hocolim}\,\Nrect_{\infty}\left(\underline{\textbf{z}}_{\Env(\mathdutchcal{O}^{\otimes})}\big(\bigotimes_i x_i,\bigotimes_j y_j\big)^{\upphi}\right)\right\| \ar[u,"\wr"'] \\
			\&  \left\|\underset{\underline{\textbf{z}}\in\EuScript{Z}^{\op}}{\hocolim}\,\Nrect_{\infty}\left(\prod_{j}\underline{\textbf{z}}_{\Qfrak(\mathdutchcal{O}^{\otimes})}\big((x_i)_{i\in\upphi^{-1}(j)},y_j\big)\right)\right\| \ar[ru, "\sim"' sloped,controls={+(5,0) and +(0.2,-0.7)}]\ar[luu,"\sim"' sloped, bend left=30]
		\end{tikzcd}.
		$$
		Hence, we have seen that $\Lfrak_{\mathdutchcal{W}}\Env(\mathdutchcal{O}^{\otimes})\to \Fin^{\amalg}$ satisfies condition $(ii)$ in Haugseng-Kock's characterization.
		
		\item[\textit{(c)}] Again by \cite[Proposition A.5 (v)]{nikolaus_topological_2018}, the functor $\Env(\mathdutchcal{O}^{\otimes})\to \Lfrak_{\mathdutchcal{W}}\Env(\mathdutchcal{O}^{\otimes})$ identifies $\Fun^{\otimes}(\Lfrak_{\mathdutchcal{W}}\Env(\mathdutchcal{O}^{\otimes}),\EuScript{V}^{\otimes})$ with the full subcategory of $\Fun^{\otimes}(\Env(\mathdutchcal{O}^{\otimes}),\EuScript{V}^{\otimes})$ spanned by strong monoidal functors that send $\Env(\mathdutchcal{W}_{\bullet})$ to equivalences. Via the equivalence
		$$
		\Fun^{\otimes}(\Env(\mathdutchcal{O}^{\otimes}),\EuScript{V}^{\otimes})\xrightarrow{\quad\sim\quad} \EuScript{A}\mathsf{lg}_{\mathdutchcal{O}}(\EuScript{V}^{\otimes}),
		$$
		induced by the adjunction $\Env\dashv \Urect$, the previous subcategory corresponds to the full subcategory $\EuScript{A}\mathsf{lg}_{\mathdutchcal{O},\mathdutchcal{W}}(\EuScript{V}^{\otimes})$ of $\EuScript{A}\mathsf{lg}_{\mathdutchcal{O}}(\EuScript{V}^{\otimes})$ spanned by algebras which send $\mathdutchcal{W}$ to equivalences. 
	\end{itemize}
\end{proof}

\begin{prop}\label{prop:Coincidence of envelopes} With the notation of Proposition \ref{prop:Localization of sm-envelope}, there is a canonical equivalence of $\infty$-categories over $\Fin^{\amalg}$
	$$
	\Lfrak_{\mathdutchcal{W}}\Env(\mathdutchcal{O}^{\otimes})\xrightarrow{\quad\sim\quad}\Env(\mathdutchcal{O}^{\otimes}\llbracket\mathdutchcal{W}_{\bullet}^{-1}\rrbracket)
	$$
	which preserves the cocartesian lifts over active maps in $\Fin_*$ and the subcategory of arrows over inert maps in $\Fin_*$ which come from inert maps in $\mathdutchcal{O}^{\otimes}$. 
\end{prop}
\begin{proof} First note that the functor $\Env(\mathdutchcal{O}^{\otimes})\to \Env(\mathdutchcal{O}^{\otimes}\llbracket\mathdutchcal{W}_{\bullet}\rrbracket)$ over $\Fin^{\amalg}$ induces a canonical functor as stated. In fact, it is compatible with the specified classes of maps because: $(i)$  cocartesian lifts over active maps on $\Env(\mathdutchcal{O}^{\otimes}\llbracket\mathdutchcal{W}_{\bullet}\rrbracket)$ are created by construction of $\Env$, and $(ii)$ the marking by arrows over inerts on $\Env(\mathdutchcal{O}^{\otimes}\llbracket\mathdutchcal{W}_{\bullet}\rrbracket)$ comes from the analogous marking on $\mathdutchcal{O}^{\otimes}\llbracket\mathdutchcal{W}_{\bullet}\rrbracket$, which is the image of the inert maps in $\mathdutchcal{O}^{\otimes}$ via the functor $\mathdutchcal{O}^{\otimes}\to \mathdutchcal{O}^{\otimes}\llbracket\mathdutchcal{W}_{\bullet}\rrbracket$. Also, observe that the induced functor $\Lfrak_{\mathdutchcal{W}}\Env(\mathdutchcal{O}^{\otimes})\to \Env(\mathdutchcal{O}^{\otimes}\llbracket\mathdutchcal{W}_{\bullet}\rrbracket)$ is essentially surjective. Thus, to conclude the claim, we are reduced to show that the map
	$$
	\Map_{\Lfrak_{\mathdutchcal{W}}\Env(\mathdutchcal{O}^{\otimes})}\big((X,\upalpha),(Y,\upbeta)\big)\longrightarrow \Map_{\Env(\mathdutchcal{O}^{\otimes}\llbracket\mathdutchcal{W}_{\bullet}\rrbracket)}\big((X,\upalpha),(Y,\upbeta)\big)
	$$
	is a weak homotopy equivalence for any pair of objects $(X,\upalpha),(Y,\upbeta)\in \Env(\mathdutchcal{O}^{\otimes})$. As in the proof of statement \textit{(b)} in Proposition \ref{prop:Localization of sm-envelope}, we use Mazel-Gee's machinery to describe mapping spaces of $\infty$-localizations \cite{mazel-gee_hammocks_2018} for this task.
	
	On the one hand, since $\Env(\mathdutchcal{O}^{\otimes}\llbracket\mathdutchcal{W}_{\bullet}\rrbracket)$ is given by a pullback of $\infty$-categories, 
	\begin{align*}
		\Map_{\Env(\mathdutchcal{O}^{\otimes}\llbracket\mathdutchcal{W}_{\bullet}\rrbracket)}\big((X,\upalpha),(Y,\upbeta)\big) &\simeq \Map_{\mathdutchcal{O}^{\otimes}\llbracket \mathdutchcal{W}_{\bullet}\rrbracket}\big(X,Y\big) \overset{\uph}{\underset{\Fin_{*}(\mathsf{s}\upalpha,\mathsf{s}\upbeta)}{\times}} \mathsf{Act}(\Fin_*)(\upalpha,\upbeta)\\[4mm]
		&\simeq \left\|\underset{\underline{\textbf{z}}\in\EuScript{Z}^{\op}}{\hocolim}\,\Nrect_{\infty}\big(\underline{\textbf{z}}_{\mathdutchcal{O}^{\otimes}}\big(X,Y\big)\big)\right\|    \overset{\uph}{\underset{\Fin_{*}(\mathsf{s}\upalpha,\mathsf{s}\upbeta)}{\times}} \mathsf{Act}(\Fin_*)(\upalpha,\upbeta)\\[4mm]
		&\simeq  \left\|\underset{\underline{\textbf{z}}\in\EuScript{Z}^{\op}}{\hocolim}\,\left[\Nrect_{\infty}\left(\underline{\textbf{z}}_{\mathdutchcal{O}^{\otimes}}\big(X,Y\big) \overset{\uph}{\underset{\Fin_{*}(\mathsf{s}\upalpha,\mathsf{s}\upbeta)}{\times}} \mathsf{Act}(\Fin_*)(\upalpha,\upbeta)\right)\right]\right\|  .
	\end{align*}
	We have applied: $(1)$ homotopy colimits in $\EuScript{S}\mathsf{pc}$ are universal, $(2)$ maps between zigzag types $\underline{\textbf{z}}\to \underline{\textbf{z}}'$ produce functors $\underline{\textbf{z}}'_{\mathdutchcal{O}^{\otimes}}\big(X,Y\big)\to \underline{\textbf{z}}_{\mathdutchcal{O}^{\otimes}}\big(X,Y\big)$ compatible with their projections into $\Fin_*\big(\uppi(X),\uppi(Y)\big)$ (because arrows in $\mathdutchcal{W}_{\bullet}$ lie over identities in $\Fin_*$), and $(3)$ $\Nrect_{\infty}$ commutes with homotopy limits.
	
	On the other hand,
	\begin{align*}
		\Map_{\Lfrak_{\mathdutchcal{W}}\Env(\mathdutchcal{O}^{\otimes})}\big((X,\upalpha),(Y,\upbeta)\big) 
		&\simeq \left\|\underset{\underline{\textbf{z}}\in\EuScript{Z}^{\op}}{\hocolim}\,\Nrect_{\infty}\big(\underline{\textbf{z}}_{\Env(\mathdutchcal{O}^{\otimes})}\big((X,\upalpha),(Y,\upbeta)\big) \big)\right\|.
	\end{align*}
	Therefore, it suffices to prove that the canonical functor
	$$
	\underline{\textbf{z}}_{\Env(\mathdutchcal{O}^{\otimes})}\big((X,\upalpha),(Y,\upbeta)\big) \xrightarrow{\quad \phantom{\sim}\quad} \underline{\textbf{z}}_{\mathdutchcal{O}^{\otimes}}\big(X,Y\big) \overset{\uph}{\underset{\Fin_{*}(\mathsf{s}\upalpha,\mathsf{s}\upbeta)}{\times}} \mathsf{Act}(\Fin_*)(\upalpha,\upbeta)
	$$
	is an equivalence of $\infty$-categories for any zigzag type $\underline{\textbf{z}}$. This follows from the fact that 
	$$
	\begin{tikzcd}[ampersand replacement=\&]
		\big(\Env(\mathdutchcal{O}^{\otimes}),\Env(\mathdutchcal{W}_{\bullet})\big)\ar[r]\ar[d] \& \big(\mathdutchcal{O}^{\otimes},\mathdutchcal{W}_{\bullet}\big)\ar[d] \\
		\big(\mathsf{Act}(\Fin_*), \{\text{identities}\}\big)\ar[r,"\mathsf{s}"] \& \big(\Fin_*, \{\text{identities}\}\big)
	\end{tikzcd}
	$$
	is a pullback of relative $\infty$-categories and since $\underline{\textbf{z}}_{\Env(\mathdutchcal{O}^{\otimes})}\big((X,\upalpha),(Y,\upbeta)\big)$ is just a shortcut notation for $\Fun_{**}(\underline{\textbf{z}}, \Env(\mathdutchcal{O}^{\otimes}))^{\Env(\mathdutchcal{W}_{\bullet})}$, which is part of a closed $\EuScript{R}\mathsf{el}\!\Catinfty$-enrichment of bipointed relative $\infty$-categories (see \cite[Notation 2.2]{mazel-gee_hammocks_2018}).    
\end{proof}

\begin{cor}\label{cor:InfinityLocalizationOfOperatorCats} The canonical functor $\mathdutchcal{O}^{\otimes}\llbracket\mathdutchcal{W}_{\bullet}^{-1}\rrbracket \to \Fin_*$ is a quasioperad and the map of quasioperads $\mathdutchcal{O}^{\otimes}\to\mathdutchcal{O}^{\otimes}\llbracket\mathdutchcal{W}_{\bullet}^{-1}\rrbracket$ exhibits the $\infty$-localization of $\mathdutchcal{O}^{\otimes}$ at $\mathdutchcal{W}$ as a quasioperad. In particular, restriction along $\mathdutchcal{O}^{\otimes}\to\mathdutchcal{O}^{\otimes}\llbracket\mathdutchcal{W}_{\bullet}^{-1}\rrbracket$ induces an equivalence of $\infty$-categories
	$$
	\EuScript{A}\mathsf{lg}_{\mathdutchcal{O}^{\otimes}\llbracket\mathdutchcal{W}_{\bullet}^{-1}\rrbracket}(\EuScript{V}^{\otimes})\xrightarrow{\,\,\sim\,\,}	\EuScript{A}\mathsf{lg}_{\mathdutchcal{O},\mathdutchcal{W}}(\EuScript{V}^{\otimes}),
	$$
	where $\EuScript{A}\mathsf{lg}_{\mathdutchcal{O},\mathdutchcal{W}}(\EuScript{V}^{\otimes})$ denotes the full subcategory of $\EuScript{A}\mathsf{lg}_{\mathdutchcal{O}}(\EuScript{V}^{\otimes})$ spanned by $\mathdutchcal{O}$-algebras which send $\mathdutchcal{W}$ to equivalences, for any symmetric monoidal $\infty$-category $\EuScript{V}^{\otimes}$.
\end{cor}
\begin{proof} Observe that there is a commutative diagram of $\infty$-categories
	$$
	\begin{tikzcd}[ampersand replacement=\&]
		\Opdinfty_{/\Fin_*} \ar[rr, hookrightarrow, "\Env"]\ar[d, hookrightarrow] \&\& \sm\!\Catinfty_{/\Fin^{\amalg}} \ar[d, hookrightarrow]\\
		\Catinfty^{\!\!\!\!\!\!\mathsf{mrk}}_{/\Fin_*} \ar[rr, hookrightarrow, "\Env"]\ar[d] \&\& (\Catinfty^{\!\!\!\!\!\!\mathsf{act}\text{-}\mathsf{cocart},\,\mathsf{mrk}}_{/\Fin_*})_{/\Fin^{\amalg}} \ar[d]\\
		\Catinfty_{/\Fin_*} \ar[rr,"\Env"] \&\& \Catinfty^{\!\!\!\!\!\!\mathsf{act}\text{-}\mathsf{cocart}}_{/\Fin_*}
	\end{tikzcd},
	$$
	where the lower horizontal arrow refers to the functor in Equation (\ref{eqt:Generalized envelope}) (see also \cite[\textsection 2.1-2.3]{barkan_envelopes_2022}) and the superscript ``$\mathsf{mrk}$" means that the $\infty$-categories are equipped with a marking given by arrows over inert maps in $\Fin_*$ and functors between them preserve these markings. 
	
	The reason for the middle horizontal arrow to be fully-faithful is the same as for the top one (it admits a right adjoint such that the unit transformation is a natural equivalence; see \cite[Lemma 2.3.3]{haugseng_infty-operads_2024}). The decorated vertical maps are fully-faithful because 
 $$\EuScript{A}\mathsf{lg}_{\mathdutchcal{Q}}(\mathdutchcal{P})\subset \Fun_{\Fin_*}(\mathdutchcal{Q}^{\otimes},\mathdutchcal{P}^{\otimes})\quad (\text{resp.\ }\Fun^{\otimes}(\EuScript{C}^{\otimes},\EuScript{D}^{\otimes})\subset \Fun_{\Fin_*}(\EuScript{C}^{\otimes},\EuScript{D}^{\otimes}))
 $$
 is the full subcategory spanned by functors preserving cocartesian lifts of inert maps (resp.\ functors preserving all cocartesian maps).
	
	Since $\Env$ applied to $(\mathdutchcal{O}^{\otimes}\llbracket\mathdutchcal{W}_{\bullet}^{-1}\rrbracket\to \Fin_*) \in \Catinfty^{\!\!\!\!\!\!\mathsf{mrk}}_{/\Fin_*}$ yields an object in the essential image of $\Env\colon \Opdinfty_{/\Fin_*}\hookrightarrow \sm\!\Catinfty_{/\Fin^{\amalg}}$ (use Propositions \ref{prop:Localization of sm-envelope} and \ref{prop:Coincidence of envelopes}), we deduce that $\mathdutchcal{O}^{\otimes}\llbracket\mathdutchcal{W}_{\bullet}^{-1}\rrbracket\to \Fin_*$ actually belongs to $\Opdinfty$.
	
	It remains to check the $\infty$-localization claim. Combining Proposition \ref{prop:Coincidence of envelopes}, statement $(c)$ in Proposition \ref{prop:Localization of sm-envelope} and the fact that $\Env\colon \Opdinfty\rightleftarrows \sm\!\Catinfty:\!\Urect$ is an adjunction, one gets a commutative diagram of $\infty$-categories
	$$
	\begin{tikzcd}[ampersand replacement=\&]
		\EuScript{A}\mathsf{lg}_{\mathdutchcal{O}^{\otimes}\llbracket\mathdutchcal{W}_{\bullet}^{-1}\rrbracket}(\EuScript{V}^{\otimes}) \ar[rr]\ar[d, leftarrow, "\wr"'] \&\&  \EuScript{A}\mathsf{lg}_{\mathdutchcal{O},\mathdutchcal{W}}(\EuScript{V}^{\otimes})\ar[d, leftarrow, "\wr"]\\
		\Fun^{\otimes}(\Env(\mathdutchcal{O}^{\otimes}\llbracket\mathdutchcal{W}_{\bullet}^{-1}\rrbracket),\EuScript{V}^{\otimes}) \ar[rr, "\sim"] \&\& \Fun^{\otimes}(\Lfrak_{\mathdutchcal{W}}\Env(\mathdutchcal{O}^{\otimes}),\EuScript{V}^{\otimes})
	\end{tikzcd}
	$$
	for any symmetric monoidal $\infty$-category $\EuScript{V}^{\otimes}\to \Fin_*$. The equivalence of $\infty$-categories of algebras implies that the canonical map of quasioperads $(\Lfrak_{\mathdutchcal{W}}\mathdutchcal{O})^{\otimes}\to \mathdutchcal{O}^{\otimes}\llbracket\mathdutchcal{W}_{\bullet}^{-1}\rrbracket$ is a DK-equivalence; just note that this map is essentially surjective and apply \cite[Proposition 4.1.23]{harpaz_little_nodate}.
\end{proof}

\paragraph{Localization criteria for $\infty$-categories.} We have established that we can check if a map of quasioperads $\mathdutchcal{O}^{\otimes}\to \mathdutchcal{P}^{\otimes}$ exhibits an $\infty$-localization of quasioperads when it exhibits an $\infty$-localization of $\infty$-categories. Now, we provide the criterion we use in the main body of the text to detect $\infty$-localizations of $\infty$-categories, which is based on $\infty$-Rezk nerves as developed in \cite{mazel-gee_universality_2019} and which generalizes Hinich's Key Lemma in \cite{hinich_dwyer-kan_2016}.

First, recall that $\infty$-categories can be found among simplicial spaces $\Fun(\Updelta^{\op},\EuScript{S}\mathsf{pc})$ as complete-Segal spaces $\CSSpc$. In fact, the canonical cosimplicial object $[\bullet]$ in $\Catinfty$ induces an equivalence of $\infty$-categories $\Nrect_{\infty}\colon \Catinfty\to\CSSpc$.

Given a relative $\infty$-category $\mathdutchcal{C}=(\C,\W)$, the $\infty$\emph{-Rezk pre-nerve} is the simplicial $\infty$-category
\[
\mathsf{pre}\Nrect^{\Rrect}_{\infty}(\mathdutchcal{C})_{\bullet}=\Fun([\bullet],\C)^{\W}\equiv \mathsf{weq} ([\bullet], \mathdutchcal{C}).
\]
That is, the composite functor
\[
\Delta^{\op}
\xrightarrow{\;[\bullet]\;} \Catinfty^{\op}\xrightarrow{\;\mathsf{min}^{\op}\;} \EuScript{R}\mathsf{el}\!\Catinfty^{\op}\xrightarrow{\;\Fun(\star,\C)^{\W}\;} \Catinfty.
\]
The $\infty$\emph{-Rezk nerve} is given by postcomposing this construction with the groupoid completion, i.e.\ it is the following simplicial space/$\infty$-groupoid
\[
\Nrect^{\Rrect}_{\infty}(\mathdutchcal{C})_{\bullet}=\left(\Fun([\bullet],\C)^{\W}\right)^{\mathsf{gpd}}\colon \Delta^{\op}\xrightarrow{\;\;\mathsf{pre}\Nrect^{\Rrect}_{\infty}\;\;} \Catinfty\xrightarrow{\;\;(-)^{\mathsf{gpd}}\;\;}\EuScript{S}\mathsf{pc}.
\]

Denoting by $\Lrect_{\mathsf{CS}}\colon \Fun(\Delta^{\op},\EuScript{S}\mathsf{pc})\rightleftarrows \CSSpc\colon \mathsf{inc}$ the reflection of $\infty$-categories within simplicial spaces  given by complete-Segal-envelope, an important result of Mazel-Gee \cite[Theorem 3.8]{mazel-gee_universality_2019} states:
\begin{prop}\label{prop_RezkNerveComputesLocalizations}
	The map $\Nrect_{\infty}(\C)_{\bullet}\to \Lrect_{\mathsf{CS}}\Nrect^{\Rrect}_{\infty}(\C,\W)_{\bullet}$ exhibits the $\infty$-localization of $\C$ at $\W$ in $\CSSpc\simeq \Catinfty$ for any relative $\infty$-category $(\C,\W)$.
\end{prop}

Equipped with such a result, the criterion is immediate.
\begin{prop}\label{prop_HinichCriteria} Let $f\colon \mathdutchcal{C}\to\mathdutchcal{C}'$ be a relative functor in $\EuScript{R}\mathsf{el}\!\Catinfty$.
If the map of simplicial sets
	\[
	\begin{tikzcd}[ampersand replacement=\&]
		f_*\colon \mathsf{weq}\left[[k],\mathdutchcal{C}\right]\ar[r] \& \mathsf{weq}\left[[k],\mathdutchcal{C}'\right]
	\end{tikzcd}
	\]
	is a weak homotopy equivalence for any $k\geq 0$, $f$ induces an equivalence of $\infty$-localizations $\C\llbracket\W^{-1}\rrbracket\longrightarrow\C'\llbracket\W'^{-1}\rrbracket$ in $\Catinfty$.	In particular, the claim holds if for any $k\geq 0$ and any $\upsigma\in \mathsf{weq}\left[[k],\mathdutchcal{C}'\right]$, the slice $\infty$-category 
	\[
	\mathsf{weq}\left[[k],\mathdutchcal{C}\right]_{\upsigma/} \qquad (\text{resp.\; } \mathsf{weq}\left[[k],\mathdutchcal{C}\right]_{/\upsigma})
	\]
	is weakly contractible.
\end{prop}
\begin{proof} Applying the $\infty$-Rezk nerve, we obtain a commutative square of simplicial spaces
	\[
	\begin{tikzcd}[ampersand replacement=\&]
		\Nrect_{\infty}^{\Rrect}(\mathdutchcal{C})_{\bullet}\ar[d]\ar[rr,"\Nrect f_{\bullet}"] \& \& \Nrect_{\infty}^{\Rrect}(\mathdutchcal{C}')_{\bullet} \ar[d]\\
		\Lrect_{\mathsf{CS}}\Nrect_{\infty}^{\Rrect}(\mathdutchcal{C})_{\bullet}\ar[rr] \& \& \Lrect_{\mathsf{CS}}\Nrect_{\infty}^{\Rrect}(\mathdutchcal{C}')_{\bullet}
	\end{tikzcd},
	  \] 
	where the left vertical maps are instances of the universal $\mathsf{CS}$-equivalence. By Proposition \ref{prop_RezkNerveComputesLocalizations}, we need to show that the lower horizontal map is a Reedy-equivalence (or equivalently a $\mathsf{CS}$-equivalence). This would follow if we prove that $\Nrect f_{\bullet}$ is a Reedy-equivalence (since Reedy implies $\mathsf{CS}$).
	
	Unwinding the definitions, we should prove that for any $k\geq 0$,
	\[
	f_*\colon \big(\Fun([k], \C)^{\W}\big)^{\mathsf{gpd}}\longrightarrow \big(\Fun([k], \C')^{\W'}\big)^{\mathsf{gpd}}
	\]
	is a weak homotopy equivalence of spaces. Equivalently,  $\Fun([k],\C)^{\W}\to \Fun([k], \C')^{\W'}$ is a weak homotopy equivalence in $\Catinfty$.
	
	By Quillen's Theorem A for $\infty$-categories \cite[Theorem 4.1.3.1]{lurie_higher_2009}, this condition follows if for any object $\upsigma\in \Fun([k],\C')^{\W'}$, the slice $\infty$-category 
\[
\left(\Fun([k],\C)^{\W}\right)_{\upsigma/} \quad (\text{resp.\ } \left(\Fun([k],\C)^{\W}\right)_{/\upsigma})
\]
	is weakly contractible. Both slices are models for the homotopy fiber of $f_*$ at $\upsigma$ if the condition holds for any $\upsigma$. 
\end{proof}

\begin{rem}\label{rem:HinichCriteriaPLUS} In Proposition \ref{prop_HinichCriteria}, we could have taken the target relative $\infty$-category to be of the form $\D^{\sharp}=
(\D,\D^{\simeq})$. In that case, the conclusion translates into $f\colon \C\to \D$ exhibiting the $\infty$-localization of $\C$ at $\W$.
\end{rem}

\begin{cor}\label{cor_CriteriaForLocalizationFor0and1} In Proposition \ref{prop_RezkNerveComputesLocalizations}, it suffices to check the condition for $k=0$ and $k=1$ when $\Nrect_{\infty}^{\Rrect}(\mathdutchcal{C})_{\bullet}$ and $\Nrect_{\infty}^{\Rrect}(\mathdutchcal{C}')_{\bullet}$ are Segal spaces.
\end{cor}
\begin{proof} In the proof of the cited result, replace ``$\Nrect_{\infty}^{\Rrect}(\mathdutchcal{C})_{\bullet}\to \Nrect_{\infty}^{\Rrect}(\mathdutchcal{C}')_{\bullet}$ is a Reedy-equivalence'' by Dwyer-Kan equivalence between Segal spaces. 
\end{proof}

\begin{rem} In \cite{karlsson_assembly_2024}, it has also been observed that a map of quasioperads $\mathdutchcal{O}^{\otimes}\to \mathdutchcal{P}^{\otimes}$ exhibits an $\infty$-localization of $\infty$-operads if it does so as $\infty$-categories (see Lemma A.6 in loc.cit.). Their result, while useful in many situations, is weaker than Corollary \ref{cor:InfinityLocalizationOfOperatorCats}. For instance, our combined use of the cited Corollary with Proposition \ref{prop_HinichCriteria} to prove Lemma \ref{lem:DRvsClosedDR} (and hence our main Theorem \ref{thm_LocalizationForTruncatedDisj}) cannot be justified only with \cite[Lemma A.6]{karlsson_assembly_2024}.
\end{rem}

\section*{Declarations}

\subsection*{Competing interests}

The authors have no competing interests to declare that are relevant to the content of this article.

\subsection*{Data availability}

We do not analyse or generate any datasets, because our work proceeds within a theoretical and mathematical approach. One can obtain the relevant materials from the references below.

\bibliography{Bibliography}

\end{document}